\documentclass{amsart}
\usepackage[utf8]{inputenc}
\usepackage{amsmath}
\usepackage{amsfonts}
\usepackage{amssymb}
\usepackage{amsthm}
\usepackage{amscd}
\usepackage{float}
\usepackage{tikz}
\usepackage{graphicx}
\usepackage[colorlinks=true]{hyperref}
\hypersetup{urlcolor=blue, citecolor=red}
\usepackage{hyperref}

\newtheorem{theorem}{Theorem}[section]
\newtheorem{corollary}[theorem]{Corollary}

\newtheorem{lemma}[theorem]{Lemma}
\newtheorem{proposition}[theorem]{Proposition}

\theoremstyle{definition}
\newtheorem{definition}[theorem]{Definition}
\newtheorem{remark}[theorem]{Remark}

\newcommand{\R}{\mathbb{R}}
\newcommand{\N}{\mathbb{N}}
\newcommand{\C}{\mathbb{C}}
\newcommand{\T}{\mathbb{T}}
\newcommand{\Z}{\mathbb{Z}}
\newcommand{\D}{\mathbb{D}}

\newcommand{\Q}{\mathbb{Q}}

\begin{document}

\title[KP-II equations on the cylinder]{Global results for weakly dispersive KP-II equations on the cylinder}

\author[S.~Herr]{Sebastian Herr}
\address[S.~Herr]{Fakultat f\"ur Mathematik, Universit\"at Bielefeld, Postfach 10 01 31, 33501 Bielefeld, Germany}
\email{herr@math.uni-bielefeld.de}

\author[R.~Schippa]{Robert Schippa*}
\address[R.~Schippa]{UC Berkeley, Department of Mathematics, 847 Evans Hall
Berkeley, CA 94720-3840}
\email{rschippa@berkeley.edu}
\author[N.~Tzvetkov]{Nikolay Tzvetkov}
\address[N.~Tzvetkov]{Ecole Normale Sup\'erieure de Lyon, Unit\'e de Math\'ematiques Pures et Appliqu\'es,
UMR CNRS 5669, Lyon, France}
\email{nikolay.tzvetkov@ens-lyon.fr}

\thanks{*Corresponding author}

\makeatletter
\@namedef{subjclassname@2020}{%
  \textup{2020} Mathematics Subject Classification}
\makeatother

\begin{abstract}
We consider the dispersion-generalized KP-II equation on a partially periodic domain in the weakly dispersive regime. We use Fourier decoupling techniques to derive essentially sharp Strichartz estimates. With these at hand, we show global well-posedness of the quasilinear Cauchy problem in $L^2(\R \times \T)$. Finally, we prove a long-time decay property of solutions with small mass by using the Kato smoothing effect in the fractional case.
\end{abstract}


\maketitle

\section{Introduction}

We consider the following dispersion-generalized KP-II equation (fKP-II) posed on the cylinder $\R \times \T$, $\T = \R / (2 \pi \Z)$:
\begin{equation}
\label{eq:GeneralizedKPII}
\left\{ \begin{array}{cl}
\partial_t u - \partial_x D_x^\alpha u + \partial_x^{-1} \partial_y^2 u &= u \partial_x u, \quad (t,x,y) \in \R \times \R \times \T, \\
u(0) &= u_0 \in L^2(\R \times \T).
\end{array} \right.
\end{equation}
$D_x^\alpha$ and $\partial_x^{-1}$ denote Fourier multipliers, which are defined by
\begin{equation*}
\widehat{(D_x^\alpha f)} (\xi) = |\xi|^\alpha \hat{f}(\xi), \; \alpha \in \R, \quad \widehat{(\partial_x^{-1} f)}(\xi) = (i \xi)^{-1} \hat{f}(\xi).
\end{equation*}
See \cite[Section 5.2.1]{Klein-Saut} for a discussion of the derivation of this model.

We consider real-valued data, which is preserved by the nonlinear evolution. In the present paper we are concerned with global well-posedness in $L^2$-based Sobolev spaces, i.e., existence, uniqueness, and continuous dependence of the solution on the initial data and the global behavior of the solution. The Cauchy problem for the KP-II equation is given by \eqref{eq:GeneralizedKPII} for $\alpha = 2$ and is known to admit a smooth data-to-solution mapping in $L^2(\R \times \T)$ by previous work of the third author with Molinet and Saut \cite{MolinetSautTzvetkov2011}.

We address the weakly dispersive regime where $\alpha<2$, in which the Cauchy problem is \emph{quasilinear}, i.e., the data-to-solution mapping, if it exists in $L^2$-based Sobolev spaces, fails to be smooth (see \cite{MST_SIAM}).

\smallskip

Conserved quantities for real initial data are the mass
\begin{equation*}
M(u) = \int_{\R \times \T}  u^2 dx dy
\end{equation*}
and the energy
\begin{equation*}
E_\alpha(u) = \int_{\R \times \T} \big( \frac{1}{2} |D_x^{\frac{\alpha}{2}} u|^2 - \frac{1}{2} |\partial_x^{-1} \partial_y u|^2 + \frac{1}{6} u^3 \big) dx dy.
\end{equation*}
Since the energy is indefinite, it can hardly be utilized in a proof of global well-posedness. Here we use the mass, which requires a local well-posedness result in $L^2(\R \times \T)$.

\smallskip

On $\R^2$ we have the scaling symmetry
\begin{equation*}
u_\lambda(t,x,y) = \lambda^\alpha u(\lambda^{\alpha+1} t, \lambda x, \lambda^{\frac{\alpha}{2}+1} y).
\end{equation*}
Define the anisotropic homogeneous Sobolev norm as
\begin{equation*}
\| u_0 \|_{\dot{H}^{s_1,s_2}} = \| |\xi|^{s_1} |\eta|^{s_2} \hat{u}_0 \|_{L^2_{\xi,\eta}}.
\end{equation*}
For $s_2 = 0$, we find that $\dot{H}^{s_1,0}$ with $s_1 = 1 - \frac{3 \alpha}{4}$ is the scaling-critical space. Since $s_2 < 0$ breaks the Galilean invariance, $\dot{H}^{s_1,0}$ is distinguished.

\smallskip

The KP-II equation was proved to be globally well-posed in $L^2(\T^2)$ and $L^2(\R^2)$ by Bourgain \cite{Bourgain1993KPII}. Global well-posedness on the cylinder in $L^2(\R \times \T)$ was considered by Molinet--Saut--Tzvetkov \cite{MolinetSautTzvetkov2011}. We remark that the Kadomtsev--Petviashvili equations were derived to model water waves in two dimensions; see the works by Kadomtsev--Petviashvili \cite{Kadomtsev1970} and Ablowitz--Segur \cite{AblowitzSegur1979} with an emphasis on modeling transverse perturbations of the KdV soliton. Stability of the line soliton was proved in \cite{mizu}; see also \cite{RoussetTzvetkov2009,RoussetTzvetkov2011}. For this reason  the cylindrical domain $\R \times \T$ is of importance when studying \eqref{eq:GeneralizedKPII}. We remark that on the domain $\T \times \R$ the evolution satisfies improved dispersive properties compared to $\R \times \T$; see \cite{GruenrockPantheeDrumond2009}. On $\R^2$ Hadac--Herr--Koch \cite{HadacHerrKoch2009,HadacHerrKoch2010} proved scattering in the critical space $\dot{H}^{-\frac{1}{2},0}(\R^2)$ for small initial data. 

The Cauchy problem \eqref{eq:GeneralizedKPII} extends the fractional KdV equation on the real line:
\begin{equation}
\label{eq:GeneralizedKdV}
\left\{ \begin{array}{cl}
\partial_t u + \partial_x D_x^\alpha u &= u \partial_x u, \quad (t,x) \in \R \times \R, \\
u(0) &= u_0 \in L^2(\R).
\end{array} \right.
\end{equation}
For $\alpha = 2$ the original KdV equation is recovered, for $\alpha = 1$ the model describes the likewise prominent Benjamin--Ono equation. But the nonlinear evolution of these two models is very different in $L^2(\R)$. The KdV evolution in $L^2(\R)$ is known to be \emph{semilinear} with a real analytic data-to-solution mapping as a consequence of the contraction mapping principle. Seminal contributions are due to Bourgain \cite{Bourgain1993B} and Kenig--Ponce--Vega \cite{KenigPonceVega1996}. The $I$-team \cite{CollianderKeelStaffilaniTakaokaTao2003} finally proved sharp semilinear global well-posedness. More recently, Killip--Vi\c{s}an \cite{KillipVisan2019} devised the \emph{method of commuting flows}, which utilizes the complete integrability pointed out many decades ago by Lax \cite{Lax1968} in unweighted Sobolev spaces. They proved the sharp global well-posedness in $H^{-1}(\R)$.

On the other hand, for $\alpha < 2$, \eqref{eq:GeneralizedKdV} is a quasilinear Cauchy problem and it is known that it cannot be solved via Picard iteration for any Sobolev regularity \cite{MST_SIAM,KochTzvetkov2005,Herr2007}. 

Herr--Ionescu--Kenig--Koch \cite{HerrIonescuKenigKoch2010} proved global well-posedness of \eqref{eq:GeneralizedKdV} in $L^2(\R)$ for $\alpha \in (1,2)$ by using a paradifferential gauge transform. Recently, Ai--Liu \cite{AiLiu2024} reported an improvement to $s>\frac34(1-\alpha)$. The gauge transform for the Benjamin--Ono equation was applied for the first time by Tao \cite{Tao2004} who showed global well-posedness in $H^1(\R)$ in combination with Strichartz estimates. Combining the gauge transform with Fourier restriction analysis Ionescu--Kenig \cite{IonescuKenig2007} showed global well-posedness for the Benjamin--Ono equation in $L^2(\R)$ (see also \cite{BurqPlanchon2008,IfrimTataru2019,MolinetPilod2012}). Recently, Killip \emph{et al.} \cite{KillipLaurensVisan2024} proved sharp global well-posedness in $H^{s}(\R)$ for $s>-\frac{1}{2}$ via the method of commuting flows. Coming back to \eqref{eq:GeneralizedKPII}: Let 
\begin{equation*}
\omega_\alpha(\xi,\eta) = \xi |\xi|^\alpha - \frac{\eta^2}{\xi}
\end{equation*}
 denote the dispersion relation. 
 We note that for the KP-II equation the resonance function is given by
\begin{equation*}
\begin{split}
\Omega_2(\xi_1,\eta_1,\xi_2,\eta_2) &= \omega_2(\xi_1+\xi_2,\eta_1+\eta_2) - \omega_2(\xi_1,\eta_1) - \omega_2(\xi_2,\eta_2) \\
&= 3 (\xi_1+\xi_2) \xi_1 \xi_2 + \frac{(\eta_1 \xi_2 - \eta_2 \xi_1)^2}{\xi_1 \xi_2 (\xi_1 + \xi_2)}.
\end{split}
\end{equation*}
The fact that both terms are of the same sign highlights a nonlinear defocusing effect. This implies that the resonance function for the KP-II equation is at least as large as the resonance function for the KdV equation, for the first time effectively taken advantage of by Bourgain \cite{Bourgain1993KPII} who proved global well-posedness in $L^2(\T^2)$ and $L^2(\R^2)$.

Low regularity local well-posedness for fKP-II \eqref{eq:GeneralizedKPII} on $\R^2$ was proved by Hadac \cite{Hadac2008} for $\alpha > 4/3$, up to which the problem is semilinear and $L^2$-subcritical. Of course, on the cylinder $\R  \times \T$ one can apply the energy method to prove a local well-posedness result in isotropic Sobolev spaces $H^{s}(\R \times \T)$ for $s > 2$.
We remark that the argument from Linares--Pilod--Saut \cite{LinaresPilodSaut2018} would yield an improvement of the energy method on $\R \times \T$, but no global result. The main point of this paper is to prove the first global result for the Cauchy problem for fKP-II on $\R \times \T$. 

\smallskip

Recently, we \cite{HerrSchippaTzvetkov2024} extended Bourgain's $L^2$-result in the fully periodic case. We obtained local well-posedness in $H^{s,0}(\T^2)$ for $s>-\frac{1}{90}$. The proof interpolates short-time bilinear Strichartz estimates and novel $L^4$-Strichartz estimates derived from an $\ell^2$-decoupling inequality due to Guth--Maldague--Oh \cite{GuthMaldagueOh2024}.

 For \eqref{eq:GeneralizedKPII} we use Fourier restriction analysis on frequency-dependent time intervals and an interpolation argument in a similar way like in \cite{HerrSchippaTzvetkov2024}. This allows us to show well-posedness in $L^2(\R \times \T)$ for \eqref{eq:GeneralizedKPII} with $\alpha < 2$. The decoupling argument recovers the scaling critical Strichartz estimates up to the endpoint provided that the $\eta$-frequencies are not too large compared to the $\xi$-frequencies. Suppose that $\text{supp}(\hat{f}) \subseteq \{ (\xi,\eta) \in \R^2 : |\xi| \sim N, \; |\eta| \lesssim N^{\frac{\alpha}{2}+1} \}.$ Then we have the following estimate:
\begin{equation*}
\| e^{t (\partial_x D_x^{\alpha} - \partial_x^{-1} \partial_y^2)} f \|_{L^4_{t}([0,1],L^4_{xy}(\R \times \T))} \lesssim_\varepsilon N^{\frac{2-\alpha}{8}+ \varepsilon} \| f \|_{L^2(\R \times \T)}.
\end{equation*} 
In Section \ref{section:LinearStrichartzEstimates} we argue in greater generality how the decoupling estimate from \cite{GuthMaldagueOh2024} yields Strichartz estimates for dispersion relations with uniformly bounded derivatives, which are sharp up to endpoints. The argument does not depend on the domain.
 
  
  A crucial difference between KP-II on $\T^2$ and fKP-II on $\R \times \T$ is the occurence of very small frequencies $N \ll 1$.  For the fKP-II equations \eqref{eq:GeneralizedKPII} with $\alpha < 2$, there is clearly no way to show local well-posedness in $H^{s_1,s_2}(\R \times \T)$  via the contraction mapping principle as this would yield a real analytic well-posedness result for \eqref{eq:GeneralizedKdV} in $H^{s_1}(\R)$.
 We emphasize that the dispersion-generalized KP-II equations are not known to be completely integrable. Let $\alpha^\star:=2-0.015$. We show the following:
  
\begin{theorem}
\label{thm:GWPFKPII}
Let $\alpha \in (\alpha^\star, 2)$. Then \eqref{eq:GeneralizedKPII} is globally well-posed in $L^2(\R \times \T)$ for real-valued initial data.
\end{theorem}

Furthermore, in Proposition \ref{prop:ExistenceSolutions} we show existence of solutions and persistence of regularity in a regularity scale.
Since employing a gauge transform for \eqref{eq:GeneralizedKPII}  in higher dimensions appears to be very complicated, we opt for the different approach of frequency-dependent time localization. Early instances are due to Koch--Tzvetkov \cite{KochTzvetkov2003}, where Strichartz estimates on frequency-dependent time intervals are used to solve the Benjamin--Ono equation. Here we modify the implementation due to Ionescu--Kenig--Tataru \cite{IonescuKenigTataru2008}. In \cite{IonescuKenigTataru2008} an analysis on frequency-dependent time intervals was carried out using Fourier restriction norms to solve the KP-I equation globally in the energy space.

Here the frequency-dependent time localization will be chosen depending on $\alpha$. Moreover, we use a modulation weight to take advantage of large modulations due to the nonlinear defocusing effect. We explain our choice of time localization in Subsection \ref{subsection:ResonanceTimeLocalization}.

We remark that recently Guo--Molinet \cite{GuoMolinet2024} obtained unconditional global well-posedness for the KP-I equation in the energy space without using the short-time Fourier restriction spaces (see also \cite{Guo2024,KinoshitaSanwalSchippa2025}). Crucially, they show a priori bounds $L_T^p L^{\infty}_{xy}$ bounds of the solution which strongly hinges on dispersive effects. The argument to conclude well-posedness still relies on frequency-dependent time localization and a trilinear smoothing estimate related to the nonlinear Loomis--Whitney inequality (cf. \cite{IonescuKenigTataru2008}). The $L^2$ regularity we are dealing with in the present paper seems to be too low to use a similar iteration scheme.

Implementing the interpolation argument in short-time Fourier restriction spaces in the present analyis yields further losses. While the $L^4$-Strichartz estimates are sharp, the interpolation with multilinear estimates can likely be improved to lower the value of $\alpha^\star$ for $L^2$ global well-posedness. We make certain choices in the interpolation argument (see Section \ref{section:TrilinearConvolutionEstimates} and Remark \ref{remark:ConstantHighLowHigh}) to simplify the computations and do not strive for optimality here.

\smallskip

With short-time bilinear and energy estimates at hand, the conclusion of the proof of local well-posedness is well-known. We shall be brief in Section \ref{section:ConclusionLWP}. More details are given when constructing regular solutions. To this end, we use a Galerkin approximation and a compactness argument to construct solutions.

The starting point to show local well-posedness are a priori estimates for solutions $u$ to \eqref{eq:GeneralizedKPII} :
\begin{equation*}
\| u_i(t) \|_{H^{s,0}(\R \times \T)} \lesssim_{t, \| u(0) \|_{L^2}} \| u(0) \|_{H^{s,0}}.
\end{equation*}

A key point in the proof of the full local well-posedness result is that we estimate differences of solutions at negative Sobolev regularity $s' = -8(2-\alpha)$. The Lipschitz dependence we show for differences $v = u_1-u_2$, $u_i$, $i=1,2$ solutions to \eqref{eq:GeneralizedKPII}, reads
\begin{equation*}
\| v(t) \|_{H^{s',0}(\R \times \T)} \lesssim_{\| u_i(0) \|_{L^2}} \| v(0) \|_{H^{s',0}(\R \times \T)}
\end{equation*}
for $0 \leq t \leq T(\| u_i(0) \|_{L^2})$. The proof is then concluded by using frequency envelopes which is implemented like in our previous work \cite{HerrSchippaTzvetkov2024} and based on the general outline detailed in \cite{IfrimTataru2023}. This argument originates in work of Tao \cite{Tao2001}.

\medskip

Further, we address the long time behavior of small solutions. 
\begin{theorem}\label{long_time}
Let $\alpha \in (\alpha^*,2)$. There exists $\mu>0$ such that for every $u_0$ such that 
$$
\|u_0\|_{L^2(\R\times\T)}<\mu
$$
the solution of \eqref{eq:GeneralizedKPII} established in Theorem \ref{thm:GWPFKPII} satisfies for any $\gamma > 0$
$$
\lim_{t\rightarrow+\infty} \int_{\T}\, \int_{x\geq \gamma t}u^{2}(t,x,y)dxdy=0. 
$$
\end{theorem}

The $L^2$ conservation law  of \eqref{eq:GeneralizedKPII}  implies the $L^2$ stability of the zero solution of  \eqref{eq:GeneralizedKPII}.
However, this stability does not give any information on asymptotical stability, i.e. it does not give any information on how the conserved $L^2$ mass evolves in time. 
The result of Theorem \ref{long_time} implies that at time $t$ there is essentially no $L^2$ mass in the region $x\geq  \gamma t$. 

It would be very interesting to extend Theorem \ref{long_time} to the case when the zero solution is replaced by a 1d soliton. This would require an approach to the transverse stability of the KP-II line solitons independent of the Miura transform used in \cite{mizu}.  The approach of  \cite{mizu_ams} is less dependent on the integrability of the KP-II equation, and it would be worthwhile to investigate whether a similar approach may lead to the transverse stability of the line solitons of \eqref{eq:GeneralizedKPII}.

Moreover, the smallness assumption in Theorem  \ref{long_time} cannot be dropped. Indeed, arguing like in the proof of \cite[Proposition~1]{KMR} (see also \cite{LR}) one may show that for every $c>0$ there is a nontrivial $L^2$ solution of  \eqref{eq:GeneralizedKPII}  of the form 
$$
u(t,x,y)=Q(x-c t)
$$
and for such a solution 
$$
\int_{x\geq \gamma t}\int_{\T}Q^{2}(x-ct)dxdy=2\pi \int_{(\gamma-c) t}^\infty Q^2(x)dx
$$
which tends  to $2\pi \| Q\|_{L^2(\R)}^2\neq 0$ as $t\rightarrow\infty$, provided $c>\gamma$. This implies that $\mu<\sqrt{2\pi} \| Q\|_{L^2(\R)}$ is necessary.
\medskip

\emph{Outline of the paper.} In Section \ref{section:Notations} we introduce notation and the short-time function spaces. Moreover, we explain our choice of frequency-dependent time localization depending on resonance considerations and short-time Strichartz estimates. In Section \ref{section:LinearStrichartzEstimates} we discuss linear Strichartz estimates obtained from decoupling inequalities and in Section \ref{section:BilinearStrichartzEstimates} we show bilinear Strichartz estimates. In Section \ref{section:TrilinearConvolutionEstimates} we show trilinear convolution estimates, which will interpolate between linear and bilinear Strichartz estimates. In Section \ref{section:ShorttimeBilinearEstimates} we show short-time bilinear estimates and in Section \ref{section:ShorttimeEnergyEstimates} short-time energy estimates, for which we rely on linear and bilinear Strichartz estimates and trilinear convolution estimates. In Section \ref{section:ConclusionLWP} we construct regular solutions 
and conclude the proof of Theorem \ref{thm:GWPFKPII} with short-time bilinear and energy estimates at hand. Finally, in Section \ref{section:long} we prove Theorem \ref{long_time}.

\medskip

\textbf{Basic notation:}
\begin{itemize}
\item $\T = \R / (2 \pi \Z)$ denotes the one-dimensional torus with period $2 \pi$.
\item Dyadic numbers $N,L,\ldots \in 2^{\Z}$ are typically denoted by capital letters.
\item Time and space variables are given by $(t,x,y)$, the dual frequency variables are denoted by $(\tau,\xi,\eta)$.
\item $L^p(\R^{d_1} \times \T^{d_2})$ denotes the space of measurable functions $f: \R^{d_1} \times \T^{d_2} \to \C$ endowed with Lebesgue norms given by
\begin{equation*}
\| f \|^p_{L^p(\R^{d_1} \times \T^{d_2})} = \int_{\R^{d_1} \times \T^{d_2}} |f(x,y)|^p dx dy
\end{equation*}
and the usual modification for $p= \infty$.
\item $\mathcal{S}(\R \times \T)$ denotes the smooth complex-valued functions which, together with their derivatives, decay faster than any polynomial: For $\alpha,\beta, \gamma \in \N_0$ we have
\begin{equation*}
\| \langle x \rangle^{\alpha} \partial_x^{\beta} \partial_y^{\gamma} f \|_{L^\infty_{x,y}(\R \times \T)} < \infty.
\end{equation*}
\item For $f \in \mathcal{S}(\R \times \T)$ the inhomogeneous Sobolev norms are defined as
\begin{equation*}
    \|f\|_{H^{s_1,s_2}}:=\|\langle \xi \rangle^{s_1}\langle \eta\rangle^{s_2} \hat{f}(\xi,\eta) \|_{L^2_{\xi,\eta}}.
\end{equation*}
We define for $s_1,s_2 \in \R$ the space $H^{s_1,s_2}$ as closure of $\mathcal{S}(\R \times \T)$ with respect to $\| \cdot \|_{H^{s_1,s_2}}$.
\item For $T>0$ and a normed function space $X$ we use the short notation $C_T X = C([-T,T],X) $, similarly for $L_T^{\infty} X$.
\item We denote an estimate $A \leq C B$ by $A \lesssim B$ with $C$ being a harmless constant, which is allowed to change from line to line.
\item Further dependencies are indicated with subindices, e.g., $A \lesssim_{\varepsilon,p} B$ means that $A \leq C(\varepsilon,p) B$ for a constant $C$ depending on $\varepsilon$ and $p$.
\end{itemize}

\section{Notations and Function spaces}
\label{section:Notations}

In the following we introduce the short-time Fourier restriction spaces to carry out the analysis. The notation and conventions closely follow \cite[Section~2]{HerrSchippaTzvetkov2024}.

\subsection{Fourier transform and Littlewood-Paley decomposition}
The space-time Fourier transform of a function $u: \R \times \R \times \T \to \C$ is given by
\begin{equation*}
(\mathcal{F}_{t,x,y} u)(\tau,\xi,\eta) = \hat{u}(\tau,\xi,\eta) = \int_{\R \times \R \times \T} e^{-it \tau} e^{-ix\xi} e^{-iy\eta} u(t,x,y) dt dx dy.
\end{equation*}
Let $f: \R \times \R \times \Z \to \C$, $f \in L^1_{\tau,\xi} \ell^1_{\eta}$. The inverse Fourier transform is defined by
\begin{equation*}
(\mathcal{F}^{-1}_{t,x,y}f )(t,x,y) = \frac{1}{(2 \pi)^3} \int_{\R^2} e^{it \tau} e^{ix \xi} \sum_{\eta} e^{iy \eta} f(\tau,\xi,\eta) d \tau d\xi.
\end{equation*}

Let $\eta_0 \in C^\infty_c(B(0,2))$ with $\eta_0 \equiv 1$ on $[0,1]$ and $\eta_0$ be radially decreasing. We define
\begin{equation*}
\eta_1(\tau) = \eta_0(\tau/2) - \eta_0(\tau).
\end{equation*}
For $L \in 2^{\Z}$ we set $\eta_L(\tau) = \eta(\tau/L)$. Let $N \in 2^{\Z}$ be a dyadic number. We define frequency projections with respect to the $\xi$-frequencies by
\begin{equation*}
P_N : L^2(\R \times \R \times \T) \to L^2(\R \times \R \times \T), \quad u \mapsto  \mathcal{F}^{-1}_{t,x,y} (\eta_N(\xi) \mathcal{F}_{t,x,y} u).
\end{equation*}

\subsection{Function spaces}

Let $\alpha \in [1,2]$. Recall that the dispersion relation reads
\begin{equation*}
\omega_\alpha(\xi,\eta) = \xi |\xi|^\alpha - \frac{\eta^2}{\xi}.
\end{equation*}
Let $u_0 \in L^2(\R \times \T)$. We denote the linear propagation on $\R \times \T$ by
\begin{equation*}
S_\alpha(t) u_0 = \frac{1}{(2 \pi)^2} \int_{\R} \sum_{\eta \in \Z} e^{i(x \xi + y \eta + t \omega_\alpha(\xi,\eta))} \hat{u}_0(\xi,\eta) d\xi.
\end{equation*}
To address the singularity $\{ \xi = 0 \}$ of the symbol, we firstly explain the evolution on the dense subspace comprised of functions with frequency support away from $\{ \xi = 0 \}$. A similar frequency cutoff is carried out in Section \ref{subsection:ExistenceSolutions}. So we presently omit the details.

We define modulation projections for $L \in 2^{\N_0} \cup \{ 0 \}$ for functions $u \in L^2(\R \times \R \times \T)$ by
\begin{equation*}
\widehat{(Q_L u)} (\tau,\xi,\eta) = \eta_L(\tau - \omega_\alpha(\xi,\eta)) \hat{u}(\tau,\xi,\eta). 
\end{equation*}
We introduce a modulation weight already used by Bourgain \cite{Bourgain1993KPII}; see also \cite{HerrSchippaTzvetkov2024}. This will be crucial to show a suitable estimate in case the high modulation is on the low frequency.

Let $N_+ = \max(1,N)$. For $g: \R \times \R \times \Z \to \C$ with $g(\tau,\xi,\eta) = 0$ unless $|\xi| \in [N/4,4N]$ we define
\begin{equation*}
\begin{split}
\| g \|_{X_N} &= \sum_{L \geq 1} L^{\frac{1}{2}} (1+ L/ N_+^{\alpha+1})^{\frac{1}{4}} \| \eta_L(\tau-\omega_\alpha(\xi,\eta)) g \|_{L_{\tau,\xi,\eta}^2}, \\
\| g \|_{\bar{X}_N} &= \sum_{L \geq 1} L^{\frac{1}{2}} \| \eta_L(\tau-\omega_\alpha(\xi,\eta)) g \|_{L_{\tau,\xi,\eta}^2}.
\end{split}
\end{equation*}
Above we indicate the integration in frequency variables in the subindices. Note that the measure is the product of the Lebesgue measure in $\xi$ and $\tau$ and the counting measure in $\eta$.

We let
\begin{equation*}
\begin{split}
D_{N,L} &= \{ (\xi,\eta,\tau) : \; |\xi| \sim N, \, |\tau- \omega_\alpha(\xi,\eta)| \sim L \}, \\
D_{N,\leq L} &= \{ (\xi,\eta,\tau) : \; |\xi| \sim N, \, |\tau- \omega_\alpha(\xi,\eta)| \leq L \}
\end{split}
\end{equation*}
denote dyadically localized regions in space-time frequencies with fixed frequency and modulation. 

For $\beta > 0$ the short-time Fourier restriction norm $F_N$ at frequencies $N \in 2^{\Z}$ is defined by
\begin{equation*}
\| u \|_{F_N} = \sup_{t_k \in \R} \| \mathcal{F}_{t,x,y}[ \eta_0(N_+^{\beta}(t-t_k)) P_N u] \|_{X_N}.
\end{equation*}
The dual norm $\mathcal{N}_N$ prescribed by Duhamel's formula is given by
\begin{equation*}
\| f \|_{\mathcal{N}_N} = \sup_{t_k \in \R} \| (\tau - \omega_\alpha(\xi,\eta) + i N_+^{\beta})^{-1} \mathcal{F}_{t,x,y}[ \eta_0(N_+^{\beta}(t-t_k)) P_N u] \|_{X_N}.
\end{equation*}
As well $F_N$ as $\mathcal{N}_N$ are localized in time by the usual means: Let $T \in (0,1]$. For a frequency localized function $u: [-T,T] \times \R \times \T \to \C$ we let
\begin{equation*}
\| u \|_{G_N(T)} = \inf_{\tilde{u} \big\vert_{[0,T]} = u \big\vert_{[0,T]}} \| \tilde{u} \|_{G_N}, \quad G \in \{F,\mathcal{N} \},
\end{equation*}
where we take the infimum over all extensions $\tilde{u}: \R \times \R \times \T \to \C$. We explain in the next subsection why we choose $\beta = (2-\alpha) + \varepsilon$ as time localization for $\alpha \in [1,2]$.

\smallskip

The following lemma (see \cite[Eq.~(2.4)]{IonescuKenigTataru2008}) allows us to localize time to any shorter time intervals than required by the short-time norms provided that we do not use the weight:
\begin{lemma}
\label{lem:TimeLocalizationUnweighted}
For $M \in 2^{\N_0}$, and $f_N \in \bar{X}_N$, the following estimate holds:
\begin{equation*}
\begin{split}
&\sum_{J \geq M} J^{\frac{1}{2}} \| \eta_J(\tau - \omega(\xi,\eta)) \int_{\R} \big| f_N(\tau',\xi,\eta) \big| M^{-1} (1+M^{-1} |\tau - \tau'|)^{-4} d \tau' \|_{L^2_{\tau,\xi,\eta}} \\
&\quad + M^{\frac{1}{2}} \| \eta_{\leq M}(\tau-\omega(\xi,\eta)) \int_{\R} |f_N(\tau',\xi,\eta) | M^{-1} (1 + M^{-1} |\tau - \tau'|)^{-4} d\tau' \|_{L^2_{\tau,\xi,\eta}} \\
&\lesssim \| f_N \|_{\bar{X}_N}
\end{split}
\end{equation*}
with implicit constant independent of $K$ and $L$.
\end{lemma}
Since we need a variant for the summation with modulation weight, we record the proof.
\begin{proof}
For the small modulations $J \lesssim M$ we can use Young's convolution inequality to find
\begin{equation*}
\begin{split}
&\quad M^{\frac{1}{2}} \| \eta_{\leq M} (\tau - \omega_\alpha(\xi,\eta)) \int |f_N(\tau',\xi,\eta)| M^{-1} (1+M^{-1} |\tau-\tau'|)^{-4} d \tau' \|_{L^2_{\tau,\xi,\eta}} \\
&\lesssim M^{\frac{1}{2}} \| f_N \|_{L^2_{\xi,\eta} L^1_{\tau}} \| M^{-1} (1+ M^{-1} |\tau|)^{-4} \|_{L^2_{\tau}} \\
&\lesssim \sum_{L \geq 1} L^{\frac{1}{2}} \| \eta_L(\tau-\omega_{\alpha}(\xi,\eta)) f_N \|_{L^2_{\tau, \xi,\eta}}.
\end{split}
\end{equation*}
We turn to an estimate of the modulations greater than $M$:
Decompose
\begin{equation*}
\begin{split}
&\quad f_N(\tau',\xi,\eta)\\
&= ( \eta_{\lesssim M}(\tau'-\omega_\alpha) + \sum_{M \ll L \ll J} \eta_L(\tau'-\omega_\alpha) + \sum_{L \sim M} \eta_L(\tau'-\omega_\alpha) + \sum_{L \gg J} \eta_L(\tau'-\omega_\alpha)) \\
&\quad \times f_N(\tau',\xi,\eta).
\end{split}
\end{equation*}
For the estimate of the first term note that $M^{-1}(1+M^{-1}|\tau-\tau'|)^{-4} \lesssim M^{-1} (J/M)^{-4}$. Then, by the Cauchy-Schwarz inequality, we find
\begin{equation*}
\begin{split}
&\quad J^{\frac{1}{2}} M^{-1} (J/M)^{-4} \| \eta_J(\tau - \omega_\alpha) \int \eta_{\lesssim M}(\tau'-\omega_\alpha) |f_N(\tau',\xi,\eta)| d\tau' \|_{L^2_{\tau,\xi,\eta}} \\
 &\lesssim J^{\frac{1}{2}} M^{-\frac{1}{2}} (J/M)^{-4} \| \eta_{\lesssim M} f_N \|_{L^2_{\tau,\xi,\eta}}.
\end{split}
\end{equation*}
The summation over $J \gg M$ is straight-forward.

For the second term we find similarly
\begin{equation*}
\begin{split}
&\sum_{L: M \ll L \ll J} J^{\frac{1}{2}} (J/M)^{-2} \| \eta_J(\tau - \omega_\alpha) \int \eta_{L}(\tau'-\omega_\alpha) |f_N(\tau',\xi,\eta)| \\
&\quad \times M^{-1} (1+M^{-1}|\tau-\tau'|)^{-4} d\tau' \|_{L^2_{\tau,\xi,\eta}} \\
&\lesssim \sum_{L: M \ll L \ll J} J^{\frac{1}{2}} (J/M)^{-2} L^{\frac{1}{2}} L^{-\frac{1}{2}} \| \eta_L(\tau-\omega_\alpha) f_N \|_{L^2_{\tau,\xi,\eta}} \lesssim (J/M)^{-\frac{5}{2}} \| f_N \|_{\bar{X}_N},
\end{split}
\end{equation*}
again with straight-forward summation.

We turn to the third term. In this case we cannot argue by rapid decay of the kernel $k_M$, but the almost orthogonality salvages the estimate.
\begin{equation*}
\begin{split}
&\quad J^{\frac{1}{2}} \| \eta_J(\tau - \omega_\alpha) \int |f_N(\tau',\xi,\eta)| \eta_L(\tau'-\omega_\alpha) M^{-1} (1+M^{-1}|\tau-\tau'|)^{-4} d\tau' \|_{L^2_{\xi,\eta} L^2_{\tau}} \\
&\lesssim J^{\frac{1}{2}} \| \eta_L f_N * k_M \|_{L^2_{\tau,\xi,\eta}} \lesssim L^{\frac{1}{2}} \| \eta_L f_N \|_{L^2_{\tau,\xi,\eta}}.
\end{split}
\end{equation*}

Lastly, for the final term we obtain the estimate
\begin{equation*}
\begin{split}
&\quad J^{\frac{1}{2}} \| \eta_J(\tau- \omega_\alpha) \int |f_N(\tau',\xi,\eta)| \eta_L(\tau'-\omega_\alpha) M^{-1} (1+M^{-1}|\tau-\tau'|)^{-4} d\tau' \|_{L^2_{\tau,\xi,\eta}} \\
&\lesssim J^{\frac{1}{2}} (L/M)^{-2} L^{-\frac{1}{2}} L^{\frac{1}{2}} \| \eta_L f_N \|_{L^2_{\tau,\xi,\eta}}.
\end{split}
\end{equation*}
The summation is again straight-forward.
\end{proof}

The following is immediate:
\begin{corollary}
\label{cor:TimeLocalizationUnweighted}
For $\gamma \in \mathcal{S}(\R)$ and $M \in 2^{\N_0}$, and $t_0 \in \R$, it follows
\begin{equation*}
\| \mathcal{F}_{t,x,y}[\gamma(M(t-t_0)) \mathcal{F}^{-1}_{t,x,y}(f_N)] \|_{\bar{X}_N}  \lesssim \| f_N \|_{\bar{X}_N}.
\end{equation*}
The implicit constant is independent of $M$, $N$, and $t_0 \in \R$.
\end{corollary}

\begin{remark}
\label{rem:TimeLocalizationUnweighted}
To interpret the above, suppose we are estimating an expression $\gamma(T'^{-1} (t-t_0)) u_N$ which is localized at frequencies $N \geq 1$, $\alpha \in [1,2)$ and on times $T' \ll T(N)=N^{-(2-\alpha)-\varepsilon}$. When we do not use the weight $(1+L/N^{\alpha+1})^{\frac{1}{4}}$, we obtain the following estimate for the dyadically localized modulation pieces: 
\begin{equation*}
\sum_{L \geq (T')^{-1}} L^{\frac{1}{2}} \| f_{N,L} \|_{L^2_{\tau,\xi,\eta}} \lesssim \| u_N \|_{F_N},
\end{equation*}
where 
\begin{equation*}
f_{N,L} = \begin{cases} 1_{D_{N,\leq T'}} \mathcal{F}_{t,x,y} [\gamma(T'^{-1} (t-t_0)) u_N], &\quad L = T', \\
1_{D_{N,L}} \mathcal{F}_{t,x,y} [\gamma(T'^{-1} (t-t_0)) u_N], &\quad L > T'.
\end{cases}
\end{equation*}
\end{remark}

There are few cases in which we want to take advantage of the low modulation weight. In comparison with the above, we can only dispose the higher time localization if it is not higher than the size of the modulation weight. We have the following variant of Lemma \ref{lem:TimeLocalizationUnweighted}:
\begin{lemma}
\label{lem:TimeLocalizationWeighted}
Under the assumptions of Lemma \ref{lem:TimeLocalizationUnweighted}, for $M \lesssim N^{\alpha+1} \in 2^{\N_0}$, and $f_N \in X_N$, the following estimate holds:
\begin{equation*}
\begin{split}
&\sum_{J \geq M} J^{\frac{1}{2}} (1+J / N^{\alpha+1})^{\frac{1}{4}} \| \eta_J(\tau - \omega(\xi,\eta)) \\
&\quad \times \int_{\R} \big| f_N(\tau,\xi,\eta) \big| M^{-1} (1+M^{-1} |\tau - \tau'|)^{-4} d \tau' \|_{L^2_{\tau,\xi,\eta}} \\
&\quad + M^{\frac{1}{2}} \| \eta_{\leq M}(\tau-\omega(\xi,\eta)) \int_{\R} |f_N(\tau',\xi,\eta) | M^{-1} (1 + M^{-1} |\tau - \tau'|)^{-4} d\tau' \|_{L^2_{\tau,\xi,\eta}} \\
&\lesssim \| f_N \|_{X_N}
\end{split}
\end{equation*}
with implicit constant independent of $K$ and $L$.
\end{lemma}
\begin{proof}
This is a reprise of the previous proof. The reason the weight becomes admissible for $M \lesssim N^{\alpha+1}$ is that for the imbalanced $\tau,\tau'$ in the decomposition of $f_N(\tau',\xi,\eta)$ we gain factors $(J/M)^{-2}$. By these we can compensate for $J \gtrsim M$
\begin{equation*}
(1+J/N^{\alpha+1})^{\frac{1}{4}} (J/M)^{-\frac{1}{4}} \lesssim (1+M/N^{\alpha+1})^{\frac{1}{4}}.
\end{equation*}
The remainder of the argument is like in the proof of Lemma \ref{lem:TimeLocalizationUnweighted}.
\end{proof}

Correspondingly,
\begin{corollary}
\label{cor:TimeLocalizationWeighted}
For $\gamma \in \mathcal{S}(\R)$ and $M \lesssim N^{\alpha+1} \in 2^{\N_0}$, and $t_0 \in \R$, it follows
\begin{equation*}
\| \mathcal{F}_{t,x,y}[\gamma(M(t-t_0)) \mathcal{F}^{-1}_{t,x,y}(f_N)] \|_{X_N}  \lesssim \| f_N \|_{X_N}.
\end{equation*}
The implicit constant is independent of $M$, $N$, and $t_0 \in \R$.
\end{corollary}

\begin{remark}
\label{rem:TimeLocalizationWeighted}
Suppose again we aim to estimate an expression $\gamma(T'^{-1} (t-t_0) u_N )$ which is localized at frequencies $N \gg 1$ and on times $T' \ll T(N)=N^{-(2-\alpha)-\varepsilon}$. When we use the weight $(1+L/N^{\alpha+1})^{-1}$, we can estimate the expressions 
\begin{equation*}
\sum_{L \geq (T')^{-1}} L^{\frac{1}{2}} (1+L/N^{\alpha+1})^{\frac{1}{4}} \| f_{N,L} \|_{L^2_{\tau,\xi,\eta}} \lesssim \| u_N \|_{F_N},
\end{equation*}
provided that $T'\lesssim N^{\alpha+1}$.
\end{remark}

\smallskip

To justify manipulations e.g. when carrying out energy estimates, we prefer to work with regular functions. To introduce the suitable regularity scale, we let $p(\xi,\eta) = 1+|\eta|/|\xi|$ denote a symbol adapted to the (indefinite) energy. The introduction of $p(\xi,\eta)$ will allow us to obtain sufficient regularity in the $y$-variable by working with the norms
\begin{equation}
\label{eq:BsNorms}
\| \phi \|_{B^s} = \| \langle \xi \rangle^s p(\xi,\eta) \hat{\phi}(\xi,\eta) \|_{L^2_{\xi,\eta}(\R \times \Z)}, \quad s \geq 0
\end{equation}
for $\phi \in L^2(\R \times \T)$. Note that there are $\phi \in \mathcal{S}(\R \times \T)$ for which $\|\phi \|_{B^s} = \infty$, but for any $f \in B^s$ we can find $(\phi_n)_{n \in \N} \subseteq \mathcal{S}$ with $\| \phi_n \|_{B^s} < \infty$ such that 
\begin{equation*}
\| f - \phi_n \|_{B^s} \to 0.
\end{equation*}

For $s \in \R$, $u \in C([0,T],L^2(\R \times \T))$ we let
\begin{equation*}
\| u \|^2_{F^s(T)} = \sum_{N \in 2^{\Z}} N_+^{2s} \| P_N u \|^2_{F_N(T)}, \quad \| u \|^2_{\mathcal{N}(T)} = \sum_{N \in 2^{\Z}} N_+^{2s} \| P_N u \|^2_{\mathcal{N}_N(T)}.
\end{equation*}
The function spaces are comprised of regular functions in the following sense:
\begin{equation*}
\begin{split}
F^s(T) &= \{ u \in C([-T,T],B^2(\R \times \T)) : \| u \|_{F^s(T)} < \infty ) \}, \\
\mathcal{N}^s(T) &= \{ v \in C([-T,T],B^2(\R \times \T)) : \| v \|_{\mathcal{N}^s(T)} < \infty ) \}.
\end{split}
\end{equation*}

\begin{remark}
\label{rem:BoundednessNonlinearity}
The regularity $B^2$ is enough to bound the quadratic derivative nonlinearity in $L^2$ by Sobolev embedding. Consider a strong solution $u \in C([0,T],B^2)$ to \eqref{eq:GeneralizedKPII}, i.e., $u(t) = S_\alpha(t) u_0 + \int_0^t S_{\alpha}(t-s) \partial_x (u^2) ds$. The following is immediate from Sobolev embedding and the Leibniz rule:
\begin{equation*}
\| \int_0^t S_\alpha(t-s) \partial_x (u_1(s) u_2(s)) ds \|_{L^2_x} \lesssim_T \| u_1 \|_{C_T B^2} \| u_2 \|_{C_T B^2}.
\end{equation*}
\end{remark}

We recall the embedding $F^s(T) \hookrightarrow C([-T,T],H^{s,0}(\R \times \T))$ (cf. \cite[Lemma~3.1]{IonescuKenigTataru2008}):
\begin{lemma}
Let $s \in \R$, and $T \in (0,1]$. Then the following estimate holds:
\begin{equation*}
\sup_{t \in [0,T]} \| u(t) \|_{H^{s,0}(\R \times \T)} \lesssim \| u \|_{F^s(T)}.
\end{equation*}
\end{lemma}

Frequency-dependent time localization erases the dependence between initial data and the solution at finite time. We define a strengthened version of energy norms to take this into account:
\begin{equation*}
\| u \|^2_{E^s(T)} = \| P_{\leq 8} u(0) \|^2_{L^2} + \sum_{N \geq 8} N^{2s} \sup_{t \in [0,T]} \| P_N u(t) \|^2_{L^2}.
\end{equation*}
We define the energy space comprised of regular functions:
\begin{equation*}
\begin{split}
E^s(T) &= \{ u \in C([-T,T]:B^2(\R \times \T)) \, : \, \| u \|_{E^s(T)} < \infty \}.
\end{split}
\end{equation*}

The linear energy estimate in short-time Fourier restriction spaces (cf. \cite[Proposition~3.2]{IonescuKenigTataru2008}) reads as follows:
\begin{lemma}[Duhamel~estimate~in~short-time~function~spaces] 
\label{lem:LinearEnergyEstimate}
Let $s \in \R$ and $u \in E^s(T) \cap F^s(T)$, $v \in \mathcal{N}^s(T)$ such that $\partial_t u - \partial_x D_x^\alpha u + \partial_x^{-1} \partial_y^2 u = v$. Then the following estimate holds:
\begin{equation*}
\| u \|_{F^s(T)} \lesssim \| u \|_{E^s(T)} + \| v \|_{\mathcal{N}^s(T)}.
\end{equation*}
\end{lemma}

We shall trade modulation regularity for powers of the time localization. This will be helpful to control the evolution for large data. Define the following variant of $X_N$-spaces:
\begin{equation*}
\| f_N \|_{X_N^b} = \sum_{L \geq 1} L^b \| \eta_L(\tau - \omega_\alpha(\xi,\eta)) f_N \|_{L^2_{\tau,\xi,\eta}}.
\end{equation*}
Replacing $X_N$ with $X_N^b$ we obtain variants $F_N^b$, $\mathcal{N}_N^b$. We have the following lemma (cf. \cite[Lemma~3.4]{GuoOh2018}):
\begin{lemma}
\label{lem:TradingModulationRegularity}
Let $T \in (0,1]$, and $b<\frac{1}{2}$. Then it holds:
\begin{equation*}
\| P_N u \|_{F_N^b} \lesssim T^{\frac{1}{2}-b-} \| P_N u \|_{F_N}
\end{equation*}
for any function $u:[-T,T] \times \R \times \T \to \C$.
\end{lemma}

\subsection{Resonance considerations, frequency interactions, and short-time nonlinear and energy estimates}
\label{subsection:ResonanceTimeLocalization}

In this section we explain our choice of fre\-quency-dependent time localization by considering specific fre\-quency interactions. Resonance considerations play a crucial role. We give an overview of the short-time estimates for functions localized in modulation and frequency. These are presently formulated as convolution estimates in frequency-variables $(\tau,\xi,\eta)$\footnote{The corresponding measure is suppressed in notation.} for functions $f_i: \R \times \R \times \Z \rightarrow \R_{\geq 0}$ supported in $D_{N_i,L_i}$.
The key building blocks in the analysis are summable estimates for expressions of the form
\begin{equation}
\label{eq:GeneralBilinearEstimate}
\begin{split}
&\quad L^{-\frac{1}{2}} (1+L/N_+^{\alpha+1})^{\frac{1}{4}} \| 1_{D_{N,L}} (f_{1,N_1,L_1} * f_{2,N_2,L_2}) \|_{L^2_{\tau,\xi,\eta}} \\
&\lesssim C(N,N_1,N_2) \prod_{i=1}^2 L_i^{\frac{1}{2}} (1+L_i/N_{i,+}^{\alpha+1})^{\frac{1}{4}} \| f_{i,N_i,L_i} \|_2
\end{split}
\end{equation}
to carry out the short-time bilinear estimates in Section \ref{section:ShorttimeBilinearEstimates} and secondly, to obtain trilinear convolution estimates
\begin{equation}
\label{eq:GeneralTrilinearEstimate}
\int (f_{1,N_1,L_1} * f_{2,N_2,L_2}) f_{3,N_3,L_3}  \lesssim C(N,N_1,N_2) \prod_{i=1}^3 L_i^{\frac{1}{2}} (1+L_i/N_{i,+}^{\alpha+1})^{\frac{1}{4}} \| f_{i,N_i,L_i} \|_2
\end{equation}
which will play a role as well when proving the short-time bilinear estimates (note that after using duality, \eqref{eq:GeneralBilinearEstimate} essentially becomes \eqref{eq:GeneralTrilinearEstimate}) as when proving the short-time energy estimates. The factors of $L, L_i$ correspond to an estimate in Fourier restriction norms with modulation regularity $\frac{1}{2}$, which is the borderline case.  The minimum localization of $L,L_i$ will depend on the time localization. For summation of \eqref{eq:GeneralBilinearEstimate} and \eqref{eq:GeneralTrilinearEstimate} and establishing the short-time bilinear and short-time energy estimates, we moreover have to take into account the derivative nonlinearity and possibly additional time-localization. 

\smallskip

We turn to an estimate for the resonance function, which is for $\xi_i, \eta_i \in \R$, $\xi_i \neq 0$ given by
\begin{equation}
\label{eq:ResonanceFunction}
\begin{split}
\Omega_{\alpha}(\xi_1,\eta_1,\xi_2,\eta_2) &=
\omega_\alpha(\xi_1+\xi_2,\eta_1+\eta_2) - \omega_\alpha(\xi_1,\eta_1) - \omega_\alpha(\xi_2,\eta_2) \\
&=  \underbrace{(\xi_1+\xi_2)|\xi_1+\xi_2|^{\alpha} - \xi_1|\xi_1|^{\alpha} - \xi_2|\xi_2|^{\alpha}}_{\Omega_{\alpha,1}(\xi_1,\xi_2)} + \underbrace{\frac{(\eta_1 \xi_2 - \eta_2 \xi_1)^2}{\xi_1 \xi_2 (\xi_1+\xi_2)}}_{\Omega_{\alpha,2}}.
\end{split}
\end{equation}

$\Omega_{\alpha,1}$ and $\Omega_{\alpha,2}$ have the same sign, generalizing the nonlinear defocusing effect for KP-II, which has already been mentioned in the Introduction. Introduce notations
\begin{equation*}
X_{\max} = \max ( X, X_1 , X_2), \quad X_{\min} = \min(X, X_1, X_2), \quad X \in \{ N, L \},
\end{equation*}
and $L_{\text{med}} = \max( \{ L, L_1, L_2 \} \backslash L_{\max} )$. 

We summarize the key resonance estimate in the following lemma:
\begin{lemma}
Let $\alpha \in [1,2]$, and $\xi = \xi_1 + \xi_2$, $\eta = \eta_1 + \eta_2$ for real numbers, $\xi, \xi_i \neq 0$, and
\begin{equation*}
|\xi_1 + \xi_2| \sim N, \quad |\xi_1| \sim N_1, \quad |\xi_2| \sim N_2, \quad \{ N, N_1, N_2 \} \subseteq 2^{\Z}.
\end{equation*}
Then we have for $\Omega_\alpha$ defined in \eqref{eq:ResonanceFunction} that
\begin{equation*}
|\Omega_\alpha(\xi_1,\eta_1,\xi_2,\eta_2)| \geq |\Omega_{\alpha,1}(\xi_1,\xi_2)| \gtrsim N_{\max}^{\alpha} N_{\min}.
\end{equation*}
\end{lemma}

The localization present in \eqref{eq:GeneralBilinearEstimate}, \eqref{eq:GeneralTrilinearEstimate} is given by
\begin{equation*}
L \sim |\tau - \omega_{\alpha}(\xi,\eta)|, \; L_1 \sim |\tau_1 - \omega_{\alpha}(\xi_1,\eta_1)|, \; L_2 \sim |\tau_2 - \omega_{\alpha}(\xi_2,\eta_2)|
\end{equation*}
and by convolution constraint, we have
\begin{equation*}
\tau - \omega_{\alpha}(\xi_1+\xi_2,\eta_1+\eta_2) - (\tau_1 - \omega_{\alpha}(\xi_1,\eta_1)) - (\tau_2 - \omega_{\alpha}(\xi_2,\eta_2)) = - \Omega_{\alpha}(\xi_1,\eta_1,\xi_2,\eta_2),
\end{equation*}
which together with the localization implies
\begin{equation*}
L_{\max} \gtrsim N_{\max}^{\alpha} N_{\min}.
\end{equation*}

When estimating \eqref{eq:GeneralBilinearEstimate} and \eqref{eq:GeneralTrilinearEstimate}, we distinguish between 
\begin{itemize}
\item the \emph{resonant case}: $L_{\max} \sim N_{\max}^{\alpha} N_{\min} \gg L_{\text{med}}$, in which case the resonance function is as small as it can possibly be, 
\item the \emph{non-resonant case}: $L_{\max} \gg N_{\max}^{\alpha} N_{\min}$, $L_{\max} \gg L_{\text{med}}$,
\item
 and the \emph{strongly non-resonant case}: $L_{\max} \sim L_{\text{med}}$. 
\end{itemize} 
 The strongly non-resonant case is most favorable and can typically be estimated by the simple Strichartz estimate from Lemma \ref{lem:SecondOrderTransversality}. So, we shall focus on the resonant and non-resonant case. Moreover, we distinguish between \textbf{High}$\times$\textbf{Low}$\rightarrow$\textbf{High}-interaction, \textbf{High}$\times$\textbf{High}$\rightarrow$\textbf{High}-interaction, and \textbf{High}$\times$\textbf{High}$\rightarrow$\textbf{Low}-interaction depending on the relative size of the involved $\xi$-frequencies, when showing estimates for \eqref{eq:GeneralBilinearEstimate}, \eqref{eq:GeneralTrilinearEstimate}.

\smallskip

We explain our choice of time-localization looking into the  resonant case in the \emph{High$\times$Low$\rightarrow$High-}interaction. The derivative loss must be eliminated to propagate the regularity. After applying the space-time Fourier transform and dyadic decomposition in Fourier space, estimates in Fourier restriction norms
\begin{equation*}
\| P_N \partial_x (P_{N_1} u P_{N_2} u) \|_{\mathcal{N}_N} \lesssim \| P_{N_1} u \|_{F_{N_1}} \| P_{N_2} u \|_{F_{N_2}}
\end{equation*}
with $N \sim N_1 \gg N_2$ follow from
\begin{equation*}
\begin{split}
&\quad L^{-\frac{1}{2}} (1+L/N_+^{\alpha+1})^{\frac{1}{4}} N \| 1_{D_{N,L}} (f_{1,N_1,L_1} * f_{2,N_2,L_2}) \|_{L^2_{\tau,\xi,\eta}} \\
&\lesssim \prod_{i=1}^2 L_i^{\frac{1}{2}} (1+L_i/N_{i,+}^{\alpha+1})^{\frac{1}{4}} \| f_{i,N_i,L_i} \|_{L^2_{\tau,\xi,\eta}}.
\end{split}
\end{equation*}
The factor $N$ in the first line reflects the derivative nonlinearity, and we require moreover that $\text{supp}(f_{i,N_i,L_i}) \subseteq D_{N_i,L_i}$. By the resonance considerations, we can suppose that $\max(L_i,L) \gtrsim N_1^\alpha N_2$. Suppose that $L \sim N_1^\alpha N_2$. By applying H\"older's inequality and the bilinear Strichartz estimate to be proved in Proposition \ref{prop:BilinearStrichartzA}, we obtain
\begin{equation*}
\begin{split}
&\quad L^{-\frac{1}{2}} \| 1_{D_{N,L}} (f_{1,N_1,L_1} * f_{2,N_2,L_2}) \|_{L^2_{\tau,\xi,\eta}} \\
 &\lesssim (N_1^\alpha N_2)^{-\frac{1}{2}} \| f_{1,N_1,L_1} * f_{2,N_2,L_2} \|_{L^2_{\tau,\xi,\eta}} \\
&\lesssim \log(N_1) (N_1^\alpha N_2)^{-\frac{1}{2}} N_2^{\frac{1}{2}} L_{12,\min}^{\frac{1}{2}} \langle L_{12,\max} / N_1^{\frac{\alpha}{2}} \rangle^{\frac{1}{2}} \prod_{i=1}^2 \| f_{i,N_i,L_i} \|_{L^2}.
\end{split}
\end{equation*}
This computation shows that in case of small modulation $L_1 \sim L_2 \sim 1$, the derivative gain recovered from the resonance does not suffice to close the iteration in Fourier restriction norms. 

Now we take into account a frequency-dependent time localization $T=T(N)=N^{-\beta}$, $0 < \beta < \alpha$. This implies that $L_1 \gtrsim N_1^{\beta}$. Then we can continue to estimate
\begin{equation*}
\begin{split}
&\quad (N_1^\alpha N_2)^{-\frac{1}{2}} N_2^{\frac{1}{2}} L_{12,\min}^{\frac{1}{2}} \langle L_{12,\max} / N_1^{\frac{\alpha}{2}} \rangle^{\frac{1}{2}} \prod_{i=1}^2 \| f_{i,N_i,L_i} \|_{L^2} \\
&\lesssim N_1^{-\frac{\alpha}{2}} N_1^{-\frac{\beta}{2}} \prod_{i=1}^2 L_i^{\frac{1}{2}} \| f_{i,N_i,L_i} \|_{L^2}.
\end{split}
\end{equation*}
A frequency-dependent time localization of $\beta = (2-\alpha)+\varepsilon$ will ameliorate the derivative loss and allows us to show short-time nonlinear estimates in $L^2$. The short-time nonlinear estimates are proved in Section \ref{section:ShorttimeBilinearEstimates}. The crucial (short-time) bilinear Strichartz estimates are proved in Section \ref{section:BilinearStrichartzEstimates}. The proof of the trilinear convolution estimates \eqref{eq:GeneralTrilinearEstimate} is carried out in Section \ref{section:TrilinearConvolutionEstimates}. These naturally appear when considering energy estimates of the $L^2$-norm of frequency localized solutions. We end the section with a brief explanation of the short-time energy estimates.

Indeed, for a sufficiently regular solution to \eqref{eq:GeneralizedKPII} we have by the fundamental theorem:
\begin{equation*}
\| P_N u(t) \|_{L^2}^2 = \| P_N u(0) \|^2_{L^2} + 2 \int_0^t \int_{\R \times \T} P_N u(s,x,y) \partial_x P_N (u \cdot u)(s,x,y) dx dy ds.
\end{equation*}
The different frequency interactions are recovered from a paraproduct decomposition:
\begin{equation*}
P_N (u \cdot u) = 2 P_N (u P_{\ll N} u) + P_N (P_{\gtrsim N} u \cdot P_{\gtrsim N} u).
\end{equation*}
Note that in the second expression the frequencies must be of comparable modulus by Littlewood-Paley dichotomy and the derivative is already acting on the lowest frequency.

By a standard commutator argument, using the real-valuedness of the solution, this can be accomplished as well for the first term. After expanding $P_{\ll N} u = \sum_{K \ll N } P_K u$ we morally have to estimate expressions
\begin{equation*}
\int_0^t \int_{\R \times \T} P_N u P_N u (\partial_x P_K u) dx dy ds
\end{equation*}
for $K \ll N$. To estimate this in terms of the short-time Fourier restriction norm, we have to subdivide $[0,t]$ into intervals of length $T=T(N)=N^{-\beta}$. This incurs a loss of size $\sim |t| N^{\beta}$. This indicates that for the proof of favorable energy estimates we should choose the frequency-dependent time localization as small as possible and distinguishes the value from above $T=T(N)=N^{-(2-\alpha+\varepsilon)}$.

Going further, after (smoothly) localizing time to an interval of length $N^{-\beta}$, we find expressions of the form
\begin{equation*}
\int_{\R \times \R \times \T} (\gamma(N^{\beta}(t-t_k)) P_N u) (\gamma(N^{\beta}(t-t_k)) P_N u) (\partial_x P_K u).
\end{equation*}
Note that we have some leeway in the time localization of $P_K u$. We can again choose $T=T(N)=N^{-\beta}$, but can also enlarge the time localization. Then, after applying the space-time Fourier transform, we see that it suffices to estimate an expression
\begin{equation*}
N_2 \sum_{L_i,L} \int (f_{1,N_1,L_1} * f_{2,N_2,L_2}) f_{3,N_3,L_3}.
\end{equation*}
with $N_1 \sim N_3 \gg N_2 \sim K$. Compared to the \emph{High$\times$Low$\rightarrow$High}-interaction encountered in the short-time nonlinear interaction, we had to add the time localization, which incurs the loss $N_1^{\beta}$, but for $N_2 \ll N_1^{1-\beta}$ the commutator argument produces a large gain. 

We remark on comparable frequencies, $N_1 \sim N_2 \sim N_3$, for which the commutator argument does not produce a gain anymore. In this case we can prove more favorable estimates relying on novel $L^4$-Strichartz estimates; see Section \ref{section:LinearStrichartzEstimates}. The interpolation of $L^4$-Strichartz estimates with short-time bilinear estimates is referred to as nonlinear interpolation argument and outlined in Section \ref{subsection:OutlineNonlinearInterpolation}.

Lastly, we note that for differences of solutions, we have less symmetries at disposal and the integration by parts and commutator arguments can no longer assign the derivative to the lowest frequency. Here the estimate at negative regularities comes to rescue. We show the short-time energy estimates in Section \ref{section:ShorttimeEnergyEstimates}.

\section{Linear Strichartz estimates via decoupling}
\label{section:LinearStrichartzEstimates}

In this section we show sharp Strichartz estimates via $\ell^2$-decoupling. To explain the key ingredients, we start with more general considerations, relating Knapp examples with the notion of flat sets.

\subsection{Decoupling into flat sets and the Knapp example}

\label{subsection:DecouplingIntro}

 Decoupling estimates origin in the work of Wolff \cite{Wolff2000} on local smoothing for the Euclidean wave equation; see also \cite{LabaWolff2002,GarrigosSeeger2009}. Finally, Bourgain--Demeter \cite{BourgainDemeter2015} obtained sharp $\ell^2$-decoupling estimates for elliptic surfaces and their conical extensions. Let
\begin{equation*}
\mathcal{E}_{\Delta} f(x,t) = \int_{|\xi| \leq 1} e^{i(\langle x, \xi \rangle + t |\xi|^2)} f(\xi) d\xi
\end{equation*}
denote the Fourier extension operator for the paraboloid. The $\ell^2$-decoupling estimates read
\begin{equation*}
\| \mathcal{E}_{\Delta} f \|_{L^p_{t,x}(B_{d+1}(0,R))} \lesssim_\varepsilon R^\varepsilon \big( \sum_{\theta:R^{-\frac{1}{2}}-\text{ball}} \| \mathcal{E}_{\Delta} f_{\theta} \|^2_{L^p_{t,x}(w_{B_{d+1}(0,R)})} \big)^{\frac{1}{2}}
\end{equation*}
for $2 \leq p \leq \frac{2(d+2)}{d}$. Above $w_{B_{d+1}}(0,R)$ denotes a weight with high polynomial decay away from $B_{d+1}(0,R)$. The $R^{-\frac{1}{2}}$-balls are distinguished for elliptic surfaces as on scales $|(x,t)| \leq R$ these trivialize the Fourier extension. This is a simple consequence of Taylor expanding the phase function. By approximating exponential sums with the Fourier extension operator, Bourgain--Demeter \cite{BourgainDemeter2015} proved sharp (up to endpoints) Strichartz estimates for elliptic Schr\"odinger equations on rational and irrational tori.


Whereas the Strichartz estimates for elliptic Schr\"odinger equations on tori for finite times resemble the estimates on Euclidean space up to arbitrarily small additional derivative loss, this drastically fails for hyperbolic Schr\"odinger equations (see also \cite{BourgainDemeter2017}). We shall see below that the fKP-II surfaces are non-elliptic.

\smallskip

Let $\Box = \partial_1^2 - \partial_2^2$ denote the hyperbolic Laplacian in two dimensions. Then, due to the dispersive estimate, on Euclidean space it holds
\begin{equation*}
\| e^{it \Box} u_0 \|_{L_t^4(\R;L^4_{xy}(\R^2))} \lesssim \| u_0 \|_{L^2_{xy}(\R^2)}.
\end{equation*}
On the square torus there is no oscillation on the diagonal $\{ \xi = \eta \}$, which points out that the following $L^4$-Strichartz estimate is sharp:
\begin{equation}
\label{eq:L4StrichartzT2}
\| P_N e^{it \Box} u_0 \|_{L_t^4([0,1],L^4_{xy}(\T^2))} \lesssim N^{\frac{1}{4}} \| u_0 \|_{L^2(\T^2)}.
\end{equation}
This has been shown in \cite{Godet}, see also \cite{Wang} and \cite{SchippaDecoupling2020}.
But on irrational tori $\T_{\gamma}^2 = \T \times \gamma \T$, $\gamma \in (1/2,1] \backslash \Q$, this example is clearly excluded, which leads to the natural question of a possible improvement of \eqref{eq:L4StrichartzT2} on $\T^2_\gamma$. This question was recently answered affirmatively by Guth--Maldague--Oh \cite{GuthMaldagueOh2024}, who showed an $\ell^2$-decoupling inequality into ``flat sets". To describe their result, let $\varphi: \R^2 \to \R$ denote a smooth function and $S=\{ ( \xi,\eta,\varphi(\xi,\eta)) : |(\xi,\eta)| \leq 1 \}$ be the graph surface. Let
\begin{equation*}
\mathcal{E}_S f(x,y,t) = \int_{|(\xi,\eta)| \leq 1} e^{i(x \xi + y \eta + t \varphi(\xi,\eta))} f(\xi,\eta) d\xi d\eta
\end{equation*}
denote the corresponding Fourier extension operator.

We make the following definition:
\begin{definition}
Let $\phi : \R^2 \to \R$ be smooth and $\delta > 0$. We say that $S \subseteq [0,1]^2$ is $(\phi,\delta)$-flat if
\begin{equation*}
\sup_{u,v \in S} | \phi(u) - \phi(v) - \nabla \phi(v) (u-v)| \leq \delta.
\end{equation*}
\end{definition}
Clearly, on flat sets there are no oscillations in $\mathcal{E}_S$.
For this reason, an estimate of $\mathcal{E}_S \hat{f}$ on the scale $\delta^{-1}$ for a function $f$ with Fourier support contained in a $\delta$-flat set amounts to
\begin{equation*}
\| \mathcal{E}_S \hat{f} \|_{L_t^p([0,\delta^{-1}],L^q_{xy}(\R^2))} \sim \delta^{-\frac{1}{p}} \| f \|_{L^q_{xy}(\R^2)}.
\end{equation*}
Indeed, we can write for fixed $(\xi_*,\eta_*) \in S$,
\begin{equation*}
\phi(\xi,\eta) = \phi(\xi_*,\eta_*) + (\xi-\xi_*,\eta-\eta_*) \nabla \phi(\xi_*,\eta_*) + \psi(\xi,\eta,\xi_*,\eta_*).
\end{equation*}
By a change of variables in $x,y$ the linear terms in $\xi,\eta$ can be eliminated and by the $\delta$-flat property, $\psi$ does not cause oscillations.

\smallskip

Now, choosing a function which exhausts Bernstein's inequality, we find
\begin{equation*}
\| f \|_{L^q_{xy}(\R^2)} \sim \text{meas}_{\R^2}(S)^{\frac{1}{2}- \frac{1}{q}} \| f \|_{L^2(\R^2)}.
\end{equation*}
In conclusion, we have the estimate
\begin{equation*}
\| \mathcal{E}_S \hat{f} \|_{L_t^p([0,\delta^{-1}],L^q_{xy}(\R^2))} \sim \delta^{-\frac{1}{p}} \text{meas}_{\R^2}(S)^{\frac{1}{2}- \frac{1}{q}} \| f \|_{L^2(\R^2)}. 
\end{equation*}

When carrying out the argument on a different domain $\D_{\lambda} \in \{ \R \times \lambda \T, \; \lambda \T_{\gamma}^2 \}$\footnote{The scaling parameter $\lambda$ will be linked to $\delta$.}, we obtain the estimate
\begin{equation*}
\| \mathcal{E}_S \hat{f} \|_{L_t^p([0,\delta^{-1}],L^q_{xy}(\D_{\lambda}))} \sim \delta^{-\frac{1}{p}} \text{meas}_{\D_{\lambda}^*}(S)^{\frac{1}{2}- \frac{1}{q}} \| f \|_{L^2(\D_{\lambda})}
\end{equation*}
where $\D_{\lambda}^*$ denotes the Pontryagin dual:
\begin{equation*}
\D_{\lambda}^* = 
\begin{cases}
\R \times \Z/\lambda, &\quad \D_\lambda = \R \times \lambda \T, \\
(\Z \times \Z_\gamma)/\lambda, &\quad \D_\lambda = \lambda \T^2_\gamma.
\end{cases}
\end{equation*}

Suppose we have an $\ell^2$-decoupling estimate into essentially disjoint $\delta$-flat sets:
\begin{equation*}
\| \mathcal{E}_S \hat{f} \|_{L_{t,x,y}^{p}(B_{\delta^{-1}})} \lesssim C(\delta) \big( \sum_{\theta: (\phi,\delta)-\text{flat}} \| \mathcal{E}_S \hat{f}_{\theta} \|_{L_{t,x,y}^p(w_{B_{\delta^{-1}}})}^2 \big)^{\frac{1}{2}}.
\end{equation*}
With the above estimate at hand, we obtain Strichartz estimates:
\begin{equation*}
\| \mathcal{E}_S \hat{f} \|_{L_{t,x,y}^{p}(B_{\delta^{-1}})} \lesssim C(\delta) \delta^{-\frac{1}{p}} M(S)^{\frac{1}{2}- \frac{1}{p}} O(S) \| f \|_{L^2},
\end{equation*}
$M(S)$ denotes the maximum of the $\D^*$-measure of $S$, and $O(S)$ denotes the maximum overlap of the $\theta$-sets. Note that the decoupling estimates strictly speaking apply only to the oscillatory integral in case $\D = \R^2$. For other domains, this can be recovered by approximating the exponential sum with an oscillatory integral; see \cite{BourgainDemeter2015} and for more details \cite[Section~2.2]{Schippa2024Refinements}.

\medskip

Recently, Guth--Maldague--Oh \cite{GuthMaldagueOh2024} proved the following $\ell^2$-decoupling result for smooth hypersurfaces in $\R^3$. We formulate the result for functions $f \in \mathcal{S}(\R^3)$ with $\text{supp} (\hat{f}) \subseteq \mathcal{N}_{\delta}(\mathcal{M}_\phi)$. Then, for $S \subseteq \R^2$ a rectangle, $f_S$ denotes a smoothed Fourier projection to $S \times \R$. We have the following result: 
\begin{theorem}[{\cite[Theorem~1.2]{GuthMaldagueOh2024}}]
\label{thm:GeneralL2Decoupling}
Let $\phi: \R^2 \to \R$ be a smooth function. Fix $\varepsilon > 0$. Then there exists a suffficiently large number $A$ depending on $\varepsilon$ and $\phi$ satisfying the following.

For any $\delta > 0$ there exists a collection $\mathcal{S}_\delta$ of finitely overlapping parallelograms $S$ such that
\begin{enumerate}
\item the overlapping number is $\mathcal{O}(\log \delta^{-1})$ in the sense that
\begin{equation*}
\sum_{S \in \mathcal{S}_\delta} \chi_S \leq C_{\varepsilon,\phi} \log(\delta^{-1}),
\end{equation*}
\item $S$ is $(\phi,A\delta)$-flat,
\item For $2 \leq p \leq 4$ we have
\begin{equation*}
\| f \|_{L^p} \leq C_{\varepsilon, \phi} \delta^{-\varepsilon} \big( \sum_{S \in \mathcal{S}_\delta} \| f_S \|_{L^p}^2 \big)^{\frac{1}{2}}
\end{equation*}
for any function $f$ with $\text{supp}(\hat{f}) \subseteq \mathcal{N}_\delta(\mathcal{M}_\phi)$.
\end{enumerate}
\end{theorem}

To illustrate this approach to Strichartz estimates, we first turn to the simpler case of the hyperbolic Schr\"odinger equation on $\R \times \T$. Let $\varphi(\xi_1,\xi_2) = \xi_1^2 - \xi_2^2$.

\medskip

We have the following Strichartz estimate, which is sharp up to endpoints:
\begin{theorem}
\label{thm:StrichartzHypSEQ}
The following estimate holds for $s>0$:
\begin{equation}
\label{eq:StrichartzHyperbolicSEQ}
\| e^{it (\partial_1^2 -\partial_2^2)} f \|_{L^4_{t,x,y}([0,1];\R \times \T)} \lesssim \| f \|_{H^s(\R \times \T)}.
\end{equation}
\end{theorem}
The estimate can be proved using elementary counting arguments reminiscent of the arguments in \cite{TakaokaTzvetkov2001}; see \cite{BSTW}. Here we give a different proof by decoupling into flat sets.

\begin{proof}[Proof~of~Theorem~\ref{thm:StrichartzHypSEQ}~via~decoupling~into~flat~sets]
For frequencies of size $N$, after rescaling, we apply parabolic rescaling $(\xi,\eta) \to (\xi,\eta)/N$, $(x,y) \to N(x,y)$, and $t \to N^2 t$. We obtain
\begin{equation*}
\begin{split}
&\quad \| P_N e^{it (\partial_1^2 - \partial_2^2)} f \|^4_{L^4_{t,x,y}([0,1] \times \T^2)} \\
 &= N^{-6} \| \int_{|\xi| \leq 1} \sum_{\eta \in \Z / N} e^{i(x \xi + y \eta + t (\xi^2 - \eta^2))} \hat{f}(N \xi, N \eta) d\xi \|_{L^4_{t,x,y}([0,N^2] \times N^2 \T^2)}^4.
\end{split}
\end{equation*}
Now we change to continuous approximation to be in the position to apply the decoupling result:
\begin{equation*}
\big\| \int_{|\xi| \leq 1} \sum_{\eta \in \Z / N} e^{i(x \xi + y \eta + t(\xi^2 - \eta^2))} \hat{f}(N \xi, N\eta) \big\|_{L^4_{t,x,y}(B_{2+1}(0,N^2))} \lesssim \big\| \mathcal{E}_{\Box} \tilde{f} \big\|_{L^4_{t,x,y}(w_{2+1}(0,N^2))} 
\end{equation*}
with
\begin{equation*}
\mathcal{E}_{\Box} \tilde{f} = \int_{|(\xi,\eta)| \leq 1} e^{i(x \xi + y \eta + t (\xi^2 - \eta^2))} \tilde{f}(\xi,\eta) d\xi d\eta.
\end{equation*}

Let $\delta = N^{-2}$. Invoking Theorem \ref{thm:GeneralL2Decoupling} (see also \cite[Theorem~2.2]{GuthMaldagueOh2024} for the construction of the set $\Theta_{\delta}$) we find
\begin{equation*}
\| \mathcal{E}_{\Box} \tilde{f} \|_{L^4_{t,x,y}(w_{B_{2+1}}(0,\delta^{-1}))} \lesssim_{\varepsilon} \delta^{-\varepsilon} \big( \sum_{\theta \in \Theta_{\delta}} \| \mathcal{E}_{\Box} \tilde{f}_{\theta} \|_{L^4_{t,x,y}(w_{B_{2+1}(0,\delta^{-1})})}^2 \big)^{\frac{1}{2}}
\end{equation*}
with $\theta$ being $\delta$-flat rectangles of dimensions $\ell_1 \times \ell_2$ in the null directions of the hyperboloid. The null directions are given by
\begin{equation*}
\mathbf{n}_1 = \frac{1}{\sqrt{2}} \begin{pmatrix}
1 \\ 1
\end{pmatrix}, \quad
\mathbf{n}_2 = \frac{1}{\sqrt{2}} \begin{pmatrix}
1 \\ - 1
\end{pmatrix}
\end{equation*}
as for the Hessian of $\phi(\xi,\eta) = \xi^2 - \eta^2$ we have
\begin{equation*}
\mathbf{n}_i^t \begin{pmatrix}
1 & 0 \\ 0 & - 1
\end{pmatrix}
\mathbf{n}_i = 0.
\end{equation*}
Now, for $(\xi_0,\eta_0) \in \theta$, consider $(\xi_0,\eta_0) + \ell_1 \mathbf{n}_1 + \ell_2 \mathbf{n}_2 \in \theta$ and we obtain by Taylor expansion:
\begin{equation}
\label{eq:AreaNormalization}
\begin{split}
&\quad \phi((\xi_0,\eta_0) + \ell_1 \mathbf{n}_1 + \ell_2 \mathbf{n}_2) \\
 &= \phi(\xi_0,\eta_0) + \langle \ell_1 \mathbf{n}_1 + \ell_2 \mathbf{n}_2, \nabla \phi(\xi_0,\eta_0) \rangle + \frac{1}{2} \langle \ell_1 \mathbf{n}_1 + \ell_2 \mathbf{n}_2, \partial^2 \phi (\ell_1 \mathbf{n}_1 + \ell_2 \mathbf{n}_2) \rangle \\
 &= \phi(\xi_0,\eta_0) + \langle \ell_1 \mathbf{n}_1 + \ell_2 \mathbf{n}_2,  \nabla \phi(\xi_0,\eta_0) \rangle + \ell_1 \ell_2.
\end{split}
\end{equation}
We obtain from the flatness condition $|\ell_1 \ell_2| \leq \delta$. Let $\ell_{\max} = \max( |\ell_1|,|\ell_2| )$, $\ell_{\min} = \min( |\ell_1|, |\ell_2|)$.
We reverse the continuous approximation and the scaling to find the discrete decoupling inequality:
\begin{equation*}
\big\| e^{it(\partial_1^2 - \partial_2^2)} P_N f \big\|_{L^4_{t,x,y}([0,1]; \R \times \T)} \lesssim_\varepsilon \delta^{-\varepsilon} \big( \sum_{\theta \in \Theta_{\delta}} \big\| e^{it (\partial_1^2 - \partial_2^2)} P_{N \theta} f \big\|^2_{L^4_{t,x,y}([0,1];\R \times \T)} \big)^{\frac{1}{2}}
\end{equation*}
with $N \theta$ denoting the $N$-dilation of $\theta \in \Theta_{\delta}$. Following the general argument from the previous subsection we can integrate out the time evolution and the final ingredient is an estimate of the product measure $\text{meas}_{\R \times \Z}( N \theta)$. Write $N\theta \cap (\R \times \Z) = \bigcup_{\eta \in \pi_{\eta}(N \theta)} I_{\eta}$. We have
\begin{equation*}
\# \, \{ \, \eta \in \pi_{\eta}(N \theta) \, \} \, \lesssim \langle N \ell_{\max} \rangle, \quad |I_{\eta}| \lesssim N \ell_{\min},
\end{equation*} 
and consequently $\text{meas}_{\R \times \Z}( N \theta) \lesssim 1$. With this, the proof is concluded by Bernstein's inequality
\begin{equation*}
\| P_{N \theta} f \|_{L^4_{xy}(\R \times \T)} \lesssim \| P_{N \theta} f \|_{L^2_{xy}(\R \times \T)}
\end{equation*}
and logarithmic overlap of $\theta \in \Theta_{\delta}$.


\end{proof}

\begin{remark}
On the domain $\D = \T^2$ the flat sets remain the same, but the measure changes. For $ A' = \{ (\eta,\eta) \in \Z^2 : \eta \in [N/2,N] \cap \Z \} $
we find $\text{meas}_{\Z^2} (A') \sim N$, which matches the $L^4$-Strichartz estimate with derivative loss.
\end{remark}

\subsection{Decoupling and Strichartz estimates for generalized KP-II dispersion relation}

In this section we apply the general argument outlined in the preceding section to fractional KP-II equations.
We note that the dispersion relation
\begin{equation*}
\omega_\alpha(\xi,\eta) = \xi |\xi|^\alpha - \frac{\eta^2}{\xi}
\end{equation*}
is hyperbolic with principal curvatures having modulus essentially $1$.
\begin{lemma}
\label{lem:Hyperbolicity}
Let $\alpha \in [1,2]$. $\mathcal{M}_\alpha = \{(\xi,\eta,\omega_\alpha(\xi,\eta)) : \xi \sim 1, \, |\eta| \lesssim 1 \}$ is a non-elliptic hypersurface with principle curvatures $\lambda_1$, $\lambda_2$ satisyfing
\begin{equation*}
|\lambda_1| \sim |\lambda_2| \sim 1, \text{ and } \text{sgn}(\lambda_1 \lambda_2) = - 1.
\end{equation*}
\end{lemma}
\begin{proof}
We compute
\begin{equation*}
\partial^2 \omega_\alpha = 
\begin{pmatrix}
(\alpha+1) \alpha \xi^{\alpha-1} - 2 \eta^2 / \xi^3 & 2 \eta/\xi^2 \\
2 \eta/\xi^2 & - 2/\xi
\end{pmatrix}
.
\end{equation*}
Clearly, it holds
\begin{equation*}
\det(\partial^2 \omega_\alpha) = -2(\alpha+1)\alpha \xi^{\alpha-2} \Rightarrow \det(\partial^2 \omega_\alpha) \sim -1.
\end{equation*}
Moreover, $|\text{tr}(\partial^2 \omega_\alpha) | \lesssim 1$. This shows that the eigenvalues of $\partial^2 \omega_\alpha$ are of opposite sign and have modulus comparable to $1$.
\end{proof}

By the above considerations we need to analyze the $\delta$-flat sets for 
\begin{equation*}
\mathcal{M}_{\alpha} = \{ (\xi,\eta,\omega_{\alpha}(\xi,\eta)) : |\xi| \sim 1, \; |\eta| \lesssim 1 \}.
\end{equation*}
 The case of $\alpha=2$ was detailed in \cite{HerrSchippaTzvetkov2024}:
\begin{proposition}[{\cite[Proposition~3.4,~Remark~3.5]{HerrSchippaTzvetkov2024}}]
Let $A \subseteq \R^2$ be a $\delta$-flat set for $\mathcal{M}_2$. Then $A$ is contained in a $\delta^{\frac{1}{3}}$-ball. Moreover, for $(\xi_0,\eta_0) \in A$, we have that $|\eta-\eta_0| \lesssim \delta^{\frac{1}{2}}$ for any $(\xi_0,\eta) \in A$.
\end{proposition}

Like in case of the KP-II dispersion relation we show the following:
\begin{proposition}
\label{prop:FlatnessAB}
Let $\alpha \in [1,2]$ and $\theta$ a $\delta$-flat rectangle for $\mathcal{M}_\alpha$. We have the following properties:
\begin{itemize}
\item[(A)] $\theta$ is contained in a $\delta^{\frac{1}{3}}$-ball.
\item[(B)] For any $(\xi_0,\eta_0), (\xi_0,\eta) \in \theta$ it holds $|\eta - \eta_0| \lesssim \delta^{\frac{1}{2}}$. 
\end{itemize}
\end{proposition}
\begin{proof}
Let $(\xi_0,\eta_0) \in \theta$ and consider a curve
\begin{equation*}
\gamma(t) = (\xi_0,\eta_0) + t (\xi',\eta')
\end{equation*}
with $|(\xi',\eta')| = 1$. For (A) it suffices to show that
\begin{equation}
\label{eq:FlatnessCoefficients}
\omega_\alpha(\gamma(t)) = a_0 + a_1 t + a_2 t^2 + a_3 t^3 + E(t)
\end{equation}
with $|a_2|+|a_3| \gtrsim 1$ and $|E(t)| \leq c t^4$. We have
\begin{equation}
\label{eq:RepOmegaGamma}
\omega_\alpha(\xi_0+t\xi',\eta_0+t\eta') = (\xi_0+t\xi')^{\alpha+1} - \frac{(\eta_0+t\eta')^2}{\xi_0 + t \xi'}.
\end{equation}
First, we handle the simple case $|\xi'| \ll 1$, hence $|\eta'| \gtrsim 1$. We find
\begin{equation*}
\begin{split}
(\xi_0+t\xi')^{\alpha+1} - \frac{(\eta_0+t \eta')^2}{\xi_0 + t \xi'} &= \xi_0^{\alpha+1} + \mathcal{O}(\xi') - \frac{\eta_0^2 + 2 t \eta_0 \eta' + t^2 (\eta')^2}{\xi_0}(1+ \mathcal{O}(\xi')) \\
&= \xi_0^{\alpha+1} - \frac{\eta_0^2}{\xi_0} + \frac{2 t \eta_0 }{\xi_0} - \frac{t^2 (\eta')^2}{\xi_0} + \mathcal{O}(\xi').
\end{split}
\end{equation*}
This gives $a_2 = - \frac{(\eta')^2}{\xi_0} + \mathcal{O}(\xi')$ and hence $|a_2| \gtrsim 1$. Moreover, this shows that $\theta$ is extended in $\eta$-direction with length $\delta^{\frac{1}{2}}$, which yields (B).

\smallskip

We turn to the case $|\xi'| \gtrsim 1$. We can now apply the Galilean invariance:
\begin{equation}
\label{eq:GalileanInvariance}
\omega_{\alpha}(\xi,\eta + A \xi) = \omega_\alpha(\xi,\eta) - 2 A \eta + A^2 \xi
\end{equation}
with $\xi = \xi_0 + t \xi'$, $\eta = \eta_0 + t \eta'$, $A = - \eta' / \xi'$. It follows
\begin{equation*}
\omega_\alpha(\xi,\eta) = \omega_\alpha(\xi,\eta+A \xi) + 2 A \eta - A^2 \xi = \omega_\alpha(\xi,\eta - \frac{\eta'}{\xi'} \xi) + \text{Lin}(t).
\end{equation*}
The linear term can be disregarded. Let $\bar{\eta} = \eta_0 - \frac{\eta'}{\xi'} \xi_0$. We find
\begin{equation*}
\begin{split}
(\xi_0 + t \xi')^{\alpha+1} &= \xi_0^{\alpha+1} + (\alpha+1) t \xi' \xi_0^{\alpha} + (\alpha + 1) \alpha \xi_0^{\alpha-1} \frac{(t\xi')^2}{2} \\
&\quad + (\alpha+1) \alpha (\alpha-1) \xi_0^{\alpha-2} \frac{(t \xi')^3}{3!} + \mathcal{O}(t^4)
\end{split}
\end{equation*}
and
\begin{equation*}
\frac{\bar{\eta}^2}{\xi_0 + t \xi'} = \frac{\bar{\eta}^2}{\xi_0} ( 1- \frac{t \xi'}{\xi_0} + \frac{t^2 (\xi')^2}{\xi_0^2} - \frac{t^3 (\xi')^3}{\xi_0^3} + \mathcal{O}(t^4)).
\end{equation*}
Taking the above identities together, we find
\begin{equation*}
\begin{split}
\omega_\alpha(\xi,\eta - \frac{\eta'}{\xi'} \xi) &= \omega_\alpha(\xi,\bar{\eta}) = (\xi_0 + t \xi')^{\alpha+1} - \frac{\bar{\eta}^2}{\xi_0 + t \xi'} \\
 \\
&= \xi_0^{\alpha+1} + \text{Lin}(t) + t^2\big( \frac{(\alpha+1) \alpha \xi_0^{\alpha-1}}{2} - \frac{\bar{\eta}^2}{\xi_0^3} \big) \\
&\quad + t^3 (\xi')^3 \big( \frac{(\alpha+1) \alpha (\alpha-1) \xi_0^{\alpha-2}}{3!} + \frac{\bar{\eta}^2}{\xi_0^4} \big) + \mathcal{O}(t^4).
\end{split}
\end{equation*}
We see that $a_2$ can vanish in the case $|\bar{\eta}| \sim 1$, which necessitates $|\eta_0| \sim 1$ since $\xi_0 \sim 1$ and $|\eta'/\xi'| \ll 1$. For (A) it suffices to observe that $|a_3| \gtrsim 1$, which yields a direction $(\xi',\eta')$, $|\xi'| \gtrsim 1$, into which $\theta$ is extended with length $\lesssim \delta^{\frac{1}{3}}$. The proof is complete.

\end{proof}

To obtain a more precise estimate for the measure of $\delta$-flat sets, we recall that these take the form of parallelograms and we observe that in case of normalized derivatives these have area $\sim \delta$ as follows basically from the same argument as in \eqref{eq:AreaNormalization}. Indeed, in the case of perturbed hyperboloids the $\delta$-flat sets $S$ constructed in \cite[Section~2.1]{GuthMaldagueOh2024} are comparable to rectangles with one side pointing into a null direction.

Note that in the more general case of $\mathcal{M}_\alpha$, we have indeed uniformly bounded derivatives in the region $\{ \xi \sim 1, \; |\eta| \lesssim 1 \}$. On the other hand, for $\alpha=2$ we find $\delta = N^{-3}$ and a flat set with area larger than $\delta^{1-\varepsilon}$ for some $\varepsilon > 0$ would give a counterexample to the Strichartz estimate on $\R^2$ without derivative loss. Recall the following due to Saut \cite{Saut1993}:
\begin{equation}
\label{eq:SautStrichartz}
\| S_\alpha(t) f \|_{L^4_{t,x,y}([0,1];\R^2)} \lesssim \| f \|_{L^2(\R^2)}.
\end{equation}
The reason is that after anisotropic dilation $\xi \to N \xi$, $\eta \to N^2 \eta$ (see below), the area of the rescaled flat set $\theta_{N,N^2}$ would be $\gtrsim N^{\varepsilon}$. We could take $\hat{f} = \chi_{\theta_{N,N^2}}$ and since there is no time oscillation, we obtain from sharpness of Bernstein's inequality
\begin{equation*}
\| S_\alpha(t) f \|_{L^4_{t,x,y}([0,1];\R^2 )} \sim \| f \|_{L^4_{x,y}(\R^2)} \sim N^{\frac{\varepsilon}{4}} \| f \|_{L^2_{x,y}}.
\end{equation*}
This clearly contradicts \eqref{eq:SautStrichartz}.

The case of perturbed hyperboloids is recovered for $\mathcal{M}_\alpha$ by mild anisotropic dilation and rotation as pointed out in Lemma \ref{lem:Hyperbolicity}. Hence, the flat sets for $\mathcal{M}_\alpha$ can be chosen as parallelograms with angle $\gtrsim 1$ and area $\sim \delta$. That having been said, we have the following characterization for $\delta$-flat sets:

\begin{proposition}
Up to Galilean transform the $\delta$-flat parallelograms for $\mathcal{M}_\alpha$, $\alpha \in [1,2]$ are contained in rectangles of size $(\ell_1,\ell_2)$, which are extended by length $\ell_i$ into $e_i$-direction, and satisfy the bounds $\delta^{\frac{1}{2}} \lesssim \ell_1 \lesssim \delta^{\frac{1}{3}}$, $\ell_2 = \delta \ell_1^{-1}$.
\end{proposition}
\begin{proof}
The null directions $(\xi',\eta')$ are the ones annihilating the quadratic term in \eqref{eq:FlatnessCoefficients} and for this reason satisfy $|\xi'| \gtrsim 1$. So, carrying out the Galilean transform, these can be transformed to the $e_1$-direction. Consequently, the first null direction of the transformed parallelogram is $e_1$ and the length satisfies $\ell_1 \lesssim \delta^{\frac{1}{3}}$.

In the $e_2$-direction, by Proposition \ref{prop:FlatnessAB} we know that the length is at most $\delta^{\frac{1}{2}}$. This points out that the extension $\ell_2$ in the second direction is at most $\delta^{\frac{1}{2}}$. By the area constraint $\ell_1 \ell_2 \lesssim \delta$, we have for some fixed $B \gg 1$:
\begin{equation*}
\delta^{\frac{1}{2}} \lesssim \ell_1 \lesssim \delta^{\frac{1}{3}}, \quad \ell_2 = \delta \ell_1^{-1} \in (B^{-1} \delta^{\frac{2}{3}},B \delta^{\frac{1}{2}} ).
\end{equation*} 
\end{proof}

We intend to show Strichartz estimates in the Fourier region:
\begin{equation*}
A_{N,N^{\frac{\alpha}{2}+1}} = \{ (\xi,\eta) \in \R^2 : |\xi| \sim N, \; |\eta| \lesssim N^{\frac{\alpha}{2}+1} \}.
\end{equation*}
Recall that the scaling is given by
\begin{equation*}
\xi \to \xi/N, \quad \eta \to \eta/N^{\frac{\alpha}{2}+1}
\end{equation*}
and with $t \to N^{\alpha+1} t$, on the unit time scale we have $\delta= N^{-(\alpha+1)}$. 

%

Consider the case $\{|\eta| \lesssim N^{\frac{\alpha}{2}+1} \}$. After rescaling the $\delta$-flat set, we find sets of length $\ell_1' \in (N^{\frac{1-\alpha}{2}},N^{\frac{2-\alpha}{3}})$ in the $\xi$-frequencies, and length $N^{-\frac{\alpha}{2}} \ell_1^{-1} = N^{1-\frac{\alpha}{2}} (\ell_1')^{-1}$ in the $\eta$-frequencies. 
\begin{figure}[h]
\label{fig:LargeDeltaSets}
        \begin{tikzpicture}
            \draw (-3,0)--(0,2);
            \draw (-3,0)--(-3,1);
            \draw (-3,1)--(0,3);
            \draw (0,2)--(0,3);
            \node[left] at(-3,0.5){$N^{-(\alpha +1)} \ell_1^{-1}$};
            \node[below] at(-1.2,0){$\ell_1 \in (B^{-1} N^{-\frac{\alpha+1}{2}}, B N^{-\frac{\alpha+1}{3}})$};
            \draw (3,0)--(3.5,0.2);
            \draw (3,0)--(3,1);
            \draw (3,1)--(3.5,1.2);
            \draw (3.5,0.2)--(3.5,1.2);
            \node[left] at(3.1,0.5){$N^{1-\frac{\alpha}{2}} (\ell_1')^{-1} $};
            \node[below] at(4,0){$\ell_1' \in (B^{-1} N^{\frac{1-\alpha}{2}}, B N^{\frac{2-\alpha}{3}})$};
           \draw[->](0.2,1.3)to[bend left](2.8,1);\node at (1.9,1.9) {$(\xi,\eta)\mapsto (N\xi,N^{\frac{\alpha}{2}+1}\eta) $};
        \end{tikzpicture} 
        \caption{Anisotropic scaling transforms $\delta$-flat sets.}
    \end{figure}
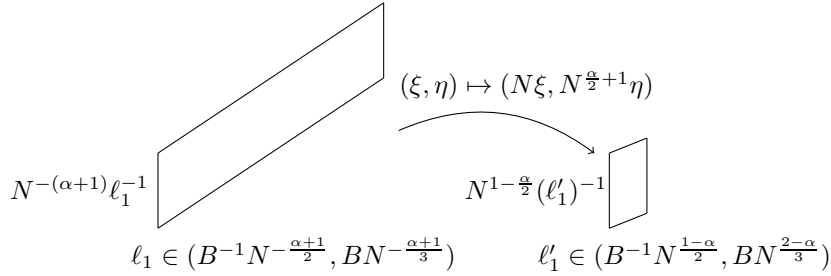
    

By following the strategy of applying decoupling outlined in Section \ref{subsection:DecouplingIntro}, we can show the following Strichartz estimates. Here we additionally cover preimages of $A_{N,N^2}$ under the Galilean transform $\eta \to \eta - A \xi$, which will be required for the nonlinear analysis.

\begin{proposition}
\label{prop:LinearStrichartzEstimate}
Let $\alpha \in [1,2]$. For any $N \in 2^{\N_0}$, $A\in \R$, and $f \in L^2(\R \times \T)$

with the property
\begin{equation*}
\text{supp}(\hat{f}) \subseteq \{ (\xi,\eta) \in \R^2 : |\xi| \sim N, \; \big| \frac{\eta}{\xi} - A \big| \lesssim N^{\frac{\alpha}{2}} \}.
\end{equation*}
we have the following estimate:
\begin{equation}
\label{eq:LinearStrichartzI}
\| S_{\alpha}(t) f \|_{L^4_{t}([0,1],L^4_{xy}(\R \times \T))} \lesssim_\varepsilon N^{\frac{2-\alpha}{8}+ \varepsilon} \| f \|_{L^2(\R \times \T)}.
\end{equation}
\end{proposition}
\begin{proof}
In the first step we handle the case $\{ |\eta| \lesssim N^{\frac{\alpha}{2}+1} \}$. Then we shall discuss the necessary adjustments under Galilean transform.

\smallskip

Moreover, we can suppose that $\{ \xi \sim N \}$ in the Fourier support of $f$. We use the anisotropic rescaling
\begin{equation*}
\xi \to \xi/N, \; \eta \to \eta/N^{\frac{\alpha}{2}+1}, \; x \to N x, \; y \to N^{\frac{\alpha}{2}+1} y, \; t \to N^{\alpha+1} t
\end{equation*}
to find
\begin{equation*}
\begin{split}
&\quad \big\| \int e^{i(x \xi + y \eta + t (\xi^{\alpha+1} - \frac{\eta^2}{\xi})} \hat{f}(\xi,\eta) d\xi (d\eta)_1 \big\|^4_{L_t^4([0,1],L^4_{xy}(\R \times \T))} = N^{-(\alpha+1)} N^{-(\frac{\alpha}{2}+2)} \\
&\quad \big\| \int e^{i(x\xi + y \eta + t (\xi^{\alpha+1} - \eta^2/\xi))} \underbrace{\hat{f}(N \xi,\eta)}_{\hat{f}_N(\xi,\eta)} N d\xi (d\eta)_{N^{-(\frac{\alpha}{2}+1)}} \big\|^4_{L_t^4([0,N^{\alpha+1}],L^4_{xy}(\R \times [0,N^{\frac{\alpha}{2}+1}])}.
\end{split}
\end{equation*}
Above we indicate with $(d \eta)_{\kappa}$ the counting measure on $\kappa \Z$.

Now we use the spatial periodicity to inflate the $y$-domain of integration to $[0,N^{\alpha+1}]$. This incurs a factor of $N^{-\frac{\alpha}{2}}$:
\small
\begin{equation*}
\lesssim N^{-(2 \alpha+3)} \big\| \int e^{i(x\xi + y \eta + t(\xi^{\alpha+1} - \eta^2 / \xi))} \hat{f}_N(\xi,\eta) d\xi (d\eta)_{N^{-(\frac{\alpha}{2}+1)}} \big\|^4_{L_t^4([0,N^{\alpha+1}],L^4_{xy}(\R \times [0,N^{\alpha+1}]))}.
\end{equation*}
\normalsize

Next, we cover $\R \times [0,N^{\alpha+1}]$ with balls of size $N^{\alpha+1}$. By translation invariance, we can estimate any ball like the ball centered at the origin.
Let
\begin{equation*}
\mathcal{E}_\alpha g(t,x,y) = \int_{\xi \sim 1, |\eta| \leq 1} e^{i( x \xi + y \eta + t \omega_\alpha(\xi,\eta))} g(\xi,\eta) d\xi d\eta.
\end{equation*}

 We approximate the partial sum with an integral, which is referred to as continuous approximation (see \cite[Section~2.2]{Schippa2024Refinements}):
\small
\begin{equation*}
\big\| \int e^{i(x \xi + y \eta + t (\xi^{\alpha+1} - \frac{\eta^2}{\xi}))} \hat{f}_N(\xi,\eta) d\xi (d\eta)_{N^{-(\alpha/2+1)}} \big\|_{L_{t,x,y}^4(B_{N^{\alpha+1}})} \lesssim \| \mathcal{E}_{\alpha} g_N \|_{L^4_{t,x,y}(B_{N^{\alpha+1}})}.
\end{equation*}
\normalsize
At this point we can apply the decoupling result with $\delta = N^{-(\alpha+1)}$ to find
\begin{equation*}
\| \mathcal{E}_{\alpha} g_N \|_{L^4_{t,x,y}(B_{N^{\alpha+1}})} \lesssim_\varepsilon N^\varepsilon \big( \sum_{\theta \in \Theta_{\delta}} \| \mathcal{E}_{\alpha} g_{N,\theta} \|_{L^4_{t,x,y}(w_{B_{N^{\alpha+1}}})}^2 \big)^{\frac{1}{2}}.
\end{equation*}
Now we reverse the continuous approximation, carry out the summation over $N^{\alpha+1}$-balls, and finally reverse the scaling to obtain the estimate:
\begin{equation}
\label{eq:DecouplingRescalingI}
\begin{split}
&\quad \big\| \int e^{i(x \xi + y \eta + t(\xi^{\alpha+1} - \eta^2/\xi))} \hat{f}(\xi,\eta) d\xi (d\eta)_1 \big\|_{L^4_{t,x,y}([0,1] \times \R \times \T)} \\
&\lesssim_\varepsilon N^\varepsilon \big( \sum_{\theta \in \Theta_{\delta}} \big\| \int_{\theta_{N,N^{\frac{\alpha}{2}+1}}} e^{i(x \xi + y \eta + t(\xi^{\alpha+1} - \eta^2/\xi))} \hat{f}(\xi,\eta) d\xi (d\eta)_1 \big\|^2_{L_t^4([0,1],L^4_{xy}(\R \times \T))} \big)^{\frac{1}{2}}.
\end{split}
\end{equation}
Above $\theta_{N,N^{\frac{\alpha}{2}+1}}$ denotes the anisotropic rescaling $(\xi,\eta) \to (N \xi, N^{\frac{\alpha}{2}+1} \eta)$ of the $N^{-(\alpha+1)}$-flat set $\theta \in \Theta_{\delta}$. It suffices now to show an estimate
\begin{equation}
\label{eq:DecouplingFlatEstimateII}
\begin{split}
&\quad \big\| \int_{\theta_{N,N^{\frac{\alpha}{2}+1}}} e^{i(x \xi + y \eta + t (\xi^{\alpha+1} - \eta^2 / \xi))} \hat{f}(\xi,\eta) d\xi (d\eta)_1 \big\|_{L_t^4([0,1],L^4_{xy}(\R \times \T))} \\
&\leq C(N) \| f_{\theta_{N,N^{\frac{\alpha}{2}+1}}} \|_{L^2}
\end{split}
\end{equation}
because by Theorem \ref{thm:GeneralL2Decoupling} the sets $\theta_{N,N^{\frac{\alpha}{2}+1}}$ are only logarithmically overlapping. Consequently, \eqref{eq:DecouplingRescalingI} and \eqref{eq:DecouplingFlatEstimateII} imply
\begin{equation*}
\| S_{\alpha}(t) f \|_{L^4_t([0,1],L^4_{xy}(\R \times \T))} \lesssim_\varepsilon N^\varepsilon C(N) \| f \|_{L^2_{xy}(\R \times \T)}.
\end{equation*}

\smallskip

When establishing \eqref{eq:DecouplingFlatEstimateII}, it suffices to show
\begin{equation*}
\| f_{\theta_{N,N^{\frac{\alpha}{2}+1}}} \|_{L^4_{x,y}(\R \times \T)} \lesssim C(N) \| f_{\theta_{N,N^{\frac{\alpha}{2}+1}}} \|_{L^2}
\end{equation*}
by integrating out the time due to flatness of $\theta_{N,N^{\frac{\alpha}{2}+1}}$.

We turn to establishing the inequality in the above display via Bernstein's inequality, where the size of $\theta_{N,N^{\frac{\alpha}{2}+1}}$ becomes relevant. 

\emph{Proof of \eqref{eq:LinearStrichartzI} in case $\{ |\eta| \lesssim N^{\frac{\alpha}{2}+1} \}$.} Then we are in the situation depicted in Figure \ref{fig:LargeDeltaSets}.
We have by Fubini's theorem
\begin{equation*}
\begin{split}
\text{meas}_{\R \times \Z}(\theta_{N,N^{\frac{\alpha}{2}+1}}) &= \int_{\xi \in \pi_{\xi}(\theta_{N,N^{\frac{\alpha}{2}+1}})} \# \{ \eta : (\xi,\eta) \in \theta_{N,N^{\frac{\alpha}{2}+1}} \} d \xi \\
&\sim \int_{\xi \in \pi_{\xi}(\theta_{N,N^{\frac{\alpha}{2}+1}})}  \text{meas}_{\R} (\{ \eta \in \R: (\xi,\eta) \in \theta_{N,N^{\frac{\alpha}{2}+1}} \}) d\xi \\
&= \text{meas}_{\R^2}(\theta_{N,N^{\frac{\alpha}{2}+1}}).
\end{split}
\end{equation*}
Above we used that 
\begin{equation*}
\# \{ \eta : (\xi,\eta) \in \theta_{N,N^{\frac{\alpha}{2}+1}} \}
\sim N^{1-\frac{\alpha}{2}} (\ell_1')^{-1} \gtrsim 1.
\end{equation*}
And readily,
\begin{equation*}
\text{meas}_{\R^2}(\theta_{N,N^{\frac{\alpha}{2}+1}}) \lesssim N^{1-\frac{\alpha}{2}}.
\end{equation*}
For this reason, an application of Bernstein's inequality yields
\begin{equation}
\label{eq:BernsteinLowEta}
\| f_{\theta_{N,N^{\frac{\alpha}{2}+1}}} \|_{L^4_{x,y}(\R \times \T)} \lesssim N^{\frac{2-\alpha}{8}} \|f_{\theta_{N,N^{\frac{\alpha}{2}+1}}} \|_{L^2_{x,y}(\R \times \T)}.
\end{equation}


\smallskip

\emph{Proof of \eqref{eq:LinearStrichartzI} in the general case $\{ \big| \frac{\eta}{\xi} - A \big| \ll N^{\frac{\alpha}{2}} \}$ for $A \gtrsim N^{\frac{\alpha}{2}}$.}
Note that the Galilean transform $G: (\xi,\eta) \mapsto (\xi'=\xi,\eta'=\eta + A\xi)$ maps the region $\big| \frac{\eta}{\xi}-A \big| \ll N^{\frac{\alpha}{2}}$ to $\big| \frac{\eta'}{\xi'} \big| \ll N^{\frac{\alpha}{2}}$. The idea is to apply decoupling to the transformed frequencies, which will give a decomposition into flat sets with respect to $(\xi',\eta')$. Then we reverse the Galilean transform and we need to estimate the measure of the preimage $G^{-1} \theta_{N,N^{\frac{\alpha}{2}+1}}$ under Galilean transform. We argue like above with the comparability
\begin{equation*}
\text{meas}_{\R \times \Z} (\theta_{N,N^{\frac{\alpha}{2}+1}}) \sim \text{meas}_{\R^2} (\theta_{N,N^{\frac{\alpha}{2}+1}}).
\end{equation*}
This allows us to dispense the Galilean transform, which clearly does not change the Lebesgue measure. We recover \eqref{eq:BernsteinLowEta} independently of $A$, which proves \eqref{eq:LinearStrichartzI} in the general case. The proof is complete.

%
\end{proof}

\begin{remark}
The derivative loss for the Strichartz estimates are sharp up to endpoints. This follows from the comparison with the sharp estimates in the Euclidean case, also for fractional dispersion relation. Indeed, for $\alpha \in [1,2]$, Hadac \cite[Theorem~3.1]{Hadac2008} proved the scaling critical estimate
\begin{equation*}
\| S_{\alpha}(t) f \|_{L_t^4([0,1],L^4_{xy}(\R^2))} \lesssim \| |D_x|^{\frac{2-\alpha}{8}} f \|_{L^2_{xy}(\R^2)}.
\end{equation*}

\end{remark}

\subsection{Strichartz estimates for $\{ |\xi| \sim N, \; |\frac{\eta}{\xi} - A| \lesssim N^{\frac{\alpha}{2}+\gamma} \}$ for $\gamma \in (0,1]$}

Decompose $[0,N^{\frac{\alpha}{2}+\gamma}]$ into intervals of length $N^{\frac{\alpha}{2}}$. Then, after applying a Galilean transform, we have that $|\bar{\eta}/\xi|$ is contained in an interval of length $N^{\frac{\alpha}{2}}$, i.e., $\bar{\eta}$ is contained in an interval of length $N^{\frac{\alpha}{2}+1}$.

\smallskip

 After shifting the range of $\eta'$-frequencies to the origin, the transformed phase function is defined on:
\begin{equation*}
\{ |\xi| \sim N, \quad |\eta'| \lesssim N^{\frac{\alpha}{2}+1} \}.
\end{equation*}
Rescaling $\xi \to \xi / N = \xi'$, $\eta' \to \eta'/ N^{\frac{\alpha}{2}+1} $ yields
\begin{equation*}
\omega_\alpha(\xi',\eta') = \xi' |\xi'|^{\alpha} - \frac{(\eta'  - \eta^*)^2}{\xi'}.
\end{equation*}
Note that $|\eta^*| \lesssim N^\gamma$. To obtain a phase function with uniform bounded derivatives we subdivide $\xi'$ into intervals of length $N^{-\gamma}$. The resulting phase function takes the form
\begin{equation}
\label{eq:ShiftedPhaseFunction}
\bar{\omega}_\alpha(\xi',\eta') = (\kappa^{-1} \xi' + \xi^*)^{\alpha+1} - \frac{(\eta' - \eta^*)^2}{\xi' \kappa^{-1} + \xi^*}
\end{equation}
for $|\xi'| \lesssim 1$, $\kappa = |\eta^*| \lesssim N^{\gamma}$, $\xi_0' \sim 1$. In the first step we obtain the $\delta$-flat sets for $\bar{\omega}_\alpha$, then we need to reverse the Galilean transform and estimate the measure of the preimages.

\smallskip

We have the following characterization of $\delta$-flat sets of $\bar{\omega}_{\alpha}$:
\begin{proposition}
\label{prop:FlatLargeEta}
Let $\xi_0,\eta_0,\xi,\eta$ be like above, and $\alpha \in [1,2]$. Then the $\delta$-flat sets of $\bar{\omega}_\alpha$ defined in \eqref{eq:ShiftedPhaseFunction} are contained in rectangles of dimensions $(\ell_1,\ell_2)$ of length
\begin{equation*}
\delta^{\frac{1}{2}} \lesssim \ell_1 \lesssim \kappa \delta^{\frac{1}{3}}, \quad \ell_2 = \delta \ell_1^{-1}.
\end{equation*}
\end{proposition}
\begin{proof}
This can be proved following along the above lines. We shall be brief. The maximum size of $\delta^{\frac{1}{2}}$ into $e_2$-direction can be found exactly like above.

Next, we consider a curve parametrized by $(\xi,\eta) = (\xi_0 + t \xi', \eta_0 + t \eta')$ with $|\xi'| \gtrsim 1$. Let
\begin{equation*}
\bar{\xi} = \kappa^{-1} (\xi_0 + t \xi') + \xi^*, \quad \bar{\eta} = \eta_0 - t \eta' - \eta^*.
\end{equation*}
We consider the Galilean transformation $\bar{\eta} \to \bar{\eta} + A \bar{\xi}$ with $A= \frac{\kappa \eta'}{\xi'}$, which leads us to consider
\begin{equation*}
\begin{split}
&\quad (\kappa^{-1} (\xi_0 + t \xi') + \xi^*)^{\alpha+1} - \frac{(\bar{\eta} + A \bar{\xi})^2}{\bar{\xi}} \\
 &= (\kappa^{-1} (\xi_0 + t \xi') + \xi^*)^{\alpha+1} - \frac{\bar{\eta}^2}{\bar{\xi}} + \text{Lin}(t) \\
&= (\kappa^{-1} (\xi_0 + t \xi') + \xi^*)^{\alpha+1} - \frac{(\eta_0 + \frac{\eta' \xi_0}{\xi'} + \frac{\kappa \eta' \xi^*}{\xi'} - \eta^*)^2}{\kappa^{-1} \xi + \xi^*} + \text{Lin}(t) \\
&= (\kappa^{-1} (\xi_0 + t \xi') + \xi^*)^{\alpha+1} - \frac{(\bar{\eta}_0 - \eta^*)^2}{\kappa^{-1}\xi + \xi^*} + \text{Lin}(t).
\end{split}
\end{equation*}

Next, the expression is expanded in $t$. Note that the derivatives are uniformly bounded by means of the renormalization in $\xi$. Let $\bar{\xi}_0 = \xi^* + \kappa^{-1} \xi_0$ for brevity.

We find
\begin{equation*}
\begin{split}
(\bar{\xi}_0 + \kappa^{-1} t \xi')^{\alpha+1} &= \bar{\xi}_0^{\alpha+1} + (\alpha+1) \bar{\xi}_0^{\alpha} \kappa^{-1} t \xi' + (\alpha+1) \alpha \bar{\xi}_0^{\alpha-1} \frac{(\kappa^{-1} t \xi')^2}{2} \\
&\quad + (\alpha+1) \alpha (\alpha-1) \bar{\xi}_0^{\alpha-2} \frac{(\kappa^{-1} t \xi')^3}{3!} + \mathcal{O}(t^4).
\end{split}
\end{equation*}
Secondly,
\begin{equation*}
\frac{(\bar{\eta}_0 - \eta^*)^2}{\bar{\xi}_0(1+ \frac{t \kappa^{-1} \xi'}{\xi_0})} = \frac{(\bar{\eta}_0 - \eta^*)^2}{\bar{\xi}_0} (1 - \frac{t \kappa^{-1} \xi'}{\bar{\xi}_0} + \frac{t^2 \kappa^{-2} (\xi')^2}{\bar{\xi}_0^2} - \frac{t^3 \kappa^{-3} (\xi')^3}{\bar{\xi}_0^3} + \mathcal{O}(t^4)).
\end{equation*}
To summarize, we find
\begin{equation*}
\begin{split}
&\quad (\bar{\xi}_0 + \kappa^{-1} t \xi')^{\alpha+1} - \frac{(\bar{\eta}_0 - \eta^*)^2}{\bar{\xi}_0 ( 1+ \frac{t \kappa^{-1} \xi'}{\xi_0}} ) \\
 &= \text{Lin}(t) + \big[ (\alpha+1) \alpha \bar{\xi}_0^{\alpha-1} \frac{\kappa^{-2} t^2 (\xi')^2}{2} - \frac{(\bar{\eta}_0  -\eta^*)^2}{\bar{\xi}_0} \frac{t^2 \kappa^{-2} (\xi')^2}{\bar{\xi}_0^3} \big] \\
&\quad + \big[ (\alpha+1) \alpha (\alpha-1) \bar{\xi}_0^{\alpha-2} \frac{(\kappa^{-1} t \xi')^3}{3!} + \frac{(\bar{\eta}_0  -\eta^*)^2}{\bar{\xi}_0} \frac{t^3 \kappa^{-3} (\xi')^3}{\bar{\xi}_0^3} \big] + \mathcal{O}(t^4).
\end{split}
\end{equation*}
Like above, the quadratic term may vanish, but the cubic term is always $\gtrsim \kappa^{-3}$. This shows that the maximal size of flat parallelograms in $e_1$-direction is at most $\kappa \delta^{\frac{1}{3}}$. By the maximum length of $\delta^{\frac{1}{2}}$ into $e_2$-direction we have found
\begin{equation*}
\delta^{\frac{1}{2}} \lesssim \ell_1 \lesssim \kappa \delta^{\frac{1}{3}}, \quad \ell_2 = \delta \ell_1^{-1}.
\end{equation*}
\end{proof}

 We are ready to show the following Strichartz estimates:
\begin{proposition}
\label{prop:StrichartzLargeEta}
Let $f: \R \times \T \to \C$ with $\text{supp}(\hat{f}) \subseteq \{(\xi,\eta) \in \R^2 : |\xi| \sim N, \xi \in I, |I| \sim N/k, |\frac{\eta}{\xi} - A| \in [k N^{\frac{\alpha}{2}}, (k+1) N^{\frac{\alpha}{2}}] \}$ with $1 \ll k \lesssim N$. Then the following estimate holds:
\begin{equation*}
\| S_\alpha(t) f \|_{L_{t,x,y}^4([0,1], \R \times \T)} \lesssim_\varepsilon (N^{\frac{2-\alpha}{12}+\varepsilon} k^{\frac{1}{12}} \vee N^{\frac{2-\alpha}{8}+\varepsilon}) \| f \|_{L^2(\R \times \T)}.
\end{equation*}
\end{proposition}
\begin{proof}
Like in the proof of Proposition \ref{prop:LinearStrichartzEstimate}, we first deal with the case
$A \lesssim N^{\frac{\alpha}{2}}$. We elaborate on the modifications in the general case below.

Use the scaling
\begin{equation*}
t \to N^{\frac{\alpha}{2}+1} t, \quad x \to N x, \quad y \to N^{\frac{\alpha}{2}+1} y
\end{equation*}
and periodicity in $y$ to find
\small
\begin{equation*}
\begin{split}
&\quad \| S_\alpha(t) f \|^4_{L_{t,x,y}^4([0,1]; \R \times \T)} \\
&= N^{-(\alpha+1)} N^{-\frac{\alpha}{2}} N^{-2} \\
&\; \times \big\| \int_{\xi' \in I'} \sum_{\eta' \in \Z / N^{\frac{\alpha}{2}+1}} e^{i(x' \xi' + y' \eta' + t'((\xi')^{\alpha+1} - \frac{(\eta'+\eta_0')^2}{\xi'})} \hat{f}(\xi',\eta') \big\|^4_{L^4_{t',x',y'}([0,N^{\alpha+1}] \times N^{\frac{\alpha}{2}+1} \T)} \\
&\lesssim N^{-(\alpha+1)} N^{-\frac{\alpha}{2}} N^{-2} N^{-\frac{\alpha}{2}} \\
&\; \times \big\| \int_{\xi' \in I'} \sum_{\eta' \in \Z / N^{\frac{\alpha}{2}+1}} e^{i(x' \xi' + y' \eta' + t'((\xi')^{\alpha+1} - \frac{(\eta'+\eta_0')^2}{\xi'})} \hat{f}(\xi',\eta') \big\|^4_{L^4_{t',x',y'}([0,N^{\alpha+1}] \times \R \times N^{\alpha+1} \T)}.
\end{split}
\end{equation*}
\normalsize
We have $\eta_0 \in [k-2,k+2]$ and $\xi' \in I'$ with $|I'| \sim 1/k \sim 1/\eta_0$.

\smallskip

We cover $[0,N^{\alpha+1}] \times \R \times N^{\alpha+1} \T$ with balls of size $N^{\alpha+1}$ and change to continuous approximation.

Normalize the $\xi'$-support by $\xi' = \xi_0 + \xi''/\eta_0'$. $x'$ is rescaled dually. The resulting domain of integration is partitioned into balls of size $\delta^{-1}$ with $\delta^{-1} =  N^{\alpha+1} (\eta_0')^{-1}$.

After continuous approximation we intend to apply decoupling to the oscillatory integral
\begin{equation*}
\begin{split}
&\quad \big\| \int_{(\xi'',\eta'') \in (-1,1)^2} e^{i(x'' \xi'' + y' \eta' + t' \bar{\omega}(\xi'',\eta'))} f''(\xi'',\eta') d\xi'' d\eta' \big\|_{L^4_{t',x'',y'}(B_{\delta^{-1}})} \\
&\lesssim \| w_{B_{\delta^{-1}}} \int_{(\xi'',\eta'') \in (-1,1)^2} e^{i(x'' \xi' + y' \eta' + t' \bar{\omega}(\xi'',\eta')} f''(\xi'',\eta') d\xi'' d\eta' \big\|_{L^4_{t',x'',y'}(\R^3)}
\end{split}
\end{equation*}

We find
\begin{equation}
\label{eq:DecouplingLargeEtaFrequencies}
\begin{split}
&\quad \big\| \int  \sum_{\eta} e^{i(x \xi + y \eta + t \omega_\alpha(\xi,\eta))} \hat{f}(\xi,\eta) d\xi \big\|_{L^4_{t,x,y}([0,1]; \R \times \T))} \\
&\lesssim_\varepsilon N^\varepsilon \big( \sum_{\substack{ \theta: \delta-\text{flat set} \\ \text{ for } \bar{\omega}}} \big\| \int  \sum_{\eta: (\xi,\eta) \in \delta_{N,N^{\frac{\alpha}{2}+1}} \theta} e^{i(x\xi + y \eta + t \omega_\alpha(\xi,\eta))} \hat{f}(\xi,\eta) d\xi \big\|^2_{L^4_{t,x,y}([0,1];\R \times \T)} \big)^{\frac{1}{2}}.
\end{split}
\end{equation}
By Proposition \ref{prop:FlatLargeEta} the $\delta$-flat sets can be contained in rectangles with long side into $e_1$-direction of length at most
\begin{equation*}
(N^{-(\alpha+1)} (\eta_0'))^{\frac{1}{2}} \lesssim \ell_1 \lesssim \eta_0' (\eta_0' N^{-(\alpha+1)})^{\frac{1}{3}}
\end{equation*}
and length $\ell_2 = \ell_1^{-1} \delta$ into $e_2$-direction.


After reversing the scaling and normalization of the Fourier support, i.e., dilating $\ell_1$ by $(\eta_0')^{-1} N$ and $\ell_2$ by $N^{\frac{\alpha}{2}+1}$ we find the following:
\begin{equation*}
N^{\frac{1-\alpha}{2}} (\eta_0')^{-\frac{1}{2}} = N^{- \frac{\alpha+1}{2}} N (\eta_0')^{-\frac{1}{2}} \lesssim \ell_1' \lesssim (\eta_0')^{\frac{1}{3}} N^{\frac{2-\alpha}{3}}.
\end{equation*}
We compute 
\begin{equation*}
\ell_2'= \ell_1^{-1} N^{-(\alpha+1)} \eta_0' N^{\frac{\alpha}{2}+1} = (\ell_1')^{-1} N^{1-\frac{\alpha}{2}}.
\end{equation*}

For $\ell_1' \lesssim N^{1-\frac{\alpha}{2}}$ we find $\ell_2' \gtrsim 1$ and 
\begin{equation*}
\text{meas}_{\R \times \Z}( \theta_{N,N^{\frac{\alpha}{2}+1}}) \sim \text{meas}_{\R^2}( \theta_{N,N^{\frac{\alpha}{2}+1}}) \sim N^{1-\frac{\alpha}{2}}.
\end{equation*}
This gives the estimate
\begin{equation*}
\big\| \int \sum_{\eta: (\xi,\eta) \in \theta_{N,N^{\frac{\alpha}{2}+1}}} e^{i(x\xi + y \eta + t \omega_\alpha(\xi,\eta))} \hat{f}(\xi,\eta) d\xi \big\|_{L^4_{t,x,y}([0,1];\R \times \T)} \lesssim N^{\frac{2-\alpha}{8}} \| f_{\theta} \|_{L^2_{xy}}.
\end{equation*}
This recovers the scaling critical estimate. However, in case $\ell_1' \gtrsim N^{1-\frac{\alpha}{2}}$, we possibly have $\ell_2' \lesssim 1$, in which case we have the worse estimate
\begin{equation*}
\text{meas}_{\R \times \Z} ( \theta_{N,N^{\frac{\alpha}{2}+1}}) \lesssim \ell_1' \lesssim (\eta_0')^{\frac{1}{3}} N^{\frac{2-\alpha}{3}}.
\end{equation*}
This finishes the proof of the estimate in case $\{ | \frac{\eta}{\xi} - A | \lesssim [k N^{\frac{\alpha}{2}}, (k+1) N^{\frac{\alpha}{2}} ] \}$ in case $|A| \lesssim N^{\frac{\alpha}{2}}$. To handle the case $|A| \gg N^{\frac{\alpha}{2}}$, we note that this is reduced to the previous case by means of a Galilean transform. But the estimates for the measures of the inflated flat sets are clearly invariant under Galilean transforms, which establishes the general case.

\end{proof}

\begin{remark}
It seems possible to improve the estimate on frequency-dependent times, which is presently omitted to lighten the argument. We refer to \cite[Section~3.5]{HerrSchippaTzvetkov2024} for further details.
\end{remark}
We record the following corollary, which will simplify the computations in the nonlinear analysis:
\begin{corollary}
\label{cor:StrichartzEstimateLargeEta}
Under the assumptions of Proposition \ref{prop:StrichartzLargeEta}, the following estimate holds:
\begin{equation*}
\| S_\alpha(t) f \|_{L_{t,x,y}^4([0,1], \R \times \T)} \lesssim_\varepsilon N^{\frac{2-\alpha}{8}+\varepsilon} k^{\frac{1}{12}} \| f \|_{L^2(\R \times \T)}.
\end{equation*}
\end{corollary}

\section{Bilinear Strichartz estimates}
\label{section:BilinearStrichartzEstimates}

We formulate two simple bilinear Strichartz estimates hinging on first and second order transversality. We start with first-order transversality:
\begin{lemma}
\label{lem:FirstOrderTransversality}
Let $N_i \in 2^{\Z}$, $L_i \in 2^{\N_0}$, $I_i \subseteq \R$, $i=1,2$ be two intervals with length $K \in 2^{\Z}$, and $f_{i,N_i,L_i}: \R \times \R \times \Z \to \R_{\geq 0}$ with $\text{supp}(f_{i,N_i,L_i}) \subseteq \{ (\xi,\eta,\tau) : \xi \in I_i, \; | \xi | \sim N_i, \; |\tau - \omega_\alpha(\xi,\eta)| \leq L_i \}$. Let $L_{\max} = \max(L_1,L_2)$, $L_{\min} = \min(L_1,L_2)$. Furthermore, we assume that
\begin{equation}
\label{eq:LowerBoundTransveralityI}
\big| \frac{\eta_1}{\xi_1} - \frac{\eta_2}{\xi_2} \big| \geq D > 0
\end{equation}
for $(\xi_i,\eta_i) \in \text{supp}_{\xi,\eta}(f_{i,N_i,L_i})$. 
Then the following estimate holds:
\begin{equation*}
\| f_{1,N_1,L_1} * f_{2,N_2,L_2} \|_{L^2_{\tau,\xi,\eta}} \lesssim K^{\frac{1}{2}} L_{\min}^{\frac{1}{2}} \langle L_{\max} / D \rangle^{\frac{1}{2}} \prod_{i=1}^2 \| f_{i,N_i,L_i} \|_{L^2}.
\end{equation*}
\end{lemma}
\begin{proof}
This is the analog of \cite[Proposition~4.2]{HerrSchippaTzvetkov2024} for functions $f_i: \R \times \R \times \Z \to \R_{\geq 0}$.
\end{proof}

Next, we record a consequence of second-order transversality, which holds independently of \eqref{eq:LowerBoundTransveralityI}. This will be very useful to estimate some limiting cases.
\begin{lemma}
\label{lem:SecondOrderTransversality}
Let $K, N_i, L_i, I_i, f_i: \R \times \R \times \Z \to \R_{\geq 0}$, $i=1,2$ be like in Lemma \ref{lem:FirstOrderTransversality}. Then the following estimate holds:
\begin{equation}
\label{eq:SecondOrderTransversality}
\| f_{1,N_1,L_1} * f_{2,N_2,L_2} \|_{L^2_{\tau,\xi,\eta}} \lesssim K^{\frac{1}{2}} L_{\min}^{\frac{1}{2}} \langle L_{\max} 
N_{\min} \rangle^{\frac{1}{4}} \prod_{i=1}^2 \| f_{i,N_i,L_i} \|_{L^2_{\tau,\xi,\eta}}.
\end{equation}
\end{lemma}
\begin{proof}
This the analog of \cite[Proposition~4.3]{HerrSchippaTzvetkov2024}.
\end{proof}

In the next proposition we combine the two estimates with almost-orthogonal decompositions in case of small transversality.
\begin{proposition}
\label{prop:BilinearStrichartzA}
Let $N_i \in 2^{\Z}$, $L_i \in 2^{\N_0}$, $i=1,2$, and $N_2 \ll N_1$. Let $f_{i,N_i,L_i} \in L^2(\R \times \R \times \Z)$ with $\text{supp}(f_{i,N_i,L_i}) \subseteq D_{N_i,L_i}$, $i=1,2$. Suppose that
\begin{equation*}
\big| \frac{\eta_1}{\xi_1} - \frac{\eta_2}{\xi_2} \big| \leq D^*
\end{equation*}
for $(\xi_i,\eta_i) \in \text{supp}_{\xi,\eta}(f_i)$, $i=1,2$. Then the following estimate holds:
\begin{equation}
\label{eq:BilinearStrichartzA}
\| f_{1,N_1,L_1} * f_{2,N_2,L_2} \|_{L^2_{\tau,\xi,\eta}} \lesssim \log(D^*) N_2^{\frac{1}{2}} L_{\min}^{\frac{1}{2}} \langle L_{\max} / N_1^{\frac{\alpha}{2}} \rangle^{\frac{1}{2}} \prod_{i=1}^2 \| f_{i,N_i,L_i} \|_{L^2}.
\end{equation}
\end{proposition}
\begin{proof}
First, suppose that we are in the case of large modulation: $L_{\max} \gtrsim N_1^{\alpha} N_2$. In this case the estimate is immediate from Lemma \ref{lem:SecondOrderTransversality}.

In the following we suppose that $L_{\max} \lesssim N_1^{\alpha} N_2$. We carry out a Whitney decomposition
\begin{equation*}
\| f_{1,N_1,L_1} * f_{2,N_2,L_2} \|_{L^2_{\tau,\xi,\eta}} \leq \sum_{\big( \frac{L_{\max}}{N_2} \big)^{\frac{1}{2}} \leq D \leq D^* } \sum_{J_1 \sim_D J_2} \| f^{J_1}_{1,N_1,L_1} * f^{J_2}_{2,N_2,L_2} \|_{L^2_{\tau,\xi,\eta}}.
\end{equation*}
Above $J_1$, $J_2$ denote intervals of length $D$ with $\text{dist}(J_1,J_2) \sim D$ such that $\eta_i / \xi_i \in J_i$ for $(\xi_i,\eta_i) \in \text{supp}_{\xi,\eta}(f_{i,N_i,L_i})$.

We note that by convolution constraint we have for $(\xi_i,\eta_i,\tau_i), \, (\xi_{i+2},\eta_{i+2},\tau_{i+2}) \in \text{supp}(f_{i,N_i,L_i})$, $i=1,2$:
\begin{equation*}
\left\{ \begin{array}{cl}
\xi_1 + \xi_2 &= \xi_3 + \xi_4, \\
\eta_1 + \eta_2 &= \eta_3 + \eta_4, \\
\xi_1^{\alpha+1} - \frac{\eta_1^2}{\xi_1} + \xi_2^{\alpha+1} - \frac{\eta_2^2}{\xi_2} &= \xi_3^{\alpha+1} - \frac{\eta_3^2}{\xi_3} + \xi_4^{\alpha+1} - \frac{\eta_4^2}{\xi_4} + \mathcal{O}(L_{\max}).
\end{array} \right.
\end{equation*}
We square the second line, divide by the first line and subtract the resulting expression from the third line to find:
\begin{equation*}
\left\{ \begin{array}{cl}
\xi_1 + \xi_2 &= \xi_3 + \xi_4, \\
\xi_1^{\alpha+1} + \xi_2^{\alpha+1} - \frac{(\eta_1 \xi_2 - \eta_2 \xi_1)^2}{\xi_1 \xi_2 (\xi_1 + \xi_2)} &= \xi_3^{\alpha+1} + \xi_4^{\alpha+1} - \frac{(\eta_3 \xi_4 - \eta_4 \xi_3)^2}{\xi_3 \xi_4 (\xi_3 + \xi_4)} + \mathcal{O}(L_{\max})
\end{array} \right.
\end{equation*}
We rescale the resulting expression to unit frequencies $\xi_i \to \xi_i / N_1$, $\eta_i \to \eta_i / N_1^{\frac{\alpha}{2}+1}$, noting that
\begin{equation*}
\frac{(\eta_1 \xi_2 - \eta_2 \xi_1)^2}{\xi_1 \xi_2 (\xi_1 + \xi_2)} \sim D^2 N_2.
\end{equation*}
This gives
\begin{equation*}
\left\{ \begin{array}{cl}
\xi_1' + \xi_2' &= \xi_3' + \xi_4', \\
(\xi_1')^{\alpha+1} + (\xi_2')^{\alpha+1} &= (\xi_3')^{\alpha+1} + (\xi_4')^{\alpha+1} + \mathcal{O} \big( \frac{L_{\max}}{N_1^{\alpha+1}} + \frac{D^2 N_2}{N_1^{\alpha+1}} \big).
\end{array} \right.
\end{equation*}

We use a variant of the C\'ordoba--Fefferman square function estimate to obtain almost orthogonal decomposition  of $\xi_i'$ into intervals of length $\frac{L_{\max}}{N_1^{\alpha+1}} + \frac{D^2 N_2}{N_1^{\alpha+1}}$, see e.g. \cite[Chapter~2]{Demeter2020}.

In the limiting case, $D = \big( \frac{L_{\max}}{N_2} \big)^{\frac{1}{2}}$, we find an almost orthogonal decomposition of $\xi_i'$ into intervals of length $L_{\max}/N_1^{\alpha+1}$. So, after inverting the scaling, we find that the $\xi_i$ are decomposed into intervals $\theta_i$ of length $\frac{L_{\max}}{N_1^{\alpha}}$. We write
\begin{equation*}
\sum_{J_1 \sim_D J_2} \| f^{J_1}_{1,N_1,L_1} * f^{J_2}_{2,N_2,L_2} \|_{L^2} \leq
\sum_{\substack{J_1 \sim_D J_2, \\ \theta_1 \sim \theta_2}} \| f^{J_1,\theta_1}_{1,N_1,L_1} * f^{J_2,\theta_2}_{2,N_2,L_2} \|_{L^2_{\tau,\xi,\eta}}.
\end{equation*}
This expression is estimated by Lemma \ref{lem:SecondOrderTransversality}:
\begin{equation*}
\| f_{1,N_1,L_1}^{J_1,\theta_1} * f_{2,N_2,L_2}^{J_2,\theta_2} \|_{L^2_{\tau,\xi,\eta}} \lesssim \big( \frac{L_{\max}}{N_1^{\alpha+1}} \big)^{\frac{1}{2}} L_{\min}^{\frac{1}{2}} \langle L_{\max} N_2 \rangle^{\frac{1}{4}} \prod_{i=1}^2 \| f_{i,N_i,L_i}^{J_i,\theta_i} \|_{L^2}.
\end{equation*}
The sum over $J_i$ and $\theta_i$ can be carried out without loss by almost orthogonality. By the assumption $L_{\max} \lesssim N_1^\alpha N_2$ this estimate suffices for \eqref{eq:BilinearStrichartzA}.

We turn to $\big( \frac{L_{\max}}{N_2} \big)^{\frac{1}{2}} < D \leq N_1^{\frac{\alpha}{2}}$: In this case the identity
\begin{equation*}
\left\{ \begin{array}{cl}
\xi_1' + \xi_2' &= \xi_3' + \xi_4', \\
(\xi_1')^{\alpha+1} + (\xi_2')^{\alpha+1} &= (\xi_3')^{\alpha+1} + (\xi_4')^{\alpha+1} + \mathcal{O} \big( \frac{D^2 N_2}{N_1^{\alpha+1}} \big)
\end{array} \right.
\end{equation*}
yields an almost orthogonal decomposition of $\xi_i'$ into intervals of length $\frac{D^2 N_2}{N_1^{\alpha+1}}$ and after rescaling, we find a decomposition of the $\xi_i$ into intervals $\theta_i$ of length $\frac{D^2 N_2}{N_1^\alpha}$.
We obtain from the decomposition and employing Lemma \ref{lem:FirstOrderTransversality}:
\begin{equation*}
\begin{split}
&\quad \sum_{\big( \frac{L_{\max}}{N_2} \big)^{\frac{1}{2}} < D \leq N_1^{\frac{\alpha}{2}}} \sum_{J_1 \sim_D J_2} \| f^{J_1}_{1,N_1,L_1} * f^{J_2}_{2,N_2,L_2} \|_{L^2} \\
&\leq \sum_{\big( \frac{L_{\max}}{N_2} \big)^{\frac{1}{2}} < D \leq N_1^{\frac{\alpha}{2}}}
\sum_{\substack{J_1 \sim_D J_2, \\ \theta_1 \sim \theta_2}} \| f^{J_1,\theta_1}_{1,N_1,L_1} * f^{J_2,\theta_2}_{2,N_2,L_2} \|_{L^2_{\tau,\xi,\eta}} \\
&\lesssim \sum_{\big( \frac{L_{\max}}{N_2} \big)^{\frac{1}{2}} < D \leq N_1^{\frac{\alpha}{2}}}
\sum_{\substack{J_1 \sim_D J_2, \\ \theta_1 \sim \theta_2}} \big( \frac{D^2 N_2}{N_1^\alpha} \big)^{\frac{1}{2}} L_{\min}^{\frac{1}{2}} \langle L_{\max} / D \rangle^{\frac{1}{2}} \prod_{i=1}^2 \| f^{J_i,\theta_i}_{i,N_i,L_i} \|_{L^2_{\tau,\xi,\eta}} \\
&\lesssim L_{\min}^{\frac{1}{2}} N_2^{\frac{1}{2}} \langle L_{\max} /  N_1^{\frac{\alpha}{2}} \rangle^{\frac{1}{2}} \prod_{i=1}^2 \| f_{i,N_i,L_i} \|_{L^2_{\tau,\xi,\eta}}.
\end{split}
\end{equation*}

Finally, for $D \in [N_1^{\frac{\alpha}{2}},D^*]$ we carry out a simple almost orthogonal decomposition of $f_{1,N_1,L_1} * f_{2,N_2,L_2}$ into $\xi$-intervals of length $N_2$ by convolution constraint. Then, we can apply Lemma \ref{lem:FirstOrderTransversality} to find
\begin{equation*}
\begin{split}
&\quad \sum_{ D \in [N_1^{\frac{\alpha}{2}},D^*] } \sum_{J_1 \sim_D J_2} \| f^{J_1}_{1,N_1,L_1} * f^{J_2}_{2,N_2,L_2} \|_{L^2} \\
&\lesssim \sum_{ D \in [N_1^{\frac{\alpha}{2}},D^*] } N_2^{\frac{1}{2}} L_{\min}^{\frac{1}{2}} \langle L_{\max} / D \rangle^{\frac{1}{2}} \prod_{i=1}^2 \| f_{i,N_i,L_i} \|_2 \\ &\lesssim N_2^{\frac{1}{2}} \log(D^*) L_{\min}^{\frac{1}{2}} \langle L_{\max} / N_1^{\frac{\alpha}{2}} \rangle^{\frac{1}{2}} \prod_{i=1}^2 \| f_{i,N_i,L_i} \|_2.
\end{split}
\end{equation*}
This completes the proof.
\end{proof}

\section{Trilinear convolution estimates}
\label{section:TrilinearConvolutionEstimates}

The purpose of this section is to obtain trilinear convolution estimates:
\begin{equation}
\label{eq:TrilinearConvolutionEst}
\int (f_{1,N_1,L_1} * f_{2,N_2,L_2}) f_{3,N_3,L_3} d\xi (d\eta)_1 d\tau \leq C(N_i) \prod_{i=1}^3 L_i^{\frac{1}{2}} (1+L_i / N_{i,+}^{\alpha+1})^{\frac{1}{4}} \| f_{i,N_i,L_i} \|_2
\end{equation}
with $\text{supp}(f_{i,N_i,L_i}) \subseteq D_{\alpha,N_i,\leq L_i}$, $\alpha < 2$. This corresponds to product estimates for solutions and differences of solutions, localized in $x$-frequency and in time.
In the following we suppose that $N_1 \sim N_3 \gtrsim N_2$. Corresponding to a time localization adjusted to the high frequencies we suppose that $L_1,L_3 \gtrsim N_{\max,+}^{2-\alpha+\varepsilon}$. The above expression will be summable to the short-time norm following Remark \ref{rem:TimeLocalizationWeighted}.

For the low frequency more care is required. Since we aim to use the low modulation weight, we suppose that $L_2 \gtrsim N_{\max,+}^{2-\alpha+\varepsilon} \wedge N_{\min,+}^{\alpha+1}$.

\smallskip

The analysis splits into several cases:
\begin{enumerate}
\item \emph{Low$\times$Low$\rightarrow$Low}-interaction: $N_{\max} \lesssim 1$. This case is readily estimated by using a Strichartz estimate based on second order transversality.
\item \emph{High$\times$High$\rightarrow$High}-interaction: $N_1 \sim N_2 \sim N_3 \gg 1$. In the resonant case we can apply the improved $L^4$-Strichartz estimates, and so do we for some non-resonant cases. When the resonance function becomes large enough, we can conclude favorable estimates using a bilinear Strichartz estimate.
\item \emph{High$\times$Low$\rightarrow$High}-interaction: $N_1 \sim N_3 \gg N_2$.
\end{enumerate}

\subsection{Overview of the nonlinear interpolation argument}
\label{subsection:OutlineNonlinearInterpolation}
To obtain \eqref{eq:TrilinearConvolutionEst}, we take advantage of the high modulation $L_{\max}$ and the corresponding weight $L_{\max}^{-\frac{1}{2}}$ by estimating the function with high modulation in $L^2$. 

One of the key cases to understand is the \textbf{High}$\times$\textbf{High}$\rightarrow$\textbf{High}-interaction in the resonant case. We shall see that the $L^4$-Strichartz estimates yield estimates in $L^2$ up to the critical $\alpha = 4/3$. This means that this estimate is very favorable. We perturb the argument to still obtain favorable estimates in the non-resonant case $N_1^{\alpha} N_2 \ll L_{\max} \lesssim M^*$ and in the \emph{High$\times$Low$\rightarrow$High}-interaction: $N_1^{1-\gamma} \lesssim N_2 \lesssim N_1$. In case $N_2 \ll N_1^{1-\gamma}$ or $L_{\max} \gg M^*$, the short-time bilinear Strichartz estimates are sufficient.

\medskip

We turn to the estimate of the resonant case in the \emph{High$\times$High$\rightarrow$High}-interaction. By symmetry we can assume that $L_3 \sim L_{\max}$. In the resonant case, note that the resonance relation gives
\begin{equation*}
\frac{| \eta_1 \xi_2 - \eta_2 \xi_1 |^2}{N^3} \lesssim L_{\max} \sim N^{\alpha+1}.
\end{equation*}
Consequently,
\begin{equation}
\label{eq:AngularLocalization}
\big| \frac{\eta_1}{\xi_1} - \frac{\eta_2}{\xi_2} \big| \lesssim N^{\frac{\alpha}{2}}.
\end{equation}
In essence, this yields after an almost orthogonal decomposition a localization
\begin{equation*}
\big| \frac{\eta_i}{\xi_i} - A \big| \lesssim N^{\frac{\alpha}{2}}.
\end{equation*}

This allows us to use $L^4$-Strichartz estimates proved in Proposition \ref{prop:LinearStrichartzEstimate}, which read
\begin{equation*}
\| S_{\alpha}(t) f \|_{L^4_{t,x,y}([0,1]; \R \times \T)} \lesssim_\varepsilon N^{\frac{2-\alpha}{8} + \varepsilon} \| f \|_{L^2}
\end{equation*}
for $\text{supp}(\hat{f}) \subseteq \{ (\xi,\eta) \in \R^2 : |\xi| \sim N, \, \big| \frac{\eta}{\xi} - A \big| \lesssim N^{\frac{\alpha}{2}} \}$. 


By H\"older's inequality and an application of the transfer principle we find
\begin{equation}
\label{eq:L4StrichartzExpo}
\begin{split}
\int (f_{1,N_1,L_1} * f_{2,N_2,L_2}) f_{3,N_3,L_3} d\xi (d \eta)_1 d\tau &\leq \| f_{3,N_3,L_3} \|_{L^2_{\tau,\xi,\eta}} \prod_{i=1}^2 \| \mathcal{F}_{t,x,y}^{-1} [f_{i,N_i,L_i}] \|_{L^4_{t,x,y}} \\
&\lesssim N^{-\frac{\alpha+1}{2}} (N^{\frac{2-\alpha}{8}+\varepsilon})^2 \prod_{i=1}^2 L_i^{\frac{1}{2}} \| f_{i,N_i,L_i} \|_2,
\end{split}
\end{equation}
which indeed gives a short-time nonlinear estimate in $L^2$ for $\alpha > \frac{4}{3}$. 


In the non-resonant case, i.e., $L_{\max} \gg N^{\alpha+1}$ \eqref{eq:AngularLocalization} becomes
\begin{equation*}
\big| \frac{\eta_1}{\xi_1} - \frac{\eta_2}{\xi_2} \big| \lesssim \big( \frac{L_{\max}}{N} \big)^{\frac{1}{2}}.
\end{equation*}
For $L_{\max}$ sufficiently close to $N^{\alpha+1}$, we can follow along the lines of \eqref{eq:L4StrichartzExpo} and apply the Strichartz estimates from Proposition \ref{prop:StrichartzLargeEta} instead of Proposition \ref{prop:LinearStrichartzEstimate}. But note that applying Proposition \ref{prop:StrichartzLargeEta} becomes less and less favorable for $L_{\max}$ becoming larger and larger. We impose a threshold $L_{\max} \ll M^*$, which still gives a favorable estimate.

For $L_{\max} \gtrsim M^*$, we instead use a bilinear Strichartz estimate provided by Lemma \ref{lem:FirstOrderTransversality}, noting that we have
\begin{equation*}
\big| \frac{\eta_1}{\xi_1} - \frac{\eta_2}{\xi_2} \big| \gtrsim N_{\max}^{\frac{\alpha}{2}}.
\end{equation*}
This gives
\begin{equation*}
\begin{split}
&\quad \int (f_{1,N_1,L_1}* f_{2,N_2,L_2}) f_{3,N_3,L_3} d\xi (d\eta)_1 d\tau \\
 &\lesssim (M^*)^{-\frac{1}{2}} (L^{\frac{1}{2}} \| f_{3,N_3,L_3} \|_{L^2_{\tau,\xi,\eta}} \| 1_{D_{N_3,L_3}} (f_{1,N_1,L_1} * f_{2,N_2,L_2}) \|_{L^2_{\tau,\xi,\eta}} \\
&\lesssim (M^*)^{-\frac{1}{2}} L^{\frac{1}{2}} N^{\frac{1}{2}} L_{12,\min}^{\frac{1}{2}} \langle L_{12,\max} / N^{\frac{\alpha}{2}} \rangle^{\frac{1}{2}} \prod_{i=1}^3 \| f_{i,N_i,L_i} \|_2.
\end{split}
\end{equation*}
This outlines the key interpolation argument for the \textbf{High}$\times$\textbf{High}$\rightarrow$\textbf{High}-inter\-action.

\medskip

We turn to the \textbf{High}$\times$\textbf{Low}$\rightarrow$\textbf{High}-interaction. In case there is a frequency significantly lower than the remaining frequencies, say $N_2 \ll N_1 \sim N_3$, we need to distinguish between the high frequency being at high modulation or the low frequency being at high modulation.

First, we consider the case of a high frequency being at high modulation, say $L_3 = L_{\max}$. In the resonant case we have
\begin{equation*}
\big| \frac{\eta_1}{\xi_1} - \frac{\eta_2}{\xi_2} \big| \lesssim N_1^{\frac{\alpha}{2}}.
\end{equation*}
Then, we can estimate \eqref{eq:TrilinearConvolutionEst} by H\"older's inequality and a refined bilinear Strichartz estimate provided by Proposition \ref{prop:BilinearStrichartzA}:
\begin{equation*}
\begin{split}
&\quad \int (f_{1,N_1,L_1} * f_{2,N_2,L_2}) f_{3,N_3,L_3} d\xi (d\eta)_1 d\tau \\
 &\leq \| f_{3,N_3,L_3} \|_{L^2} \| 1_{D_{N,L}} (f_{1,N_1,L_1} * f_{2,N_2,L_2}) \|_{L^2_{\tau,\xi,\eta}} \\
&\lesssim (N_1^{\alpha} N_2)^{-\frac{1}{2}} \log(N_1) L_3^{\frac{1}{2}} N_2^{\frac{1}{2}} L_{12,\min}^{\frac{1}{2}} \langle L_{12,\max} / N_1^{\alpha/2} \rangle^{\frac{1}{2}} \prod_{i=1}^3 \| f_{i,N_i,L_i} \|_{L^2}.
\end{split}
\end{equation*}
By the minimum modulation localization $L_i \gtrsim N_i^{2-\alpha+\varepsilon}$ for $i=1,3$, that is to say our choice of the frequency-dependent time localization, we can continue the above as
\begin{equation*}
\lesssim N_1^{-1-\frac{\varepsilon}{2}} \prod_{i=1}^3 L_i^{\frac{1}{2}} \| f_{i,N_i,L_i} \|_2.
\end{equation*}

The short-time bilinear estimate becomes less favorable for the nonlinear analysis as $N_2 \to N_1$. The key point to show local well-posedness as low as in $L^2$ is to interpolate with $L^4$-estimates when $N_1^{1-\gamma} \lesssim N_2 \lesssim N_1$ for $\gamma$ small enough. In the present implementation we will use $\gamma = \frac{1}{8}$ to simplify some of the resulting expressions. It might well be the case that the local well-posedness result can be refined for a different choice of $\gamma$.

Indeed, we find in the resonant case
\begin{equation*}
\big| \frac{\eta_1}{\xi_1} - \frac{\eta_2}{\xi_2} \big| \lesssim N_1^{\frac{\alpha}{2}},
\end{equation*}
which implies that $\mathcal{F}^{-1}_{t,x,y}[f_{1,N_1,L_1}]$ can
be estimated with Proposition \ref{prop:LinearStrichartzEstimate} after additional almost orthogonal decomposition
\begin{equation*}
\big| \frac{\eta_i}{\xi_i} - A \big| \lesssim N_1^{\frac{\alpha}{2}}.
\end{equation*}
For the estimate of $\mathcal{F}^{-1}_{t,x,y}[f_{2,N_2,L_2}]$ note the following: an application of Proposition \ref{prop:StrichartzLargeEta} is not too lossy since we have for the resonance gain $(N_1^{\alpha} N_2)^{-\frac{1}{2}} \ll N_1^{-1}$ in case $\alpha \to 2$ and $\gamma \to 0$.

In the non-resonant case, we intend to follow along the above lines, where
\begin{equation*}
\big| \frac{\eta_1}{\xi_1} - \frac{\eta_2}{\xi_2} \big| \lesssim \big( \frac{L_{\max}}{N_2} \big)^{\frac{1}{2}}.
\end{equation*}
This requires to estimate $\mathcal{F}^{-1}_{t,x,y}[f_{i,N_i,L_i}]$ for $i=1,2$ with the lossier estimate provided by Proposition \ref{prop:StrichartzLargeEta}. Still the resonance gain $L_{\max}^{-\frac{1}{2}}$ outweighs the loss from the linear Strichartz estimates imposing a threshold $L_{\max} \lesssim M^*$. In case $L_{\max} \gtrsim M^*$ we can instead use a bilinear Strichartz estimate provided by Lemma \ref{lem:FirstOrderTransversality} to find
\begin{equation*}
\begin{split}
\int (f_{1,N_1,L_1} * f_{2,N_2,L_2}) &f_{3,N_3,L_3} 
\leq \| f_{3,N_3,L_3} \|_{L^2} \| 1_{D_{N,L}} (f_{1,N_1,L_1} * f_{2,N_2,L_2}) \|_{L^2_{\tau,\xi,\eta}} \\
&\lesssim (M^*)^{-\frac{1}{2}} L_3^{\frac{1}{2}} N_2^{\frac{1}{2}} L_{12,\min}^{\frac{1}{2}} \langle L_{12,\max} / N_1^{\alpha/2} \rangle^{\frac{1}{2}} \prod_{i=1}^3 \| f_{i,N_i,L_i} \|_2.
\end{split}
\end{equation*}
We will choose $M^* = N_1^{\frac{9}{4}} N_2$ to obtain favorable estimates for $\alpha \in (\alpha_0,2)$.

It remains to shed light on the case $N_2 \ll N_1 \sim N_3$ with the low frequency carrying the high modulation. For the sake of exposition we suppose that $N_2 \gtrsim 1$. The case of very low frequencies requires possibly a different, but simpler argument.

 We would like to use a similar interpolation argument between linear and bilinear Strichartz estimates. A notable change is that 
\begin{equation*}
\big| \frac{\eta_1}{\xi_1} - \frac{\eta_3}{\xi_3} \big| \lesssim \big( \frac{L_{\max}}{N_1} \big)^{\frac{1}{2}}.
\end{equation*}
Consequently, in the resonant case $L_{\max} \sim N_1^{\alpha} N_2$, we always have after almost orthogonal decomposition
\begin{equation*}
\big| \frac{\eta_i}{\xi_i} - A \big| \ll N_1^{\frac{\alpha}{2}}.
\end{equation*}
This enables us to apply the almost sharp linear Strichartz estimates from Proposition \ref{prop:LinearStrichartzEstimate} to estimate
\begin{equation*}
\| \mathcal{F}^{-1}_{t,x,y} [f_{i,N_i,L_i}] \|_{L^4_{t,x,y}} \lesssim_\varepsilon N_i^{\frac{2-\alpha}{8}+\varepsilon} L_i^{\frac{1}{2}} \| f_{i,N_i,L_i} \|_2
\end{equation*}
for $i=1,3$. We obtain taking into account the modulation weight
\begin{equation*}
\begin{split}
&\quad \int (f_{1,N_1,L_1} * f_{2,N_2,L_2}) f_{3,N_3,L_3} d\xi (d\eta)_1 d\tau \\
&\leq \| f_{2,N_2,L_2} \|_{L^2_{\tau,\xi,\eta}} \prod_{i=1,3} \| \mathcal{F}^{-1}_{t,x,y}[f_{i,N_i,L_i}] \|_{L^4_{t,x,y}} \\
&\leq (N_1^{\alpha} N_2)^{-\frac{1}{2}} \big( \frac{N_2}{N_1} \big)^{\alpha+1} (N_1^{\frac{2-\alpha}{8}+\varepsilon})^2 \prod_{i=1,3} L_i^{\frac{1}{2}} \| f_{i,N_i,L_i} \|_2.
\end{split}
\end{equation*}

Consequently, although the transversality between the high frequencies is decreased, which leads to potentially worse bilinear Strichartz estimates, the sharp $L^4$-Strichartz estimates together with the modulation weight on the low frequency come to rescue.

In the non-resonant case, we can choose a different threshold $L_{\max} \ll M^{**}$, $M^{**} \gg M^*$ for which the linear Strichartz estimates still gives us the sharp estimate. With $M^{**}$ being sufficiently high, the remaining case $L_{\max} \gtrsim M^{**}$ can be  estimated via a simple bilinear Strichartz estimate from Lemma  \ref{lem:SecondOrderTransversality}. We turn to the implementation.

\subsection{Proofs of the trilinear convolution estimates}

\textbf{Low$\times$Low$\rightarrow$Low-\-inter\-action:} $N_{\max} \lesssim 1$. In this case we obtain a sufficient estimate by invoking Lemma \ref{lem:SecondOrderTransversality}.

\smallskip

\textbf{High$\times$High$\rightarrow$High-interaction:} $N_1 \sim N_2 \sim N_3 \gg 1$. 

\begin{proposition}
Let $\alpha \in (1,2)$, $N_1 \sim N_2 \sim N_3 \gg 1$, and $L_i \gtrsim N_{\max}^{2-\alpha+\varepsilon}$. Then the estimate \eqref{eq:TrilinearConvolutionEst} holds with
\begin{equation}
\label{eq:TrilinearEstHighHighHigh}
C(N) =
N_1^{-\frac{2}{14}-\frac{9\alpha}{14}+\varepsilon}.
\end{equation}
\end{proposition}

\begin{proof}
\emph{Strongly non-resonant case:} $L_{\max} \sim L_{\text{med}} \gtrsim N_1^{\alpha+1}$. Suppose that $L_{\max} \sim L_3 \sim L_2$ by symmetry. An estimate at scaling-critical regularity is obtained from applying Lemma \ref{lem:SecondOrderTransversality}:
\begin{equation*}
\begin{split}
\int (f_{1,N_1,L_1} * f_{2,N_2,L_2} ) f_{3,N_3,L_3} d\xi (d\eta)_1 d\tau &\lesssim N_1^{-\frac{\alpha+1}{2}} L_3^{\frac{1}{2}} N_1^{\frac{1}{2}} \langle L_2 N_1 \rangle^{\frac{1}{4}} L_1^{\frac{1}{2}} \prod_{i=1}^3  \| f_{i,N_i,L_i} \|_2 \\
&\lesssim N_1^{-\frac{3 \alpha}{4}} \prod_{i=1}^3 L_i^{\frac{1}{2}} \| f_{i,N_i,L_i} \|_2.
\end{split}
\end{equation*}

The estimates established in the \emph{resonant and non-resonant cases} are inferior.

\smallskip

\emph{Resonant case:} $L_{\max} \sim N_1^{\alpha+1}$. Suppose that $L_3 \sim L_{\max}$ by symmetry.
By almost orthogonality we can suppose that 
\begin{equation*}
\big| \frac{\eta_i}{\xi_i} - A \big| \lesssim N_1^{\frac{\alpha}{2}}.
\end{equation*}
In this case two $L^4$-Strichartz estimates from Proposition \ref{prop:LinearStrichartzEstimate} give
\begin{equation}
\label{eq:TrilinearHighHighHighAuxI}
\begin{split}
&\quad \int (f_{1,N_1,L_1} * f_{2,N_2,L_2}) f_{3,N_3,L_3} d\xi (d\eta)_1 d\tau \\
&\lesssim L_3^{\frac{1}{2}} N_1^{-\frac{\alpha+1}{2}} \| f_{3,N_3,L_3} \|_{L^2_{\tau,\xi,\eta}} \prod_{i=1}^2 \| \mathcal{F}^{-1}_{t,x,y}[ f_{i,N_i,L_i} ] \|_{L^4_{t,x,y}} \\
&\lesssim_\varepsilon N_1^{-\frac{\alpha+1}{2}} (N_1^{\frac{2-\alpha}{8}+\varepsilon})^2 \prod_{i=1}^3 L_i^{\frac{1}{2}} (1+L_i / N_i^{\alpha+1})^{\frac{1}{4}} \| f_{i,N_i,L_i} \|_2 \\
&\lesssim_\varepsilon N_1^{-\frac{3 \alpha}{4} + 2 \varepsilon} \prod_{i=1}^3 L_i^{\frac{1}{2}} \| f_{i,N_i,L_i} \|_2.
\end{split}
\end{equation}

\emph{Non-resonant case:} $L_{\max} \sim M \gg N^{\alpha+1}$, $L_{\max} \gg L_{\text{med}}$. Suppose that $L_3 \sim L_{\max}$. We find by the resonance relation:
\begin{equation*}
\big| \frac{\eta_1}{\xi_1} - \frac{\eta_2}{\xi_2} \big| \sim \big( \frac{M}{N_1} \big)^{\frac{1}{2}}.
\end{equation*}
We can carry out an almost orthogonal decomposition to write
\begin{equation*}
\int (f_{1,N_1,L_1} * f_{2,N_2,L_2}) f_{3,N_3,L_3} = \sum_A \int (f^A_{1,N_1,L_1} * f^A_{2,N_2,L_2}) f^A_{3,N_3,L_3}
\end{equation*}
such that
\begin{equation*}
\text{supp}(f_{i,N_i,L_i}^A) \subseteq \{ (\tau,\xi,\eta) : \big| \frac{\eta}{\xi} - A \big| \lesssim \big( \frac{L_{\max}}{N_1} \big)^{\frac{1}{2}} \}, \quad i =1,2,3.
\end{equation*}
Note that the identity for $i=1,2$ implies the third one by the triangle inequality. We have to carry out a further partition:
\begin{equation*}
f_{i,N_i,L_i}^A = \sum_{k_i} f_{i,N_i,L_i}^{A,k_i}
\end{equation*}
such that
\begin{equation*}
\text{supp}(f_{i,N_i,L_i}^{A_i,k_i}) \subseteq \{ (\tau,\xi,\eta) \in \R^3 : \big| \frac{\eta}{\xi} - A \big| \in [k_i N_1^{\frac{\alpha}{2}+1}, (k_i+1) N_1^{\frac{\alpha}{2}+1} ] \}.
\end{equation*}
The maximal number of $k_i$ is bounded by $k^* = \big( \frac{M}{N_1^{\alpha+1}} \big)^{\frac{1}{2}}$. And note that it suffices to impose this condition for one $i \in \{1,2,3\}$. After $k_i$ being fixed, for any $k_j$, $j \neq i$ there are only finitely many $k_m$ with $i \neq m \neq j$ such that
\begin{equation*}
\int (f^{A,k_1}_{1,N_1,L_1} * f^{A,k_2}_{2,N_2,L_2}) f^{A,k_3}_{3,N_3,L_3} \neq 0.
\end{equation*}

Consequently, after applying a Galilean transform for fixed $A$ we have that
\begin{equation*}
|\eta_i'| \in [k_i N_1^{\frac{\alpha}{2}+1}, (k_i+1) N_i^{\frac{\alpha}{2}+1} ]
\end{equation*}
for $k_i \lesssim \big( \frac{M}{N_1^{\alpha+1}} \big)^{\frac{1}{2}}$. (Note that in general $|k_1 - k_2| \gg 1$).

\smallskip

We intend to apply Proposition \ref{prop:StrichartzLargeEta}. This requires a further subdivision of the $\xi$-support into $(M / N_1^{\alpha+1} )^{\frac{1}{2}}$-subintervals: We let
\begin{equation*}
\int (f_{1,N_1,L_1} * f_{2,N_2,L_2}) f_{3,N_3,L_3} = \sum_{A,k_i,\ell_i} (f_{1,N_1,L_1}^{A,k_1,\ell_1} * f^{A,k_2,\ell_2}_{2,N_2,L_2}) f^{A,k_3,\ell_3}_{3,N_3,L_3}
\end{equation*}
with the summation in $A$ being lossless and the summation in $k_i,\ell_i$ incuring a total loss of $M / N_1^{\alpha+1}$ by two applications of the Cauchy-Schwarz inequality. We summarize
\begin{equation*}
\text{supp}(f_{i,N_i,L_i}^{A,k_i,\ell_i}) \subseteq \{ (\tau,\xi,\eta) : \big| \frac{\eta}{\xi} - A \big| \in [k_i N_1^{\frac{\alpha}{2}}, (k_i+1) N_1^{\frac{\alpha}{2}}], \; |\xi| \in I_{\ell_i} \}
\end{equation*}
with $|I_{\ell_i}| \sim N_1 / k^*$.

First note the estimate as a consequence of H\"older's inequality and the $L^4$-Strichartz estimate from Proposition \ref{prop:StrichartzLargeEta}:
\begin{equation*}
\| f^{A,k_1,\ell_1}_{1,N_1,L_1} * f_{2,N_2,L_2}^{A,k_2,\ell_2} \|_{L^2_{\tau,\xi,\eta}} \lesssim_\varepsilon (N_1^{\frac{2-\alpha}{8}+\varepsilon} )^2 (M^{\frac{1}{2}} N_1^{-\frac{\alpha+1}{2}} )^{\frac{1}{6}} \prod_{i=1}^2 L_i^{\frac{1}{2}} \| f_{i,N_i,L_i}^{A,k_i,\ell_i} \|_{L^2_{\tau,\xi,\eta}}.
\end{equation*}

Together with the modulation size of $f_{3,N_3,L_3}$ and carrying out the sums over $k_i$ and $\ell_i$ we find
\begin{equation}
\label{eq:TrilinearHighHighHighAuxII}
\begin{split}
&\quad \sum_{k_i,\ell_i,A} \int (f_{1,N_1,L_1}^{A,k_1,\ell_1} * f^{A,k_2,\ell_2}_{2,N_2,L_2} ) f_{3,N_3,L_3}^{A,k,\ell} d\xi (d\eta_1) d\tau \\
&\lesssim_\varepsilon M^{-\frac{1}{2}} M^{\frac{1}{2}} N_1^{-\frac{\alpha+1}{2}} (N_1^{\frac{2-\alpha}{8}+\varepsilon})^2 M^{\frac{1}{12}} N_1^{-\frac{\alpha+1}{12}} \\
&\quad \quad  \prod_{i=1}^3 L_i^{\frac{1}{2}} (1+L_i / N_i^{\alpha+1})^{\frac{1}{4}} \| f_{i,N_i,L_i} \|_2 \\
&\lesssim_\varepsilon M^{\frac{1}{12}} N_1^{-\frac{1}{12}-\frac{5 \alpha}{6} +2\varepsilon} \prod_{i=1}^3 L_i^{\frac{1}{2}} (1+L_i / N_i^{\alpha+1})^{\frac{1}{4}} \| f_{i,N_i,L_i} \|_2. 
\end{split}
\end{equation}

We obtain an alternative estimate observing that the transversality satisfies
\begin{equation*}
\big| \frac{\eta_1}{\xi_1} - \frac{\eta_2}{\xi_2} \big| \gtrsim \big( \frac{M}{N_1} \big)^{\frac{1}{2}}.
\end{equation*}
In this case applying the bilinear Strichartz estimate provided by Lemma \ref{lem:FirstOrderTransversality} yields
\begin{equation}
\label{eq:TrilinearHighHighHighAuxIII}
\begin{split}
&\quad \int (f_{1,N_1,L_1} * f_{2,N_2,L_2} ) f_{3,N_3,L_3} \\
 &\leq \| f_{3,N_3,L_3} \|_{L^2} \| 1_{D_{N_3,L_3}}( f_{1,N_1,L_1} * f_{2,N_2,L_2}) \|_{L^2} \\
 &\lesssim M^{-\frac{1}{2}} N_1^{\frac{1}{2}} N_1^{-\frac{2-\alpha}{2}} \prod_{i=1}^3 (1+L_i / N_i^{\alpha+1})^{\frac{1}{4}} L_i^{\frac{1}{2}} \|f_{i,N_i,L_i} \|_2 \\
&\lesssim M^{-\frac{1}{2}} N_1^{-\frac{1}{2}+\frac{\alpha}{2}-\frac{\varepsilon}{2}} \prod_{i=1}^3 (1+L_i / N_i^{\alpha+1})^{\frac{1}{4}} L_i^{\frac{1}{2}} \|f_{i,N_i,L_i} \|_2.
\end{split}
\end{equation}
In conclusion, taking the above estimates \eqref{eq:TrilinearHighHighHighAuxI}-\eqref{eq:TrilinearHighHighHighAuxIII}, we find \eqref{eq:TrilinearConvolutionEst} to hold with
\begin{equation*}
C(N) = M^{\frac{1}{12}} N_1^{-\frac{1}{12}-\frac{5 \alpha}{6}+ 2\varepsilon} \vee M^{-\frac{1}{2}} N_1^{-\frac{1}{2}+\frac{\alpha}{2}-\frac{\varepsilon}{2}} \vee N_1^{-\frac{3 \alpha}{4}+2\varepsilon}.
\end{equation*}
The first and second estimate are balanced for $M^{\frac{7}{12}} = N_1^{-\frac{5}{12}+ \frac{8\alpha}{6}}$, which results in
\begin{equation*}
M^{\frac{1}{12}} N_1^{-\frac{3 \alpha}{4}} = M^{-\frac{1}{2}} N_1^{-\frac{1}{2}+\frac{\alpha}{2}} = N_1^{-\frac{2}{14}- \frac{9 \alpha}{14}}.
\end{equation*}
Note that this estimate is at strictly subcritical regularity. The proof is complete.

\end{proof}

\textbf{High$\times$Low$\rightarrow$High-interaction:} Next, we consider the case $N_1 \sim N_3 \gg N_2$. Here we distinguish between $N_2 \lesssim N_1^{1-\gamma}$ and $N_2 \gtrsim N_1^{1-\gamma}$.

\begin{proposition}
\label{prop:TrilinearConvolutionHighLowHigh}
Let $\alpha \in (\frac{7}{4},2)$, $1 \ll N_1 \sim N_3 \gg N_2$, $N_2 \gtrsim N_1^{-2}$, and $L_i \gtrsim N_{\max}^{2-\alpha+\varepsilon}$ for $i=1,3$ and $L_2 \gtrsim N_{\max}^{2-\alpha+\varepsilon} \wedge N^{\alpha+1}_{\min,+}$.

Let $N_2 \lesssim N_1^{1-\gamma}$ for some $0<\gamma<1$. Then the estimate \eqref{eq:TrilinearConvolutionEst} holds with
\begin{equation}
\label{eq:TrilinearEstHighLowHighI}
C_1(N) = N_1^{-1-\frac{\varepsilon}{2}} \vee N_1^{-\frac{3 \alpha}{4}+ \frac{2-\alpha}{4}+ 2 \varepsilon} N_2^{\frac{\alpha}{4}-\frac{1}{2}}.
\end{equation}

For $N_1^{1-\gamma} \lesssim N_2 \lesssim N_1$ we find with $M^* \gg N_1^{\alpha} N_2$ the estimate \eqref{eq:TrilinearConvolutionEst} to hold with:
\begin{equation}
\label{eq:TrilinearEstHighLowHighII}
\begin{split}
C_2(N) &= N_2^{-\frac{\alpha+1}{2}} (M^* / N_2^{\alpha+1})^{\frac{1}{24}} (N_1 N_2)^{\frac{1}{8}+\frac{2-\alpha}{12}+\varepsilon} (M^* / (N_2 N_1^{\alpha}) )^{\frac{1}{24}} \\
&\quad \quad \vee N_2^{\frac{1}{2}} N_1^{-\frac{2-\alpha+\varepsilon}{2}} (M^*)^{-\frac{1}{2}}.
\end{split}
\end{equation}

\end{proposition}

\begin{remark}
\label{remark:ConstantHighLowHigh}
We note that in the relevant range $\alpha \in (7/4,2)$, $N_2 \gtrsim N_1^{-2}$ the first value in \eqref{eq:TrilinearEstHighLowHighI} is dominating such that we have $C_1(N) = N_1^{-1-\frac{\varepsilon}{2}}$.

\smallskip

We shall frequently apply \eqref{eq:TrilinearEstHighLowHighII} with $\gamma = \frac{1}{8}$ and $M^* = N_1^{\frac{9}{4}} N_2$:
The constants become
\begin{equation*}
C_{21}(N) = N_1^{\frac{7}{16}-\frac{\alpha}{6}+\varepsilon} N_2^{-\frac{2 \alpha}{3}-\frac{1}{4} +\varepsilon}, \quad
C_{22}(N) = N_1^{-\frac{2-\alpha+\varepsilon}{2}} N_1^{-\frac{9}{8}}.
\end{equation*}
\end{remark}

\begin{proof}

\emph{Case }$N_2 \lesssim N_1^{1-\gamma}$. We can estimate the case $N_2 \lesssim N_1^{-\alpha}$ by applying Lemma \ref{lem:SecondOrderTransversality} and taking into account the minimum size of modulation:
\begin{equation*}
\begin{split}
&\quad \int (f_{1,N_1,L_1} * f_{2,N_2,L_2} ) f_{3,N_3,L_3} d\xi (d \eta)_1 d \tau \\
&\leq \| f_{3,N_3,L_3} \|_{L^2_{\tau,\xi,\eta}} \| f_{1,N_1,L_1} * f_{2,N_2,L_2} \|_{L^2_{\tau,\xi,\eta}} \\
&\lesssim N_1^{- \frac{2-\alpha + \varepsilon}{2}} L_3^{\frac{1}{2}} N_2^{\frac{1}{2}} L_{12,\min}^{\frac{1}{2}} \langle L_{12,\max} N_2 \rangle^{\frac{1}{4}} \prod_{i=1}^3 \| f_{i,N_i,L_i} \|_{L^2_{\tau,\xi,\eta}} \\
&\lesssim N_2^{0+} N_1^{-1-\varepsilon} \prod_{i=1}^3 L_i^{\frac{1}{2}} \| f_{i,N_i,L_i} \|_2,
\end{split}
\end{equation*}
which is sufficient.

We turn to the main case $N_1^{-\alpha} \lesssim N_2 \lesssim N_1^{1-\gamma}$. We distinguish between the case with the high frequency carrying the high modulation $L_1 = L_{\max} \vee L_3 = L_{\max}$ or the low frequency carrying the high modulation $L_2 = L_{\max}$.

\smallskip

\emph{Case }$L_3 = L_{\max}$. \emph{High frequency carries the high modulation. }

\smallskip

\emph{Resonant case} $L_{\max} \sim N_1^{\alpha} N_2$, $L_{\text{med}} \ll L_{\max}$:
\begin{equation*}
\begin{split}
&\quad \int (f_{1,N_1,L_1} * f_{2,N_2,L_2}) f_{3,N_3,L_3} d\xi (d\eta)_1 d\tau \\
&\lesssim \| f_{3,N_3,L_3} \|_{L^2} \| 1_{D_{N_3,L_3}} (f_{1,N_1,L_1} * f_{2,N_2,L_2} ) \|_{L^2_{\tau,\xi,\eta}}.
\end{split}
\end{equation*}
Then we can use the refined bilinear Strichartz estimate from Proposition \ref{prop:BilinearStrichartzA} as the resonance relation yields
\begin{equation*}
\big| \frac{\eta_1}{\xi_1} - \frac{\eta_2}{\xi_2} \big| \lesssim N_1^{\alpha+1}.
\end{equation*}

This gives
\begin{equation*}
\begin{split}
\| 1_{D_{N_3,L_3}} (f_{1,N_1,L_1} * f_{2,N_2,L_2}) \|_{L^2_{\tau,\xi,\eta}} &\lesssim \log(N_1) N_2^{\frac{1}{2}} L_{12,\min}^{\frac{1}{2}} \\
&\quad \times \langle L_{12,\max} / N_1^{\frac{\alpha}{2}} \rangle^{\frac{1}{2}} \prod_{i=1}^2 \| f_{i,N_i,L_i} \|_2.
\end{split}
\end{equation*}
We obtain
\begin{equation}
\label{eq:HighLowHighAuxEstI}
\int (f_{1,N_1,L_1} * f_{2,N_2,L_2}) f_{3,N_3,L_3} d\xi (d\eta)_1 d\tau \lesssim N_1^{-1-\frac{\varepsilon}{2}} \prod_{i=1}^3 L_i^{\frac{1}{2}} \| f_{i,N_i,L_i} \|_2.
\end{equation}

\emph{Non-resonant case} $L_3 = L_{\max} \gg L_{\text{med}}$, $L_{\max} \gg N_1^\alpha N_2$. Note that we have from the resonance identity
\begin{equation*}
D \sim \big| \frac{\eta_1}{\xi_1} - \frac{\eta_2}{\xi_2} \big| \sim \big( \frac{L_{\max}}{N_2} \big)^{\frac{1}{2}}.
\end{equation*}
With $L_{\max} \gg N_1^\alpha N_2$, we have $D \gtrsim N_1^{\alpha/2}$ and applying Proposition \ref{prop:BilinearStrichartzA} yields
\begin{equation*}
\begin{split}
&\quad \int (f_{1,N_1,L_1} * f_{2,N_2,L_2}) f_{3,N_3,L_3} d\xi (d\eta)_1 d\tau \\
 &\leq \| f_{3,N_3,L_3} \|_{L^2} \| 1_{D_{N_3,L_3}} (f_{1,N_1,L_1} * f_{2,N_2,L_2}) \|_{L^2} \\
&\lesssim L_{\max}^{-\frac{1}{2}} N_2^{\frac{1}{2}} N_1^{-(2-\alpha+\varepsilon)/2} \prod_{i=1}^3 L_i^{\frac{1}{2}} \| f_{i,N_i,L_i} \|_2.
\end{split}
\end{equation*}

By the condition $L_{\max} \gg N_1^{\alpha} N_2$ we find again
\begin{equation*}
\int (f_{1,N_1,L_1} * f_{2,N_2,L_2}) f_{3,N_3,L_3} d\xi (d\eta)_1 d\tau \lesssim N_1^{-1-\frac{\varepsilon}{2}} \prod_{i=1}^3 L_i^{\frac{1}{2}} \| f_{i,N_i,L_i} \|_2.
\end{equation*}

\smallskip

\emph{Case $L_{\max} = L_2$. Low frequency carries the high modulation.} 

In case $N_1^{-\alpha} \ll N_2 \ll 1$ we can argue like in the previous case. We use a bilinear Strichartz estimate provided by Proposition \ref{prop:BilinearStrichartzA} and $L_{\max} \gtrsim N_1^{\alpha} N_2 \gg 1$ to find the estimate
\begin{equation*}
\int (f_{1,N_1,L_1} * f_{2,N_2,L_2} ) f_{3,N_3,L_3} d\xi (d\eta)_1 d\tau \lesssim N_1^{-1-\varepsilon/2} \prod_{i=1}^3 L_i^{\frac{1}{2}} \| f_{i,N_i,L_i} \|_2.
\end{equation*}

\smallskip

We turn to the case $N_2 \gtrsim 1$, where the linear Strichartz estimates can yield improved estimates.

\emph{Resonant case } $L_{\max} = L_2 \sim N_1^\alpha N_2$, $L_{\text{med}} \ll L_{\max}$.

We observe that by the resonance relation we have the bound
\begin{equation*}
\big| \frac{\eta_1}{\xi_1} - \frac{\eta_3}{\xi_3} \big| \lesssim N_1^{\frac{\alpha}{2}-1} N_2
\end{equation*}
and carrying out an almost orthogonal decomposition we can suppose that
\begin{equation*}
\big| \frac{\eta_i}{\xi_i} - A \big| \ll N_1^{\frac{\alpha}{2}}, \quad i =1,3.
\end{equation*}

Consequently, we find by applying two $L^4$-Strichartz estimates  from Proposition \ref{prop:LinearStrichartzEstimate} and taking into account the modulation weight:
\begin{equation}
\label{eq:HighLowHighAuxEstII}
\begin{split}
&\quad \int (f_{1,N_1,L_1} * f_{2,N_2,L_2}) f_{3,N_3,L_3} d\xi (d\eta)_1 d\tau  \\
&\lesssim_\varepsilon (N_1^{\frac{2-\alpha}{8}+\varepsilon})^2 (N_1^\alpha N_2)^{-\frac{1}{2}} (N_1/N_2)^{-\frac{\alpha}{4}} \prod_{i=1}^3 L_i^{\frac{1}{2}} (1+L_i/N_i^{\alpha+1})^{\frac{1}{4}} \| f_{i,N_i,L_i} \|_2 \\
&\lesssim_\varepsilon N_1^{-\frac{3\alpha}{4}+\frac{2-\alpha}{4}+2\varepsilon} N_2^{\frac{\alpha}{4}-\frac{1}{2}} \prod_{i=1}^3 L_i^{\frac{1}{2}} (1+L_i/N_i^{\alpha+1})^{\frac{1}{4}} \| f_{i,N_i,L_i} \|_2.
\end{split}
\end{equation}

\emph{Non-resonant case: $L_2 = L_{\max} \gg N_1^{\alpha} N_2$.} By the resonance relation we find:
\begin{equation*}
\big| \frac{\eta_1}{\xi_1} - \frac{\eta_3}{\xi_3} \big| \sim \big( \frac{L_{\max} N_2}{N_1^2} \big)^{\frac{1}{2}}.
\end{equation*}
For $L_{\max} \lesssim \frac{N_1^{\alpha+2}}{N_2} $, we can use the improved $L^4$-Strichartz estimates from Proposition \ref{prop:LinearStrichartzEstimate}, which gives like above
\begin{equation*}
\begin{split}
&\quad \int (f_{1,N_1,L_1} * f_{2,N_2,L_2}) f_{3,N_3,L_3} d\xi (d\eta)_1 d\tau  \\
&\lesssim_\varepsilon N_1^{-\frac{3\alpha}{4}+\frac{2-\alpha}{4}+\varepsilon} N_2^{\frac{\alpha}{4}-\frac{1}{2}} \prod_{i=1}^3 L_i^{\frac{1}{2}} (1+L_i/N_i^{\alpha+1})^{\frac{1}{4}} \| f_{i,N_i,L_i} \|_2.
\end{split}
\end{equation*}

For $L_{\max} \gtrsim \frac{N_1^{\alpha+2}}{N_2}$ we use bilinear Strichartz estimates based on the transversality bound:
\begin{equation*}
\big| \frac{\eta_1}{\xi_1} - \frac{\eta_3}{\xi_3} \big| \gtrsim N_1^{\frac{\alpha}{2}}.
\end{equation*}

In this case we find from the bilinear Strichartz estimate provided by Proposition \ref{prop:BilinearStrichartzA}:
\begin{equation}
\label{eq:HighLowHighAuxEstIII}
\begin{split}
&\quad \int (f_{1,N_1,L_1} * f_{2,N_2,L_2}) f_{3,N_3,L_3} d\xi (d\eta)_1 d\tau  \\
&\leq \| f_{2,N_2,L_2} \|_{L^2_{\tau,\xi,\eta}} \| 1_{D_{N_2,L_2}} (\tilde{f}_{1,N_1,L_1} * f_{3,N_3,L_3} ) \|_{L^2_{\tau,\xi,\eta}} \\
&\lesssim (N_1^{\alpha+2} / N_2)^{-\frac{1}{2}} (N_1^{\alpha+2}/N_2^{\alpha+2})^{-\frac{1}{4}} N_2^{\frac{1}{2}} N_1^{-(2-\alpha+\varepsilon)/2} \prod_{i=1}^3 L_i^{\frac{1}{2}} (1+L_i/N_i^{\alpha+1})^{\frac{1}{4}} \| f_{i,N_i,L_i} \|_2 \\
&\lesssim N_1^{-2-\frac{\varepsilon}{2}} (N_1 / N_2)^{-\frac{\alpha+2}{4}} \prod_{i=1}^3 L_i^{\frac{1}{2}} (1+L_i / N_i^{\alpha+1})^{\frac{1}{4}} \| f_{i,N_i,L_i} \|_2.
\end{split}
\end{equation}

\emph{Case: $N_1^{1-\gamma} \lesssim N_2 \lesssim N_1$.} 

\smallskip

\emph{Case: } $L_3 = L_{\max} \gg L_{\text{med}}$. \emph{High frequency carries the high modulation.} 

With the low frequency still being quantifiably comparable to the high frequency, we aim to apply two $L^4$-Strichartz estimates. This is supposed to be favorable compared to the bilinear estimate.

\smallskip

\emph{Resonant case: $L_3 = L_{\max} \sim N_1^\alpha N_2$.} By the resonance relation we find
\begin{equation*}
\big| \frac{\eta_1}{\xi_1} - \frac{\eta_2}{\xi_2} \big| \lesssim N_1^{\frac{\alpha}{2}}.
\end{equation*}
We carry out an almost orthogonal decomposition 
\begin{equation*}
\int (f_{1,N_1,L_1} * f_{2,N_2,L_2}) f_{3,N_3,L_3} d\xi (d\eta)_1 d\tau = \sum_A \int (f^A_{1,N_1,L_1} * f^A_{2,N_2,L_2}) f_{3,N_3,L_3} d\xi (d\eta)_1 d\tau
\end{equation*}
such that we have the support conditions
\begin{equation*}
\big|\frac{\eta_i}{\xi_i} - A \big| \lesssim N_1^{\frac{\alpha}{2}}.
\end{equation*}
Consequently, for $f_{1,N_1,L_1}$ we can apply the $L^4$-Strichartz estimate provided by Proposition \ref{prop:LinearStrichartzEstimate}. For $f_{2,N_2,L_2}$ we can apply Corollary \ref{cor:StrichartzEstimateLargeEta} after additional decomposition
\begin{equation*}
\big| \frac{\eta_2}{\xi_2} - A \big| \in [k N_2^{\alpha},(k+1) N_2^{\alpha} ].
\end{equation*}

The maximal value of $k$ is bounded by $(N_1/N_2)^{\frac{\alpha}{2}}$. We need an additional decomposition of the $\xi$-frequencies into intervals of length $N_2 / k_{\max}$. The functions with reduced Fourier support are denoted by $f_{2,N_2,L_2}^{J_1,J_2}$. Then we can apply the $L^4$-Strichartz estimates from Corollary \ref{cor:StrichartzEstimateLargeEta} to find
\begin{equation*}
\| f^{J_1,J_2}_{2,N_2,L_2} \|_{L^4_{t,x,y}} \lesssim_\varepsilon N_2^{\frac{2-\alpha}{8}+\varepsilon} \big( \frac{N_1}{N_2} \big)^{\frac{\alpha}{24}} \| f_{2,N_2,L_2}^{J_1,J_2} \|_{L^2}.
\end{equation*}
The summation over $J_i$ incurs another factor of $(N_1/N_2)^{\frac{\alpha}{2}}$.

This yields
\begin{equation}
\label{eq:HighLowHighAuxEstIV}
\begin{split}
&\quad \int (f_{1,N_1,L_1} * f_{2,N_2,L_2}) f_{3,N_3,L_3} d\xi (d\eta)_1 d\tau \\
&\lesssim (N_1^\alpha N_2)^{-\frac{1}{2}} (N_1 N_2)^{\frac{2-\alpha}{8}+\varepsilon} \big( \frac{N_1}{N_2} \big)^{\frac{13 \alpha}{24}} \prod_{i=1}^3 L_i^{\frac{1}{2}} \| f_{i,N_i,L_i} \|_2.
\end{split}
\end{equation}

\emph{ Non-resonant case: $L_3 = L_{\max} \gg N_1^{\alpha} N_2$, $L_{\text{med}} \ll L_{\max}$.} In this case we obtain for the transversality
\begin{equation*}
\big| \frac{\eta_1}{\xi_1} - \frac{\eta_2}{\xi_2} \big| \sim \big( \frac{L_{\max}}{N_2} \big)^{\frac{1}{2}}.
\end{equation*}

We turn to the estimate in case $L_{\max} \lesssim M^*$ with $M^* \gg N_1^\alpha N_2$. Let $M = L_{\max}$ for brevity. In this case we use again the above argument based on two $L^4$-Strichartz estimates. We carry out the almost orthogonal decomposition:
\begin{equation*}
\int (f_{1,N_1,L_1} * f_{2,N_2,L_2}) f_{3,N_3,L_3} d\xi (d\eta)_1 d\tau = \sum_A \int (f_{1,N_1,L_1}^A * f_{2,N_2,L_2}^A) f^A_{3,N_3,L_3} d\xi (d\eta)_1 d\tau
\end{equation*}
with support conditions for $i=1,2$:
\begin{equation*}
\big| \frac{\eta_i}{\xi_i} - A \big| \sim \big( M / N_2 \big)^{\frac{1}{2}} \lesssim N_1.
\end{equation*}
The procedure to decompose the Fourier support of $f_i$ is similar to the estimate in the non-resonant case of the \emph{High}$\times$\emph{High}$\rightarrow$\emph{High}$-$interaction.

\smallskip

We need to decompose the support of $f_2$ such that
\begin{equation*}
\big| \frac{\eta_2}{\xi_2} - A \big| \in [k_2 N_2^{\frac{\alpha}{2}}, (k_2+1) N_2^{\frac{\alpha}{2}} ]
\end{equation*}
with $k_{\max} =  \big( \frac{M}{N_2^{\alpha+1}} \big)^{\frac{1}{2}}$ and $\xi_2 \in I_2$ with $|I_2| \sim N_2/k_{\max}$. This yields by convolution constraint corresponding decompositions for $f_{1,N_1,L_1}^A$, and we can write
\begin{equation*}
\int (f_{1,N_1,L_1}^A * f_{2,N_2,L_2}^A) f_{3,N_3,L_3}^A = \sum_{k_i,\ell_i} \int (f_{1,N_1,L_1}^{A,k_1,\ell_1} * f_{2,N_2,L_2}^{A,k_2,\ell_2}) f_{3,N_3,L_3}^{A,k,\ell}.
\end{equation*}
After the additional decomposition, we can apply Proposition \ref{prop:StrichartzLargeEta} to $f_{i,N_i,L_i}^{A,k_i,\ell_i}$, $i=1,2$. By construction, the support of $f_{2,N_2,L_2}^{A,k_2,\ell_2}$ is small enough. Secondly, we can suppose that for $f_{1,N_1,L_1}^{A,k_1,\ell_1}$ it holds
\begin{equation*}
\big| \frac{\eta_1}{\xi_1} - A \big| \in [k_1' N_1^{\frac{\alpha}{2}}, (k_1'+1) N_1^{\frac{\alpha}{2}} ]
\end{equation*}
with $k_1' \lesssim \big( \frac{M}{N_2 N_1^{\alpha}} \big)^{\frac{1}{2}}$. Applying Proposition \ref{prop:StrichartzLargeEta} requires a $\xi$-support contained in intervals of length $N_1 / (M/N_2 N_1^{\alpha})^{\frac{1}{2}} \sim \frac{N_1^{\frac{\alpha}{2}+1} N_2^{\frac{1}{2}}}{M^{\frac{1}{2}}}$. With the $\xi$-support of $f_{2,N_2,L_2}^{A,k_2,\ell_2}$ contained in intervals of length $N_2^{\frac{\alpha+1}{2}} \frac{N_2}{M^{\frac{1}{2}}} \lesssim \frac{N_1^{\frac{\alpha}{2}+1} N_2^{\frac{1}{2}}}{M^{\frac{1}{2}}}$, the $\xi$-support of $f_{1,N_1,L_1}^{A,k_1,\ell_1}$ is small enough by convolution constraint.

We obtain the estimate
\begin{equation*}
\begin{split}
\| f_{1,N_1,L_1}^{A,k_1,\ell_1} * f_{2,N_2,L_2}^{A,k_2,\ell_2} \|_{L^2_{\tau,\xi,\eta}} &\lesssim_\varepsilon N_1^{\frac{2-\alpha}{8}+\varepsilon} \big( \frac{M}{N_2^{\alpha+1}} \big)^{\frac{1}{24}} N_2^{\frac{2-\alpha}{8}+\varepsilon} \big( \frac{M}{N_2 N_1^{\alpha}} \big)^{\frac{1}{24}} \\
&\quad \times \prod_{i=1}^2 L_i^{\frac{1}{2}} \| f_{i,N_i,L_i}^{A,k_i,\ell_i} \|_{L^2}.
\end{split}
\end{equation*}

Next, we find from taking into account the modulation size and carrying out the summations over $k_i,\ell_i$ for $M \lesssim M^*$:
\begin{equation}
\label{eq:HighLowHighAuxEstV}
\begin{split}
&\quad \sum_{k_i,\ell_i} \int ( f_{1,N_1,L_1}^{A,k_1,\ell_1} * f_{2,N_2,L_2}^{A,k_2,\ell_2}) f^{A,k,\ell}_{3,N_3,L_3} \lesssim_\varepsilon M^{-\frac{1}{2}} \big( \frac{M}{N_2^{\alpha+1}} \big)^{\frac{1}{2}} \\
&\quad \quad \times N_1^{\frac{2-\alpha}{8}+\varepsilon} \big( \frac{M}{N_2^{\alpha+1}} \big)^{\frac{1}{24}} N_2^{\frac{2-\alpha}{8}+\varepsilon} \big( \frac{M}{N_2 N_1^{\alpha}} \big)^{\frac{1}{24}} \prod_{i=1}^3 L_i^{\frac{1}{2}} \| f_{i,N_i,L_i} \|_{L^2} \\
&\lesssim_\varepsilon N_2^{-\frac{\alpha+1}{2}} (M^* / N_2^{\alpha+1})^{\frac{1}{24}} (N_1 N_2)^{\frac{2-\alpha}{8}+\varepsilon} (M^* / (N_2 N_1^{\alpha}) )^{\frac{1}{24}} \prod_{i=1}^3 L_i^{\frac{1}{2}} \| f_{i,N_i,L_i} \|_2.
\end{split}
\end{equation}

It remains to consider the case $L_3 \sim L_{\max} \gtrsim M^*$, $L_{\text{med}} \ll L_{\max}$. We obtain from the resonance relation
\begin{equation*}
\big| \frac{\eta_1}{\xi_1} - \frac{\eta_2}{\xi_2} \big| \gtrsim N_1^{\frac{\alpha}{2}}
\end{equation*}
and an application of Lemma \ref{lem:FirstOrderTransversality}:
\begin{equation*}
\int (f_{1,N_1,L_1} * f_{2,N_2,L_2}) f_{3,N_3,L_3} d\xi (d\eta)_1 d\tau \lesssim_\varepsilon N_2^{\frac{1}{2}} N_1^{-\frac{2-\alpha+\varepsilon}{2}} (M^*)^{-\frac{1}{2}} \prod_{i=1}^3 L_i^{\frac{1}{2}} \| f_{i,N_i,L_i} \|_2.
\end{equation*}

In case the low frequency carries the high modulation we use the same arguments like in the case $N_2 \lesssim N_1^{1-\gamma}$.

We summarize the different estimates: From \eqref{eq:HighLowHighAuxEstI}, \eqref{eq:HighLowHighAuxEstII}, \eqref{eq:HighLowHighAuxEstIII} in case $N_2 \lesssim N_1^{1-\gamma}$ we obtain the conditions
\begin{equation*}
C(N) = N_1^{-1-\frac{\varepsilon}{2}} \vee N_1^{-\frac{3\alpha}{4}+\frac{2-\alpha}{4}+\varepsilon} N_2^{\frac{\alpha}{4}-\frac{1}{2}} \vee N_1^{-2} (N_1 / N_2)^{-\frac{\alpha+2}{4}}.
\end{equation*}
In the relevant range $\alpha \in (\frac{3}{2},2)$ the first value is largest, and we have $C_1(N) = N_1^{-1-\frac{\varepsilon}{2}}$.

\smallskip

We find further conditions from \eqref{eq:HighLowHighAuxEstIV}, \eqref{eq:HighLowHighAuxEstV} in case $N_1^{1-\gamma} \lesssim N_2 \lesssim N_1$:
\begin{equation*}
\begin{split}
C(N) &= (N_1^\alpha N_2)^{-\frac{1}{2}} (N_1 N_2)^{\frac{2-\alpha}{8}+\varepsilon} \big( \frac{N_1}{N_2} \big)^{\frac{13 \alpha}{24}} \\
&\quad \vee N_2^{-\frac{\alpha+1}{2}} (M^* / N_2^{\alpha+1})^{\frac{1}{24}} (N_1 N_2)^{\frac{2-\alpha}{8}+\varepsilon} (M^* / (N_2 N_1^{\alpha}) )^{\frac{1}{24}} \\
&\quad \quad \vee N_2^{\frac{1}{2}} N_1^{-\frac{2-\alpha+\varepsilon}{2}} (M^*)^{-\frac{1}{2}}.
\end{split}
\end{equation*}
Note that the second value coincides with the first in case of $M^* = N_1^\alpha N_2$. Since $M^* \gg N_1^\alpha N_2$, the second value is always larger than the first and we simplify the above to:
\begin{equation*}
\begin{split}
C(N) &= N_2^{-\frac{\alpha+1}{2}} (M^* / N_2^{\alpha+1})^{\frac{1}{24}} (N_1 N_2)^{\frac{2-\alpha}{8}+\varepsilon} (M^* / (N_2 N_1^{\alpha}) )^{\frac{1}{24}} \\
&\quad \quad \vee N_2^{\frac{1}{2}} N_1^{-\frac{2-\alpha+\varepsilon}{2}} (M^*)^{-\frac{1}{2}}.
\end{split}
\end{equation*}

The proof is complete.
\end{proof}

\section{Short-time bilinear estimates}
\label{section:ShorttimeBilinearEstimates}
The purpose of this section is to show short-time bilinear estimates. We show the following:
\begin{proposition}
\label{prop:ShorttimeBilinearEstimate}
For $s \geq 0$ the following estimate holds:
\begin{equation}
\label{eq:ShorttimeEstimateSolutions}
\| \partial_x (uv) \|_{\mathcal{N}^s(T)} \lesssim T^\delta (\| u \|_{F^s(T)} \| v \|_{F^0(T)} + \| u \|_{F^0(T)} \| v \|_{F^s(T)})
\end{equation}
provided that $\alpha > 1.94$.

\smallskip

Let $s'= 8(\alpha-2)$. The following estimate holds
\begin{equation}
\label{eq:ShorttimeEstimateDifferences}
\| \partial_x(uv) \|_{\mathcal{\bar{N}}^{s'}(T)} \lesssim T^\delta \| u \|_{F^0(T)} \| v \|_{F^{s'}(T)}
\end{equation}
provided that $\alpha \in (1.985,2)$.
\end{proposition}
\begin{proof}
We begin with the proof of \eqref{eq:ShorttimeEstimateSolutions}. 

\smallskip

First, we reduce the estimates to convolution estimates for functions localized in frequency and modulation. By choosing suitable extensions of $u$ and $v$ (which are redenoted by $u$, $v$ for convenience) and Littlewood-Paley decomposition, it suffices to show estimates
\begin{equation}
\label{eq:DyadicReducedShorttime}
\| P_N \partial_x (P_{N_1} u P_{N_2} v) \|_{\mathcal{N}_N} \lesssim C(N,N_1,N_2) \| P_{N_1} u \|_{F_{N_1}^b} \| P_{N_2} v \|_{F_{N_2}^b}
\end{equation}
for some $b<\frac{1}{2}$ and $N,N_1,N_2$ satisfying the Littlewood-Paley dichotomy. We use $b<\frac{1}{2}$ and Lemma \ref{lem:TradingModulationRegularity} to create the gain $T^\delta$ in \eqref{eq:ShorttimeEstimateSolutions} and similarly in \eqref{eq:ShorttimeEstimateDifferences}.

We plug in the definition of $\mathcal{N}_N$ which leads us to
\begin{equation}
\label{eq:DefNn}
\begin{split}
\| P_N \partial_x (P_{N_1} u P_{N_2} v) \|_{\mathcal{N}_N} 
&= \sup_{t_k \in \R} \| (\tau - \omega_\alpha(\xi,\eta) + i N_+^{2-\alpha+\varepsilon})^{-1} \\
&\quad \times \mathcal{F}_{t,x,y} [\eta_0(N^{2-\alpha+\varepsilon}_+(t-t_k)) P_N \partial_x (P_{N_1} u P_{N_2} v) ] \|_{\bar{X}_N}.
\end{split}
\end{equation}

\smallskip

\textbf{High$\times$High$\rightarrow$High-interaction:} First, we suppose that $N \sim N_1 \sim N_2$. We continue \eqref{eq:DefNn} by considering
\begin{equation*}
\sum_{l \in \Z} \gamma^2(t-l) \equiv 1
\end{equation*}
and let
\begin{equation*}
\begin{split}
&\quad \sum_{t_l \in \Z} \eta_0(N_+^{2-\alpha+\varepsilon}(t-t_k)) \gamma^2(2^{10} N_+^{2-\alpha+\varepsilon} t-t_l) P_N \partial_x (P_{N_1} u P_{N_2} v) \\
&= \sum_{t_l \in \Z} \eta_0(N_+^{2-\alpha+\varepsilon}(t-t_k)) P_N \partial_x (\gamma(2^{10} N_+^{2-\alpha+\varepsilon}(t-t_l)) P_{N_1} u \\
&\quad \quad \times \gamma(2^{10} N_+^{2-\alpha+\varepsilon}(t-t_l)) P_{N_2} v).
\end{split}
\end{equation*}
Note that the set
\begin{equation*}
L = \{ l \in \Z : \eta_0(N_+^{2-\alpha+\varepsilon}(t-t_k)) \gamma^2(2^{10} N_+^{2-\alpha+\varepsilon} t-t_l) \neq 0 \}
\end{equation*}
is finite. We obtain uniform estimates for $l \in L$, for which reason it is suppressed in notation in the following. Let
\begin{equation*}
\begin{split}
f_{1,N_1} &= \mathcal{F}_{t,x,y} [ \eta_0(N_+^{2-\alpha+\varepsilon}(t-t_k)) \gamma(2^{10} N_+^{2-\alpha+\varepsilon} t - t_l) P_{N_1} u], \\
f_{2,N_2} &= \mathcal{F}_{t,x,y} [  \gamma(2^{10} N_+^{2-\alpha+\varepsilon} t - t_l) P_{N_2} u].
\end{split}
\end{equation*}
Next, we break up the modulation
\begin{equation*}
f_{i,N_i,L_i} = \begin{cases}
1_{D_{N_i,\leq L_i}} f_{i,N_i}, \quad &L_i = N_1^{2-\alpha+\varepsilon}, \\
1_{D_{N_i,L_i}} f_{i,N_i}, \quad &L_i > N_1^{2-\alpha+\varepsilon}.
\end{cases}
\end{equation*}

We need to establish estimates 
\begin{equation}
\label{eq:HighHighHighShorttimeEstI}
\begin{split}
&\quad N L^{-\frac{1}{2}} (1+L/N^{\alpha+1})^{\frac{1}{4}} \| 1_{D_{N,L}} (f_{1,N_1,L_1} * f_{2,N_2,L_2}) \|_{L^2_{\tau,\xi,\eta}} \\
&\lesssim (\min(1,N))^{0+} L_{12,\max}^{0-} \prod_{i=1}^2 L_i^{\frac{1}{2}} \| f_{i,N_i,L_i} \|_{L^2}
\end{split}
\end{equation}
for $L_i \geq N_+^{2-\alpha+\varepsilon}$. Then the claim follows from dyadic summation and properties of the function spaces; see Remarks \ref{rem:TimeLocalizationUnweighted} and \ref{rem:TimeLocalizationWeighted}. The slack in modulation suffices to infer \eqref{eq:DyadicReducedShorttime} for $b<\frac{1}{2}$.

The case of low frequencies $N \lesssim 1$ is readily estimated by Lemma \ref{lem:SecondOrderTransversality}, which yields
\begin{equation}
\label{eq:ShorttimeHighHighHighAux}
N L^{-\frac{1}{2}} \| 1_{D_{N,L}} (f_{1,N_1,L_1} * f_{2,N_2,L_2}) \|_{L^2_{\tau,\xi,\eta}} \lesssim N^{\frac{3}{2}} 
L_{12,\min}^{\frac{1}{2}} L_{12,\max}^{\frac{1}{4}} \prod_{i=1}^2 \| f_{i,N_i,L_i} \|_2.
\end{equation}

We turn to the case of high frequencies $N \gg 1$. 
Moreover, interpolation with the above estimate, shows that it suffices to show for $N \gg 1$:
\begin{equation*}
\begin{split}
&\quad N L^{-\frac{1}{2}} (1+L/N^{\alpha+1})^{\frac{1}{4}} \| 1_{D_{N,L}} (f_{1,N_1,L_1} * f_{2,N_2,L_2}) \|_{L^2_{\tau,\xi,\eta}} \\
&\lesssim N^{0-} \prod_{i=1}^2 L_i^{\frac{1}{2}} (1+L_i/N_i^{\alpha+1})^{\frac{1}{4}} \| f_{i,N_i,L_i} \|_2.
\end{split}
\end{equation*}

\smallskip

\emph{Case $L \lesssim N^3$:} We apply duality and use the trilinear convolution estimate \eqref{eq:TrilinearEstHighHighHigh}. Taking into account the factor $(1+L/N^{\alpha+1})^{\frac{1}{4}}$ we find
\begin{equation*}
\begin{split}
&\quad N L^{-\frac{1}{2}} (1+L/N^{\alpha+1})^{\frac{1}{4}} \| 1_{D_{N,L}} (f_{1,N_1,L_1} * f_{2,N_2,L_2}) \|_{L^2_{\tau,\xi,\eta}} \\
&\lesssim N N^{\frac{2-\alpha}{2}} N^{-\frac{2}{14}-\frac{9 \alpha}{14}+\varepsilon} \prod_{i=1}^2 L_i^{\frac{1}{2}} (1+L_i/N^{\alpha+1})^{\frac{1}{4}} \| f_{i,N_i,L_i} \|_2 \\
&\lesssim N^{\frac{26}{14}-\frac{16 \alpha}{14}+\varepsilon} \prod_{i=1}^2 L_i^{\frac{1}{2}} (1+L_i/N^{\alpha+1})^{\frac{1}{4}} \| f_{i,N_i,L_i} \|_2.
\end{split}
\end{equation*}
This is favorable for $\alpha > \frac{13}{8}$, but note that we proved the trilinear convolution estimate for  
\begin{equation}
\label{eq:NonlinearBilinearThresholdI}
\alpha > \frac{7}{4}.
\end{equation}

\emph{Case $L \geq N^3$:} The strongly non-resonant case $L_{\max} \sim L_{\text{med}}$ can be estimated by the bilinear Strichartz estimate due to Lemma \ref{lem:SecondOrderTransversality}. Hence, we can suppose that $L_{\max} \gg L_{\text{med}}$. From the resonance relation follows
\begin{equation*}
\big| \frac{\eta_1}{\xi_1} - \frac{\eta_2}{\xi_2} \big| \gtrsim N_1
\end{equation*}
for $(\xi_i,\eta_i) \in \text{supp}_{\xi,\eta}(f_i)$, $i=1,2$.

 We apply a bilinear Strichartz estimate provided by Proposition \ref{prop:BilinearStrichartzA}:
\begin{equation*}
\begin{split}
&\quad N L^{-\frac{1}{2}} (1+L/N^{\alpha+1})^{\frac{1}{4}} \| 1_{D_{N,L}} (f_{1,N_1,L_1} * f_{2,N_2,L_2}) \|_{L^2_{\tau,\xi,\eta}} \\
&\lesssim N N^{-\frac{3}{4}} N^{-\frac{\alpha+1}{4}} N^{\frac{1}{2}} N^{-\frac{2-\alpha+\varepsilon}{2}} \prod_{i=1}^2 L_i^{\frac{1}{2}} \| f_{i,N_i,L_i} \|_2 \\
&\lesssim N^{\frac{\alpha}{4}-\frac{1}{2}-\frac{\varepsilon}{2}} \prod_{i=1}^2 L_i^{\frac{1}{2}} \| f_{i,N_i,L_i} \|_2.
\end{split}
\end{equation*}
This is favorable for $\alpha < 2$.

\medskip

\textbf{High$\times$Low$\rightarrow$High-interaction:} In this case we suppose that $N \sim N_1 \gg N_2$. The time localization for the function with low frequency is chosen depending on the relative size with respect to the high frequency. 

In the case of very low frequency $N_{2,+}^{\alpha+1} \lesssim N_1^{2-\alpha+\varepsilon}$, we choose the time localization for the low frequency as $N_{2,+}^{\alpha+1}$. We write
\begin{equation*}
\begin{split}
&\quad \sum_{t_l \in \Z} \eta_0(N_+^{2-\alpha+\varepsilon}(t-t_k)) \eta_0(2^{10} N_+^{2-\alpha+\varepsilon} t-t_l) P_N \partial_x (P_{N_1} u P_{N_2} v) \\
&= \sum_{t_l \in \Z} P_N \partial_x ([\eta_0(N_+^{2-\alpha+\varepsilon}(t-t_k)) \eta_0(2^{10} N_+^{2-\alpha+\varepsilon} t-t_l) P_{N_1} u] \\
&\quad \times \eta_0(N_{2,+}^{\alpha+1} t- t_l) P_{N_2} v).
\end{split}
\end{equation*}
Again there are only finitely many $l$ contributing in the above. Let
\begin{equation*}
\begin{split}
f_{1,N_1} &= \mathcal{F}_{t,x,y} [ \eta_0(N_+^{2-\alpha+\varepsilon}(t-t_k)) \gamma(2^{10} N_+^{2-\alpha+\varepsilon} t - t_l) P_{N_1} u], \\
f_{2,N_2} &= \mathcal{F}_{t,x,y} [  \eta_0(N_{2,+}^{\alpha+1} t- t_l) P_{N_2} v].
\end{split}
\end{equation*}
Next, we break up the modulation
\begin{equation*}
f_{1,N_1,L_1} = \begin{cases}
1_{D_{N_1,\leq L_1}} f_{1,N_1}, \quad &L_1 = N_1^{2-\alpha+\varepsilon}, \\
1_{D_{N_1,L_1}} f_{1,N_1}, \quad &L_1 > N_1^{2-\alpha+\varepsilon},
\end{cases}
\end{equation*}
and for the second function:
\begin{equation*}
f_{2,N_2,L_2} = \begin{cases}
1_{D_{N_2,\leq L_2}} f_{2,N_2}, \quad &L_2 = N_{2,+}^{\alpha+1}, \\
1_{D_{N_2,L_2}} f_{2,N_2}, \quad &L_2 > N_{2,+}^{\alpha+1}.
\end{cases}
\end{equation*}

In the main case $N_{2,+}^{\alpha+1} \gtrsim N_1^{2-\alpha+\varepsilon}$, we use the reductions from the \textbf{High}$\times$\textbf{High}$\rightarrow$ \textbf{High-\-interaction}, that is we localize as well $P_{N_1} u$ as $P_{N_2} v$ on times $N_1^{-(2-\alpha+\varepsilon)}$ and break up the modulation into dyadic pieces of minimum size $L_i \gtrsim N_1^{2-\alpha+\varepsilon}$.

After these reductions it suffices to establish the frequency and modulation localized estimate:
\begin{equation}
\label{eq:HighLowHighReductionI}
\begin{split}
&\quad N L^{-\frac{1}{2}} (1+L/N^{\alpha+1})^{\frac{1}{4}} \| 1_{D_{N,L}} (f_{1,N_1,L_1} * f_{2,N_2,L_2}) \|_{L^2_{\tau,\xi,\eta}} \\
&\lesssim N_1^{0-} N_{2,+}^{0+} L_{\max}^{0-} \prod_{i=1}^2 L_i^{\frac{1}{2}} (1+L_i/N_{i,+}^{\alpha+1})^{\frac{1}{4}} \| f_{i,N_i,L_i} \|_2
\end{split}
\end{equation}
for $L,L_i \gtrsim N_{1,+}^{2-\alpha+\varepsilon}$ and $L_2 \gtrsim N_1^{2-\alpha+\varepsilon} \wedge N_{2,+}^{\alpha+1}$, which can be supposed due to the frequency-dependent time localization.

Firstly, note that the case $N_1 \ll 1$ readily follows from Lemma \ref{lem:SecondOrderTransversality}. So, we suppose in the following that $N_1 \gtrsim 1$. If $N_2 \ll N_1^{-2}$, another application of Lemma \ref{lem:SecondOrderTransversality} implies
\begin{equation*}
\begin{split}
&\quad N (1+L/N^{\alpha+1})^{\frac{1}{4}} L^{-\frac{1}{2}} \| 1_{D_{N,L}}( f_{1,N_1,L_1} * f_{2,N_2,L_2}) \|_{L^2_{\tau,\xi,\eta}} \\
 &\lesssim N N^{-\frac{2-\alpha}{2}} L_{\max}^{0-} N_2^{\frac{1}{2}} \prod_{i=1}^2 L_i^{\frac{1}{2}} (1+L_i/N_i^{\alpha+1})^{\frac{1}{4}} \| f_{i,N_i,L_i} \|_2 \\
&\lesssim N_2^{0+} L_{\max}^{0-} \prod_{i=1}^2 L_i^{\frac{1}{2}}
(1+L_i/N_i^{\alpha+1})^{\frac{1}{4}} \| f_{i,N_i,L_i} \|_2.
\end{split}
\end{equation*}
This is a favorable estimate.

In summary, we can always suppose that $N_2 \gtrsim N_1^{-2}$ and \eqref{eq:HighLowHighReductionI} reduces to
\begin{equation*}
N L^{-\frac{1}{2}} \| 1_{D_{N,L}} (f_{1,N_1,L_1} * f_{2,N_2,L_2}) \|_{L^2_{\tau,\xi,\eta}} \lesssim N_1^{0-}  L_{\max}^{0-} \prod_{i=1}^2 L_i^{\frac{1}{2}} (1+L_i/N_{i,+}^{\alpha+1})^{\frac{1}{4}} \| f_{i,N_i,L_i} \|_2
\end{equation*}
for $N_1 \gtrsim 1$, $N_2 \gtrsim N_1^{-2}$, and $L, L_1 \gtrsim N_1^{2-\alpha+\varepsilon}$ and $L_2 \gtrsim N_1^{2-\alpha+\varepsilon} \wedge N_{2,+}^{\alpha+1}$. Since we can always interpolate with the estimate provided by Lemma \ref{lem:SecondOrderTransversality}, it suffices to show the above without an $L_{\max}^{0-}$ factor.

\medskip

\emph{Case $L \lesssim N^{\alpha+1}$:} After invoking duality, we can apply the estimates from Proposition \ref{prop:TrilinearConvolutionHighLowHigh} with $\gamma = \frac{1}{8}$. For $N_2 \lesssim N_1^{\frac{7}{8}}$ we find from \eqref{eq:TrilinearEstHighLowHighI}:
\begin{equation*}
N_1 L^{-\frac{1}{2}} \| 1_{D_{N,L}} (f_{1,N_1,L_1} * f_{2,N_2,L_2}) \|_{L^2_{\tau,\xi,\eta}} \lesssim N_1^{-\frac{\varepsilon}{2}} \prod_{i=1}^2 L_i^{\frac{1}{2}} (1+L_i/N_{i,+}^{\alpha+1})^{\frac{1}{4}} \| f_{i,N_i,L_i} \|_2.
\end{equation*}
This is acceptable for $\alpha < 2$.

For $N_1^{\frac{7}{8}} \lesssim N_2 \lesssim N_1$ we find from \eqref{eq:TrilinearEstHighLowHighII} with $M^* = N_1^{\frac{9}{4}} N_2$ following Remark \ref{remark:ConstantHighLowHigh}:
\begin{equation*}
N_1 C_{21}(N)= N_1 N_1^{\frac{7}{16}-\frac{\alpha}{6}+\varepsilon} N_2^{-\frac{2 \alpha}{3}-\frac{1}{4}+\varepsilon} \lesssim N_1^{1+\frac{7}{16}-\frac{7}{32}-\frac{3 \alpha}{4}+2\varepsilon}.
\end{equation*}
This is favorable for $\alpha > \frac{13}{8}$, but imposes no new condition on \eqref{eq:NonlinearBilinearThresholdI}.

From the second constant we find the condition
\begin{equation*}
N_1 C_{22}(N) \lesssim N_1^{-\frac{2-\alpha+\varepsilon}{2}} N_1^{-\frac{1}{8}}.
\end{equation*}
This is favorable for $\alpha < 2$.

\medskip

\emph{Case $L \gtrsim N^{\alpha+1}$:} If $L_{12,\max} \gtrsim L$, we can obtain a favorable estimate by invoking Lemma \ref{lem:SecondOrderTransversality}.

So we suppose that $L=L_{\max}$. We apply a bilinear Strichartz estimate from Proposition \ref{prop:BilinearStrichartzA} to find
\begin{equation*}
\begin{split}
&\quad N L^{-\frac{1}{2}} (1+L/N_+^{\alpha+1})^{\frac{1}{4}} \| 1_{D_{N,L}} (f_{1,N_1,L_1} * f_{2,N_2,L_2}) \|_{L^2_{\tau,\xi,\eta}} \\ &\lesssim N_1^{1-\frac{\alpha+1}{2}+} L_{\max}^{0-} \log(L_{\max})  N_2^{\frac{1}{2}} N_1^{-\frac{2-\alpha+\varepsilon}{2}} \prod_{i=1}^2 L_i^{\frac{1}{2}} (1+L_i / N_{i,+}^{\alpha+1})^{\frac{1}{4}} \| f_{i,N_i,L_i} \|_2.
\end{split}
\end{equation*}
This is acceptable for any $\alpha < 2$.

\smallskip

\textbf{High$\times$High$\rightarrow$Low-interaction:} In this case, we need to add time localization in the $\mathcal{N}_N$-norm to estimate the functions in the short-time norms. We write
\begin{equation*}
\begin{split}
&\quad \sum_{t_l \in \Z} \eta_0(N_+^{2-\alpha+\varepsilon}(t-t_k)) \gamma^2(2^{10} N_1^{2-\alpha+\varepsilon} t-t_l) P_N \partial_x (P_{N_1} u P_{N_2} v) \\
&= \sum_{t_l \in \Z} P_N \partial_x ([\eta_0(N_+^{2-\alpha+\varepsilon}(t-t_k)) \gamma(2^{10} N_1^{2-\alpha+\varepsilon} t-t_l) P_{N_1} u] \\
&\quad \times \gamma(2^{10} N_1^{2-\alpha+\varepsilon} t-t_l) P_{N_2} v).
\end{split}
\end{equation*}
Note that in this case for the set of relevant times
\begin{equation*}
L = \{ l : \eta_0(N_+^{2-\alpha+\varepsilon}(t-t_k)) \gamma^2(2^{10} N_1^{2-\alpha+\varepsilon} t-t_l) \neq 0 \}
\end{equation*}
we have the bound $|L| \lesssim (N_1 / N_+)^{2-\alpha+\varepsilon}$.

Taking additionally the modulation weight and derivative loss into account, we need to prove estimates for
\begin{equation}
\label{eq:HighHighLowExpression}
N (N_1/N_+)^{2-\alpha+\varepsilon} L^{-\frac{1}{2}} (1+L/N_+^{\alpha+1})^{\frac{1}{4}} \| 1_{D_{N,L}} (f_{1,N_1,L_1} * f_{2,N_2,L_2}) \|_{L^2_{\tau,\xi,\eta}}
\end{equation}
with $L \gtrsim N_+^{2-\alpha+\varepsilon}$, $L_i \gtrsim N_i^{2-\alpha+\varepsilon}$.

\emph{Case $L \ll N_1^\alpha N$:} 
Suppose that $N \lesssim N_1^{\frac{7}{8}}$. By assumption on $L$ and resonance relation, we have $L_i \gtrsim N_1^\alpha N$ for some $i \in \{1,2\}$. Suppose that $L_2 \gtrsim N_1^\alpha N$ by symmetry. Then we can apply the refined bilinear Strichartz estimate from Proposition \ref{prop:BilinearStrichartzA} after invoking duality on the dual function and $f_{1,N_1,L_1}$ to find
\begin{equation}
\label{eq:HighHighLowNonlinearEstI}
\eqref{eq:HighHighLowExpression} \lesssim N (N_1/N_+)^{2-\alpha+\varepsilon} (N_1^\alpha N)^{-\frac{1}{2}} N_1^{-\frac{2-\alpha+\varepsilon}{2}} (N_1 / N)^{\frac{\alpha}{4}} \prod_{i=1}^2 L_i^{\frac{1}{2}} \| f_{i,N_i,L_i} \|_2.
\end{equation}
For $N \geq 1$ we proceed 
\begin{equation*}
\eqref{eq:HighHighLowNonlinearEstI} \lesssim N^{\frac{1}{2}-\frac{\alpha}{4}-(2-\alpha+\varepsilon)} N_1^{1- \frac{3 \alpha}{4} + \frac{\varepsilon}{2}} \prod_{i=1}^2 L_i^{\frac{1}{2}} \| f_{i,N_i,L_i} \|_2.
\end{equation*}
This is favorable for $\alpha \in (1,2)$.
For $N \leq 1$ the above becomes
\begin{equation*}
\eqref{eq:HighHighLowNonlinearEstI} \lesssim N^{\frac{1}{2}-\frac{\alpha}{4}} N_1^{1- \frac{3 \alpha}{4} + \frac{\varepsilon}{2}} \prod_{i=1}^2 L_i^{\frac{1}{2}} \| f_{i,N_i,L_i} \|_2.
\end{equation*}
This is likewise favorable for $\alpha \in (1,2)$.

\smallskip

Suppose that $N_1^{\frac{7}{8}} \lesssim N \lesssim N_1$. We apply \eqref{eq:TrilinearEstHighLowHighI} to find the following estimate, in case the first constant is largest. We use the simplification from Remark \ref{remark:ConstantHighLowHigh}, which incurs an additional factor of $(1+L/N^{\alpha+1})^{\frac{1}{2}} \lesssim (N_1 / N)^{\frac{\alpha}{2}}$:
\begin{equation}
\label{eq:HighHighLowNonlinearEstII}
\begin{split}
\eqref{eq:HighHighLowExpression} &\lesssim N_1^{\frac{7}{16}-\frac{\alpha}{6}+\varepsilon} N^{1-\frac{2 \alpha}{3} - \frac{1}{4} + \varepsilon} \big( \frac{N_1}{N} \big)^{2-\alpha+\varepsilon} \big( \frac{N_1}{N} \big)^{\frac{\alpha}{2}} \prod_{i=1}^2 L_i^{\frac{1}{2}} (1+L_i/N_i^{\alpha+1})^{\frac{1}{4}} \| f_{i,N_i,L_i} \|_2 \\
&\lesssim N_1^{\frac{43}{32}-\alpha \big( \frac{1}{6}+\frac{7}{12}+\frac{1}{16}) } \prod_{i=1}^2 L_i^{\frac{1}{2}} (1+L_i/N_i^{\alpha+1})^{\frac{1}{4}} \| f_{i,N_i,L_i} \|_2.
\end{split}
\end{equation}

\normalsize
This is favorable for 
\begin{equation}
\label{eq:NonlinearBilinearThresholdIII}
\alpha > 1.654.
\end{equation}

We consider the second constant in \eqref{eq:TrilinearEstHighLowHighI}, which yields the estimate
\begin{equation}
\label{eq:HighHighLowNonlinearEstIII}
\begin{split}
\eqref{eq:HighHighLowExpression} &\lesssim N (N_1 / N_+)^{2-\alpha+\varepsilon} N_1^{-\frac{2-\alpha+\varepsilon}{2}} N_1^{-\frac{9}{8}} (N_1 / N)^{\frac{\alpha}{2}} \prod_{i=1}^2 L_i^{\frac{1}{2}} (1+L_i / N_i^{\alpha+1})^{\frac{1}{4}} \| f_{i,N_i,L_i} \|_2 \\
&\lesssim N_1^{-\frac{3(2-\alpha+\varepsilon)}{8}} N_1^{\frac{\alpha}{16}-\frac{1}{8}} \prod_{i=1}^2 L_i^{\frac{1}{2}} (1+L_i / N_i^{\alpha+1})^{\frac{1}{4}} \| f_{i,N_i,L_i} \|_2.
\end{split}
\end{equation}
This is favorable for $\alpha < 2$.

For $L \gtrsim N_1^\alpha N$ and $L_{\max} \lesssim N_1^{\alpha+2} / N$ we can apply two $L^4$-Strichartz estimates favorably. In this case the resonance relation implies for $(\xi,\eta) \in \text{supp}_{\xi,\eta}(f_{i,N_i,L_i})$:
\begin{equation*}
\frac{|\eta_1 \xi_2 - \eta_2 \xi_1|^2}{N_1^2 N} \lesssim L_{\max} \Rightarrow \big| \frac{\eta_1}{\xi_1} - \frac{\eta_2}{\xi_2} \big| \lesssim \big( \frac{L_{\max} N}{N_1^2} \big)^{\frac{1}{2}} \lesssim N_1^{\frac{\alpha}{2}}.
\end{equation*}
Applying Proposition \ref{prop:LinearStrichartzEstimate} we find for $N \geq 1$
\begin{equation}
\label{eq:HighHighLowNonlinearEstIV}
\begin{split}
\eqref{eq:HighHighLowExpression} &\lesssim N (N_1/ N_+)^{2-\alpha+\varepsilon} L^{-\frac{1}{4}} / N^{\frac{\alpha+1}{4}} (N_1^{\frac{2-\alpha}{8}+\varepsilon})^2 \prod_{i=1}^2 L_i^{\frac{1}{2}} \| f_{i,N_i,L_i} \|_2 \\
&\lesssim N_1^{\frac{2-\alpha}{4}+2 \varepsilon - \frac{\alpha}{4} + 2-\alpha + \varepsilon} N^{1-(2-\alpha+\varepsilon)-\frac{1}{4}-\frac{\alpha+1}{4}} \prod_{i=1}^2 L_i^{\frac{1}{2}} \| f_{i,N_i,L_i} \|_2 \\
&\lesssim N_1^{\frac{5}{2}-\frac{3\alpha}{2}+3\varepsilon} N^{-\frac{3}{2}+\frac{3\alpha}{4}+\varepsilon} \prod_{i=1}^2 L_i^{\frac{1}{2}} \| f_{i,N_i,L_i} \|_2.
\end{split}
\end{equation}
This is favorable for $\alpha > 5/3$.
For $N \leq 1$ we find following the above lines
\begin{equation*}
\begin{split}
\eqref{eq:HighHighLowExpression} &\lesssim N (N_1/ N_+)^{2-\alpha+\varepsilon} L^{-\frac{1}{4}} / N^{\frac{\alpha+1}{4}} (N_1^{\frac{2-\alpha}{8}+\varepsilon})^2 \prod_{i=1}^2 L_i^{\frac{1}{2}} \| f_{i,N_i,L_i} \|_2 \\
&\lesssim N^{\frac{1}{2}-\frac{\alpha}{4}} N_1^{\frac{5}{2}-\frac{3\alpha}{2}+3\varepsilon} \prod_{i=1}^2 L_i^{\frac{1}{2}} \| f_{i,N_i,L_i} \|_2,
\end{split}
\end{equation*}
which is likewise favorable for $\alpha > 5/3$.

Finally, we turn to $L_{\max} \gtrsim N_1^{\alpha+2}/N$. We consider only $L_{\max} = L$ as the other cases can be estimated more favorably using a bilinear Strichartz estimate. We apply Lemma \ref{lem:SecondOrderTransversality} to find
\begin{equation}
\label{eq:HighHighLowNonlinearEstV}
\eqref{eq:HighHighLowExpression} \lesssim \big( \frac{N_1^{\alpha+2}}{N} \big)^{-\frac{1}{4}} N^{-\frac{\alpha+1}{4}} \big( \frac{N_1}{N_+} \big)^{2-\alpha+\varepsilon} N^{\frac{3}{2}} N_1^{\frac{1}{4}} N_1^{-\frac{2-\alpha}{4}} \prod_{i=1}^2 L_i^{\frac{1}{2}} \| f_{i,N_i,L_i} \|_2.
\end{equation}
This gives for $N \geq 1$
\begin{equation*}
\eqref{eq:HighHighLowNonlinearEstV} \lesssim N_1^{-\alpha+\frac{5}{4}+\varepsilon} N^{-1+\frac{3 \alpha}{4}+\varepsilon} \prod_{i=1}^2 L_i^{\frac{1}{2}} \| f_{i,N_i,L_i} \|_2.
\end{equation*}
This is favorable for $\alpha > \frac{5}{4}$. For $N \leq 1$ we have
\begin{equation*}
\eqref{eq:HighHighLowNonlinearEstV} \lesssim N_1^{-\alpha+\frac{5}{4}+\varepsilon} N^{1-\frac{ \alpha}{4}+\varepsilon} \prod_{i=1}^2 L_i^{\frac{1}{2}} \| f_{i,N_i,L_i} \|_2,
\end{equation*}
which is likewise favorable for $\alpha > \frac{5}{4}$.

\smallskip

We turn to the estimate for differences of solutions. For the \textbf{High$\times$High$\rightarrow$High}- and \textbf{High$\times$Low$\rightarrow$High-}interaction nothing changes as this amounts to a shift in regularity.

We need to revisit the \textbf{High$\times$High$\rightarrow$Low}-interaction. Taking into account the additional input regularity, we need to obtain estimates for
\begin{equation}
\label{eq:HighHighLowExpressionII}
N_+^{-8(2-\alpha)} N (N_1 / N_+)^{2-\alpha+\varepsilon} L^{-\frac{1}{2}} (1+L/N_+^{\alpha+1})^{\frac{1}{4}} \| 1_{D_{N,L}} (f_{1,N_1,L_1} * f_{2,N_2,L_2}) \|_{L^2_{\tau,\xi,\eta}}.
\end{equation}
We use the same arguments to estimate \eqref{eq:HighHighLowExpressionII} as previously \eqref{eq:HighHighLowExpression}.
The estimate carried out in \eqref{eq:HighHighLowNonlinearEstI} becomes
\begin{equation*}
\eqref{eq:HighHighLowExpressionII} \lesssim N^{-\frac{3}{2}+\frac{3 \alpha}{4} -8 (2-\alpha)} N_1^{1-\frac{3\alpha}{4}+8(2-\alpha)} N_1^{-8(2-\alpha)} \prod_{i=1}^2 L_i^{\frac{1}{2}} \| f_{i,N_i,L_i} \|_2.
\end{equation*}
This is favorable for $\alpha > 68/67$ and causes no further constraint on $\alpha$.

Next, we turn to \eqref{eq:HighHighLowNonlinearEstII} becomes taking into account the additional regularity factors:
\begin{equation*}
\begin{split}
\eqref{eq:HighHighLowExpressionII} &\lesssim N_1^{\frac{5}{16}+
\frac{1}{6}-\frac{\alpha}{8}+\varepsilon} N^{-\frac{5}{24}-\frac{5 \alpha}{8} + 1 + \varepsilon - 8(2-\alpha)} \\
&\quad \big( \frac{N_1}{N} \big)^{2-\alpha+\varepsilon} \big( \frac{N_1}{N} \big)^{\frac{\alpha}{2}} \prod_{i=1}^2 L_i^{\frac{1}{2}} (1+L_i / N_i^{\alpha+1})^{\frac{1}{4}} \| f_{i,N_i,L_i} \|_2.
\end{split}
\end{equation*}
This is acceptable for 
\begin{equation}
\label{eq:NonlinearBilinearThresholdIV}
\alpha > 1.973.
\end{equation}

The estimate \eqref{eq:HighHighLowNonlinearEstIII} becomes
\begin{equation*}
\begin{split}
\eqref{eq:HighHighLowNonlinearEstIII} &\lesssim N (N_1 / N_+)^{2-\alpha + \varepsilon} N_1^{-\frac{2-\alpha+\varepsilon}{2}} N_1^{-\frac{9}{8}} (N_1 / N)^{\frac{\alpha}{2}} \prod_{i=1}^2 L_i^{\frac{1}{2}} \| f_{i,N_i,L_i} \|_2 \\
&\lesssim N^{-1+\frac{\alpha}{2}-8(2-\alpha)-\varepsilon} N_1^{\frac{2-\alpha+\varepsilon}{2}- \frac{9}{8} + \frac{\alpha}{2} + 8(2-\alpha)} N_1^{-8(2-\alpha)} \\
&\quad \quad \prod_{i=1}^2 L_i^{\frac{1}{2}} (1+L_i / N_i^{\alpha+1})^{\frac{1}{4}} \| f_{i,N_i,L_i} \|_2.
\end{split}
\end{equation*}
This is acceptable for $\alpha > 2-\frac{1}{64}$, which is the case for
\begin{equation}
\label{eq:NonlinearBilinearThresholdV}
\alpha > 1.985.
\end{equation}

We note that \eqref{eq:HighHighLowNonlinearEstIV} becomes
\begin{equation*}
\eqref{eq:HighHighLowExpressionII} \lesssim N_1^{\frac{5}{2}-\frac{3 \alpha}{4}} N^{-\frac{3}{2}+\frac{3 \alpha}{4} - 8(2-\alpha) + \varepsilon} N_1^{-8(2-\alpha)} N_1^{8(2-\alpha)} \prod_{i=1}^2 L_i^{\frac{1}{2}} \| f_{i,N_i,L_i} \|_2.
\end{equation*}
This is acceptable for $\alpha \in (2-\frac{1}{19},2)$, which causes no further constraint.

Finally, we compute that \eqref{eq:HighHighLowNonlinearEstV} becomes
\begin{equation*}
\eqref{eq:HighHighLowExpressionII} \lesssim N_1^{-\alpha+\frac{5}{4}+\varepsilon} N^{-1+\frac{3 \alpha}{4}+\varepsilon} N^{-8(2-\alpha)} N_1^{-8(2-\alpha)} N_1^{8(2-\alpha)} \prod_{i=1}^2 L_i^{\frac{1}{2}} \| f_{i,N_i,L_i} \|_2.
\end{equation*}
This is favorable for $\alpha \in (68/35,2)$.

Taking \eqref{eq:NonlinearBilinearThresholdI} and \eqref{eq:NonlinearBilinearThresholdIII} together we find that \eqref{eq:ShorttimeEstimateSolutions} holds for $\alpha > 1.94$.
Taking \eqref{eq:NonlinearBilinearThresholdIV}, \eqref{eq:NonlinearBilinearThresholdV} together we find that \eqref{eq:ShorttimeEstimateDifferences} holds for $\alpha > 1.985$.

\end{proof}

\section{Energy estimates}
\label{section:ShorttimeEnergyEstimates}
In this section we propagate the short-time energy norms of solutions and their differences. In addition to \eqref{eq:GeneralizedKPII} we need to consider the regularized equation when we construct solutions.
To this end let $\varphi \in C^\infty_c(\R)$ be a radially decreasing function with 
\begin{equation*}
\varphi(x) \equiv 1 \text{ for } |x| \leq 1 \text{ and } \varphi \equiv 0 \text{ on } B(0,2)^c.
\end{equation*}
Then we set $\varphi_M(\xi,\eta) = \varphi(|(\xi,\eta)|/M) (1-\varphi(M|\xi|))$ and define the corresponding Fourier multipliers as
\begin{equation*}
 \widehat{(\tilde{P}_M f)} (\xi,\eta) = \varphi_M(\xi,\eta) \hat{f}(\xi,\eta).
\end{equation*}
Note that the first factor in the definition of $\varphi_M$ accounts for cutting off the high total frequencies $|(\xi,\eta)|$, the second factor accounts for cutting off low $\xi$-frequencies. In this frequency range the generator of the linear evolution $L_\alpha = - \partial_x D_x^{\alpha} + \partial_x^{-1} \partial_y^2$ is bounded (with bounds depending on $M$).

The Fourier-truncated equation reads
\begin{equation}
\label{eq:FourierTruncatedEquation}
\left\{ \begin{array}{cl}
\partial_t u - \partial_x D_x^{\alpha} u + \partial_x^{-1} \partial_y^2 u &= \tilde{P}_M (\partial_x (\tilde{P}_M u)^2), \quad (t,x,y) \in \R \times \R \times \T, \\
u(0) &= \phi \in E^s.
\end{array} \right.
\end{equation}
We formally recover \eqref{eq:GeneralizedKPII} by letting $M = \infty$.

 The estimates for strong solutions $u \in F^s(T)$ to \eqref{eq:FourierTruncatedEquation} read:
\begin{equation}
\label{eq:ShorttimeEnergyEstimateSolution}
\| u \|^2_{E^s(T)} \lesssim \| u_0 \|^2_{H^{s,0}} + T^\delta \| u \|^2_{F^s(T)} \| u \|_{F^0(T)}.
\end{equation}
In the above estimate we take advantage of the conservative derivative nonlinearity and real-valuedness to place the derivative on the lowest frequency after Littlewood-Paley decomposition.

\smallskip

Secondly, we show estimates for differences of solutions $v = u_1-u_2$ with $u_i$ solving \eqref{eq:FourierTruncatedEquation} at negative Sobolev regularity $s' = -8(2-\alpha)$:
\begin{equation}
\label{eq:ShorttimeEnergyEstimateDifferences}
\| v \|^2_{E^{s'}(T)} \lesssim \| v(0) \|^2_{H^{s',0}} + T^{\delta} \| v \|^2_{F^{s'}(T)} (\| u_1 \|_{F^s(T)} + \| u_2 \|_{F^s(T)})
\end{equation}
for $s \geq 0$.

\subsection{Energy estimates for solutions}

For solutions we show the following:
\begin{proposition}
\label{prop:ShorttimeEnergyEstimate}
Let $T \in (0,1]$, $s \geq 0$, and $M \in 2^{\N_0} \cup \{ \infty \}$. Let $u$ be a strong solution to \eqref{eq:FourierTruncatedEquation} with $u \in F^s(T)$. Then the estimate \eqref{eq:ShorttimeEnergyEstimateSolution} holds for $\alpha \in (1.89,2)$.
\end{proposition}

\begin{proof}
The following reductions are standard and we shall be brief. We have by the definition of the function spaces:
\begin{equation*}
\| P_{\leq 8} u \|_{E^s(T)} \lesssim \| u_0 \|_{H^{s,0}}.
\end{equation*}
For high frequencies we find for $0<t\leq T$ by the fundamental theorem of calculus:
\begin{equation*}
\| P_N u(t) \|_{L^2}^2 = \| P_N u(0) \|^2_{L^2} + 2 \int_0^t \int_{\R \times \T} P_N \tilde{P}_M u \partial_x P_N(\tilde{P}_M u \cdot \tilde{P}_M u) dx ds.
\end{equation*}
In the following we focus on estimating $\| P_N u(t) \|_{L^2}^2$ for $t> 0 $; the estimate for negative times is obtained by symmetric arguments.

Following Remark \ref{rem:BoundednessNonlinearity} we have for $u \in F^s(T)$ that $\| P_N u\|^2_2 \in C^1(0,T)$. Indeed,
\begin{equation*}
\begin{split}
&\quad \| P_N u(t) \|^2_{L^2} = \| S_\alpha(t) P_N u_0 + P_N \int_0^t S_\alpha(t-s) \partial_x \tilde{P}_M ( \tilde{P}_M u(s) \cdot \tilde{P}_M u(s) ) ds \|_2^2 \\
&= \| P_N u_0 \|^2_{L^2} + 2 \langle P_N u_0, \int_0^t S_\alpha(-s) \partial_x \tilde{P}_M (\tilde{P}_M u(s) \tilde{P}_M u(s)) \rangle_{L^2_x} \\
&\quad + \langle P_N \int_0^t S_\alpha(-s) \partial_x \tilde{P}_M ( \tilde{P}_M u \cdot \tilde{P}_M u), P_N \int_0^t S_\alpha(-s) \partial_x \tilde{P}_M ( \tilde{P}_M u \cdot \tilde{P}_M u) \rangle_{L^2_x}.
\end{split}
\end{equation*}
Moreover,
\begin{equation*}
\partial_t \int_0^t S_{\alpha}(-s) \partial_x \tilde{P}_M (\tilde{P}_M u(s) \tilde{P}_M u(s)) ds = \partial_x \tilde{P}_M (\tilde{P}_M u(t))^2 \in L^2.
\end{equation*}
This justifies the application of the fundamental theorem. 
 
We aim to show estimates independent of $M$. For the sake of brevity we abuse notation and redenote $\tilde{P}_M u \to u$.
To show favorable estimates, we carry out a paraproduct decomposition:
\begin{equation}
\label{eq:ParaproductDecomposition}
P_N(u^2) = 2 P_N (u P_{\ll N} u) + P_N (P_{\gtrsim N} u P_{\gtrsim N} u).
\end{equation}

We turn to the contribution of the first term: After decomposition
\begin{equation*}
P_N(u P_{\ll N} u) = \sum_{K \ll N} P_N(u P_K u) 
\end{equation*}
we turn to the estimate of
\begin{equation}
\label{eq:LowHighEnergyEstSol}
\sum_{K \ll N} \int_0^t \int_{\R \times \T} P_N u \partial_x P_N (u P_{K} u) dx ds.
\end{equation}
By a standard commutator argument, we can arrange the derivative on the low frequency. To this end, we write
\begin{equation*}
P_N (u P_{K} u) = P_N u P_K u + [ P_N (u P_{K} u) - P_N u P_K u].
\end{equation*}
For the first term, integration by parts is straight-forward:
\begin{equation*}
\int_0^t \int_{\R \times \T} P_N u \partial_x (P_N u P_K u) dx ds = \frac{1}{2} \int_0^t \int_{\R \times \T} (P_N u)^2 \partial_x P_K u \, dx ds.
\end{equation*}

For the second term, we can expand the multiplier in Fourier space to see that
\begin{equation*}
 [ P_N (u P_{K} u) - P_N u P_K u] \sim \frac{K}{N} \tilde{P}_N u P_K u.
\end{equation*}
Here $\tilde{P}_N$ denotes a mild enlargement of $P_N$.

In both cases, the estimate of \eqref{eq:LowHighEnergyEstSol} is reduced to
\begin{equation*}
\eqref{eq:LowHighEnergyEstSol} \lesssim \sum_{K \ll N} K \int_0^t \int_{\R \times \T} P_N u \tilde{P}_N u P_K u dx ds.
\end{equation*}
To estimate the above expression in short-time function spaces, we need to localize time according to the highest frequency occuring in the expression. However, for the low frequency we shall only localize time up to $T'= K_{+}^{-(\alpha+1)} \vee N^{-(2-\alpha+\varepsilon)}$. We discuss the contribution $T' = N^{-(2-\alpha+\varepsilon)}$ first.

We obtain a smooth decomposition\footnote{Up to the endpoints of the interval, which strictly speaking require a different estimate. Since there are only $\leq 4$ terms of this kind, which are easier to estimate, we omit the details and refer to \cite[p.~291]{IonescuKenigTataru2008}.}:
\begin{equation*}
1_{[0,t]} (s) = \sum_{k \in \Z} \gamma^3 (N_1^{2-\alpha+\varepsilon}(s-s_k)) 1_{[0,t]}.
\end{equation*}
For the bulk of the terms, whose number is comparable to $\sim t N^{2-\alpha+\varepsilon}$, we have
\begin{equation*}
\gamma^3 (N_1^{2-\alpha+\varepsilon}(s-s_k)) 1_{[0,t]} = \gamma^3 (N_1^{2-\alpha+\varepsilon}(s-s_k)).
\end{equation*}
We have to estimate
\begin{equation*}
\begin{split}
&\quad K \int_{\R \times (\R \times \T)} \gamma^3(N_1^{2-\alpha+\varepsilon}(s-s_k)) P_{N_1} u P_{N_3} u P_K u dx ds \\
&= K \int_{\R \times (\R \times \T)} (\gamma(N_1^{2-\alpha+\varepsilon}(s-s_k)) P_{N_1} u) (\gamma(N_1^{2-\alpha+\varepsilon}(s-s_k)) P_{N_3} u) \\
&\quad \quad \times (\gamma(N_1^{2-\alpha+\varepsilon}(s-s_k)) P_K u) dx ds
\end{split}
\end{equation*}
with $N_1 \sim N_3 \gg K$. To harmonize with notations from the previous section, redenote $K \to N_2$. We let
\begin{equation*}
\begin{split}
f_{1,N_1} &= \mathcal{F}_{t,x,y} [\gamma(N_1^{2-\alpha+\varepsilon}(s-s_k)) P_{N_1} u], \quad f_{3,N_3} = \mathcal{F}_{t,x,y}[\gamma(N_1^{2-\alpha+\varepsilon}(s-s_k)) P_{N_3} u], \\
f_{2,N_2} &= \mathcal{F}_{t,x,y} [\gamma(N_1^{2-\alpha+\varepsilon}(s-s_k)) P_K u].
\end{split}
\end{equation*}

Finally, we break the support of $f_{i,N_i}$ into dyadic regions corresponding to modulation localization:
\begin{equation}
\label{eq:ModulationLocalizationEnergyEstimateI}
f_{i,N_i} = \sum_{L_i \geq N_1^{2-\alpha+\varepsilon}} f_{i,N_i,L_i}
\end{equation}
with
\begin{equation}
\label{eq:ModulationLocalizationEnergyEstimateII}
\text{supp}(f_{i,N_i,L_i}) \subseteq
\begin{cases}
 \{ (\xi,\eta,\tau) : |\xi| \sim N_i, \; |\tau - \omega_\alpha(\xi,\eta) | \sim L_i \}, \quad  L_i > N_1^{2-\alpha+\varepsilon}, \\
 \{ (\xi,\eta,\tau) : |\xi| \sim N_i, \; |\tau - \omega_\alpha(\xi,\eta) | \leq L_i \}, \quad L_i = N_1^{2-\alpha+\varepsilon}.
 \end{cases}
\end{equation}
By Plancherel's theorem we find
\begin{equation*}
\begin{split}
&\quad \int_{\R \times (\R \times \T)} (\gamma(N_1^{2-\alpha+\varepsilon}(s-s_k)) P_{N_1} u) (\gamma(N_1^{2-\alpha+\varepsilon}(s-s_k)) P_{N_3} u) \\
&\quad \quad \times (\gamma(N_1^{2-\alpha+\varepsilon}(s-s_k)) P_K u) dx ds \\
&= \sum_{L_i \geq N_1^{2-\alpha+\varepsilon}} \int (f_{1,N_1,L_1} * f_{2,N_2,L_2}) f_{3,N_3,L_3} d\xi (d\eta)_1 d\tau.
\end{split}
\end{equation*}

For the contribution of $T' = K_+^{-(\alpha+1)}$ we let
\begin{equation*}
1_{[0,t]}(s) = \sum_{k \in \Z} \bar{\gamma}^2(N_1^{2-\alpha+\varepsilon}(s-s_k)) 1_{[0,t]}(s).
\end{equation*}
We have to estimate
\begin{equation*}
\begin{split}
&\quad K \int_{\R \times \R \times \T} (\bar{\gamma}(N_1^{2-\alpha+\varepsilon}(s-s_k)) P_{N_1} u) (\bar{\gamma}(N_1^{2-\alpha+\varepsilon}(s-s_k)) P_{N_3} u) P_K u \\
&= K \int_{\R \times \R \times \T} (\bar{\gamma}(N_1^{2-\alpha+\varepsilon}(s-s_k)) P_{N_1} u) (\bar{\gamma}(N_1^{2-\alpha+\varepsilon}(s-s_k)) P_{N_3} u) \\
&\quad \times (\eta_0(K^{\alpha+1}(s-s_k)) P_K u dx dy ds.
\end{split}
\end{equation*}
In this case we let
\begin{equation*}
\begin{split}
f_{1,N_1} &= \mathcal{F}_{t,x,y} [\bar{\gamma}(N_1^{2-\alpha+\varepsilon}(s-s_k)) P_{N_1} u], \quad f_{3,N_3} = \mathcal{F}_{t,x,y}[\bar{\gamma}(N_1^{2-\alpha+\varepsilon}(s-s_k)) P_{N_3} u], \\
f_{2,N_2} &= \mathcal{F}_{t,x,y} [\eta_0(K_+^{\alpha+1}(s-s_k)) P_K u].
\end{split}
\end{equation*}
The modulation localization for $i=1,3$ is like in \eqref{eq:ModulationLocalizationEnergyEstimateI} and \eqref{eq:ModulationLocalizationEnergyEstimateII}. For $i=2$ we let
\begin{equation*}
f_{2,N_2} = \sum_{L_2 \geq K_+^{\alpha+1}} f_{2,N_2,L_2}
\end{equation*}
with
\begin{equation*}
\text{supp}(f_{2,N_2,L_2}) \subseteq
\begin{cases}
 \{ (\xi,\eta,\tau) : |\xi| \sim N_2, \; |\tau - \omega_\alpha(\xi,\eta) | \sim L_2 \}, \quad  L_2 > K_+^{\alpha+1}, \\
 \{ (\xi,\eta,\tau) : |\xi| \sim N_2, \; |\tau - \omega_\alpha(\xi,\eta) | \leq L_2 \}, \quad L_2 = K_+^{\alpha+1}.
 \end{cases}
\end{equation*}

At this point both cases fit under a common umbrella and the trilinear estimates from Section \ref{section:TrilinearConvolutionEstimates} become applicable. In the present case we apply Proposition \ref{prop:TrilinearConvolutionHighLowHigh} with $\gamma=\frac{1}{8}$. For $N_2 \lesssim N_1^{\frac{7}{8}}$ we obtain the estimate \eqref{eq:TrilinearEstHighLowHighI} taking into account the derivative loss (factor $N_2$) and the time localization (factor $N_1^{2-\alpha+\varepsilon}$). Note that the in- and output regularity are the same, for which reason we presently omit keeping track.

\begin{equation}
\label{eq:HighLowHighEnergyEstAuxI}
\begin{split}
&\quad N_2 N_1^{2-\alpha+\varepsilon} \int (f_{1,N_1,L_1} * f_{2,N_2,L_2}) f_{3,N_3,L_3} d\xi (d\eta)_1 d\tau \\
&\lesssim N_2 N_1^{2-\alpha+\varepsilon} C_1(N_1) \prod_{i=1}^3 L_i^{\frac{1}{2}} (1+L_i/N_{i,+}^{\alpha+1})^{\frac{1}{4}} \| f_{i,N_i,L_i} \|_2.
\end{split}
\end{equation}
We check that for $C_1(N_1) = N_1^{-1+\frac{\varepsilon}{2}}$ we find a favorable estimate for $\alpha > \frac{15}{8}$.


Next, we turn to the case $N_1^{\frac{7}{8}} \lesssim N_2 \lesssim N_1$, in which case we have with $C_{21}$ defined in \eqref{eq:TrilinearEstHighLowHighII} following Remark \ref{remark:ConstantHighLowHigh} with $M^* = N_1^{\frac{9}{4}} N_2$:
\begin{equation}
\label{eq:HighLowHighEnergyEstAuxII}
\begin{split}
&\quad N_2 N_1^{2-\alpha+\varepsilon} \int (f_{1,N_1,L_1} * f_{2,N_2,L_2}) f_{3,N_3,L_3} d\xi (d\eta)_1 d\tau \\
&\lesssim N_2 N_1^{2-\alpha+\varepsilon} N_1^{\frac{5}{16}+\frac{1}{6}-\frac{\alpha}{8}+\varepsilon} N_2^{-\frac{5}{24}-\frac{5 \alpha}{8}+\varepsilon} \prod_{i=1}^3 L_i^{\frac{1}{2}} (1+L_i/N_{i,+}^{\alpha+1})^{\frac{1}{4}} \| f_{i,N_i,L_i} \|_2 \\
&\lesssim N_1^{\frac{7 \cdot 19}{196}-\frac{35 \alpha}{64}+\frac{7 \varepsilon}{8}} N_1^{2-\alpha+\varepsilon + \frac{5}{16}+\frac{1}{6}-\frac{\alpha}{8}+\varepsilon} \prod_{i=1}^3 L_i^{\frac{1}{2}} (1+L_i/N_{i,+}^{\alpha+1})^{\frac{1}{4}} \| f_{i,N_i,L_i} \|_2.
\end{split}
\end{equation}
This is acceptable for $\alpha > 1.89$.

\smallskip

We obtain with $C_{22}$ defined in \eqref{eq:TrilinearEstHighLowHighII} following Remark \ref{remark:ConstantHighLowHigh}:
\begin{equation}
\label{eq:HighLowHighEnergyEstAuxIII}
\begin{split}
&\quad N_2 N_1^{2-\alpha+\varepsilon} N_1^{-\frac{2-\alpha+\varepsilon}{2}} N_1^{-\frac{9}{8}} \prod_{i=1}^3 L_i^{\frac{1}{2}} (1+L_i / N_{i,+}^{\alpha+1})^{\frac{1}{4}} \| f_{i,N_i,L_i} \|_2 \\
&\lesssim N_1^{-\frac{1}{8}+\frac{2-\alpha+\varepsilon}{2}} \prod_{i=1}^3 L_i^{\frac{1}{2}} (1+L_i / N_{i,+}^{\alpha+1})^{\frac{1}{4}} \| f_{i,N_i,L_i} \|_2.
\end{split}
\end{equation}
This is acceptable for $\alpha > 1.75$.


This finishes the estimate of the contribution of the first term in \eqref{eq:ParaproductDecomposition}.

\smallskip

We turn to the second term. Expand the second term on the left hand side of \eqref{eq:ParaproductDecomposition} as
\begin{equation*}
\begin{split}
P_N (P_{\gtrsim N} u P_{\gtrsim N} u ) &= \sum_{N_1 \sim N_2 \gtrsim N} P_N (P_{N_1} u P_{N_2} u) \\
&= \sum_{N_1 \sim N_2 \sim N} P_N (P_{N_1} u P_{N_2} u) + \sum_{N_1 \sim N_2 \gg N} P_N (P_{N_1} u P_{N_2} u).
\end{split}
\end{equation*}
The High$\times$High$\rightarrow$High-interaction corresponds to the first term in the above display and the High$\times$High$\rightarrow$Low-interaction corresponds to the second term.

\smallskip

We estimate the High$\times$High$\rightarrow$High-interaction as follows. Apply \eqref{eq:TrilinearEstHighHighHigh} to find
\begin{equation*}
\int (f_{1,N_1,L_1} * f_{2,N_2,L_2}) f_{3,N_3,L_3} d\xi (d \eta)_1 d\tau \lesssim C(N) \prod_{i=1}^3 L_i^{\frac{1}{2}} (1+L_i/N_{i,
+}^{\alpha+1})^{\frac{1}{4}} \| f_{i,N_i,L_i} \|_2.
\end{equation*}
with $C(N) = N_1^{-\frac{4 \alpha}{7}-\frac{1}{14}+\frac{\varepsilon}{2}}$ in the relevant range of $\alpha$. Taking into account time localization (factor $N_1^{2-\alpha+\varepsilon}$) and the derivative loss (factor $N_1$), we find the above to be favorable for $\alpha > \frac{41}{22} \approx 1.863$.

\smallskip

We turn to High$\times$High$\rightarrow$Low-interaction: In this case we need to obtain an estimate for $N_2 \ll N_1 \sim N_3$:
\begin{equation}
\label{eq:HighHighLowEnergyEstSol}
\begin{split}
&\quad N_1^{2-\alpha+\varepsilon} N_2^{2s} N_2 \int (f_{1,N_1,L_1} * f_{2,N_2,L_2}) f_{3,N_3,L_3} d\xi (d \eta)_1 d\tau \\
&\lesssim N_2^{0-} N_1^{2s} \prod_{i=1}^3 L_i^{\frac{1}{2}} (1+L_i / N_{i,+}^{\alpha+1})^{\frac{1}{4}} \| f_{i,N_i,L_i} \|_2.
\end{split}
\end{equation}
In the High$\times$Low$\rightarrow$High-interaction we have already obtained an estimate for $\alpha > 1.89$:
\begin{equation*}
\begin{split}
&\quad N_1^{2-\alpha+\varepsilon} N_2  \int (f_{1,N_1,L_1} * f_{2,N_2,L_2}) f_{3,N_3,L_3} d\xi (d \eta)_1 d\tau \\
&\lesssim N_2^{0-} \prod_{i=1}^3 L_i^{\frac{1}{2}} (1+L_i / N_{i,+}^{\alpha+1})^{\frac{1}{4}} \| f_{i,N_i,L_i} \|_2.
\end{split}
\end{equation*}
This implies \eqref{eq:HighHighLowEnergyEstSol}.
\end{proof}

\subsection{Estimates for differences of solutions}

Next, we prove estimates at negative Sobolev regularity $s' = -8(2-\alpha)$. Let $u_i$, $i=1,2$ solve \eqref{eq:GeneralizedKPII}. Note that $v= u_1-u_2$ solves
\begin{equation}
\label{eq:DifferenceEquation}
\left\{ \begin{array}{cl}
\partial_t v + \partial_x D_x^\alpha v - \partial_x^{-1} \partial_y^2 v &= \partial_x \tilde{P}_M (\tilde{P}_M v \, \tilde{P}_M (u_1+u_2)), \quad (t,x,y) \in \R \times \R \times \T, \\
v(0) &= u_1(0) - u_2(0).
\end{array} \right.
\end{equation}

\begin{proposition}
\label{prop:ShorttimeEnergyEstimateDifferences}
Let $T \in (0,1]$, $s \geq 0$, $\alpha \in (1.972,2)$, $M \in 2^{\N_0} \cup \{ \infty \}$ and $u_i \in F^s(T)$, $i=1,2$ solve \eqref{eq:FourierTruncatedEquation}. Let $s'=-8(2-\alpha)$. Then \eqref{eq:ShorttimeEnergyEstimateDifferences} holds.
\end{proposition}

\begin{proof}
We follow the same arguments like in the previous subsection. By the fundamental theorem of calculus, using that $\| P_{N} v(t) \|^2_{L^2}$ is continuously differentiable in $t$, we find for $N \geq 1$:
\begin{equation*}
\| P_N v(t) \|^2_{L^2} = \| P_N v(0) \|^2_{L^2} + 2 \int_0^t \int_{\R \times \T} P_N (\tilde{P}_M v) \partial_x P_N (\tilde{P}_M v \, \tilde{P}_M(u_1+u_2)) dx ds.
\end{equation*}
We redenote $\tilde{P}_M v \to v$ and $\tilde{P}_M u_i \to u_i$ to ease notation. The nonlinearity is decomposed as a paraproduct:
\begin{equation}
\label{eq:DifferenceParaproductDecomposition}
\begin{split}
P_N (v(u_1+u_2)) &= P_N (v P_{\ll N} (u_1+u_2)) \\
&\quad + P_N ((u_1+u_2) P_{\ll N} v ) + P_N (P_{\gtrsim N} v P_{\gtrsim N}(u_1+u_2)).
\end{split}
\end{equation}
The first and third term can be estimated like previously as the derivative can be assigned to the low frequency. We turn to the estimate of the second term, in which case we cannot integrate by parts. This term amounts to an estimate of
\begin{equation*}
\begin{split}
N_1^{-16(2-\alpha)} N_1 &\big| \int_0^t \int_{\R \times \T} P_{N_1} v P_{N_2} v P_{N_3} u_i dx ds \big| \\
&\lesssim T^\delta C(\underline{N}) \| P_N v \|_{F_N} \| P_{N_2} v \|_{F_{N_2}} \| P_{N_3} u_i \|_{F_{N_3}}
\end{split}
\end{equation*}
with $N_1 \sim N_3 \gg N_2$ for a suitable constant, after which the claim follows from dyadic summation.

Here the smoothing provided by the estimate at negative regularity comes to rescue. We add time localization, which incurs a factor of $T N_1^{2-\alpha+\varepsilon}$, change to Fourier variables and break the modulation size into dyadic regions $L_i \geq N_1^{2-\alpha+\varepsilon}$, $i=1,3$ dictated by time localization. Moreover, we suppose $L_2 \geq N_1^{2-\alpha+\varepsilon} \wedge N_{2,+}^{\alpha+1}$. This reduces to the trilinear convolution estimates from Section \ref{section:TrilinearConvolutionEstimates}:
\begin{equation*}
\begin{split}
&\quad \big| \int (f_{1,N_1,L_1} * f_{2,N_2,L_2}) f_{3,N_3,L_3} d\xi (d\eta)_1 d\tau \\ 
&\lesssim C'(N) \prod_{i=1}^3 L_i^{\frac{1}{2}} (1+L_i / N_i^{\alpha+1})^{\frac{1}{4}} \| f_{i,N_i,L_i} \|_{L^2_{\tau,\xi,\tau}}.
\end{split}
\end{equation*}

We estimate the contribution of $N_2 \lesssim N_1^{\frac{7}{8}}$ by \eqref{eq:TrilinearEstHighLowHighI}:
\begin{equation*}
\big| \int (f_{1,N_1,L_1} * f_{2,N_2,L_2}) f_{3,N_3,L_3} d\xi (d\eta)_1 d\tau \big| \lesssim N_1^{-1-\frac{\varepsilon}{2}} \prod_{i=1}^3 L_i^{\frac{1}{2}} (1+L_i / N_i^{\alpha+1})^{\frac{1}{4}} \| f_i \|_{L^2_{\tau,\xi,\eta}}.
\end{equation*}
Now adding time localization $N_1^{2-\alpha+\varepsilon}$, taking into account the negative regularity $N_1^{-16(2-\alpha)}$, taking into account the derivative loss, we find by the separation assumption $N_2 \lesssim N_1^{\frac{7}{8}}$:
\begin{equation*}
N_1^{2-\alpha + \varepsilon} N_1^{-16(2-\alpha)} \lesssim N_1^{-7(2-\alpha)} N_1^{-8(2-\alpha)} \lesssim N_2^{-8(2-\alpha)} N_1^{-8(2-\alpha)}.
\end{equation*}

Next, we consider the case of almost comparable frequencies: $N_1^{\frac{7}{8}} \lesssim N_2 \lesssim N_1$. Here we use estimate \eqref{eq:TrilinearEstHighLowHighII} with $M^* = N_1^{\frac{9}{4}}N_2$. We check the contribution of $C_{21}$ and $C_{22}$ separately. We obtain the constant $C_{21}$ following Remark \ref{remark:ConstantHighLowHigh}, taking into account derivative loss, time localization, and negative regularity:
\begin{equation*}
\begin{split}
&\quad N_1 N_1^{2-\alpha+\varepsilon} N_1^{-16(2-\alpha)} N_1^{\frac{5}{16}+\frac{1}{6}-\frac{\alpha}{8}+\varepsilon} N_2^{-\frac{5}{24}-\frac{5 \alpha}{8}+\varepsilon} \\
&\lesssim N_1^{-16(2-\alpha)} N_1^{2-\alpha+\varepsilon + 1 + \frac{5}{16}+\frac{1}{6}-\frac{\alpha}{8}-\frac{35}{192}-\frac{35 \alpha}{64}+\frac{7 \varepsilon}{8}}.
\end{split}
\end{equation*}
This is favorable for 
\begin{equation}
\label{eq:NonlinearThresholdIII}
\alpha \geq 1.972.
\end{equation}

\smallskip

%

The second contribution is given by 
\begin{equation*}
C_{22}(N) = N_2^{\frac{1}{2}} N_1^{-\frac{2-\alpha+\varepsilon}{2}} (N_1^{\frac{9}{4}} N_2)^{-\frac{1}{2}} = N_1^{-\frac{2-\alpha+\varepsilon}{2}} N_1^{-\frac{9}{8}}.
\end{equation*}
Subsuming the derivative loss, the negative Sobolev regularity and the time localization, we find
\begin{equation*}
N_1^{-16(2-\alpha)} N_1^{\frac{2-\alpha+\varepsilon}{2}} N_1^{-\frac{1}{8}} \lesssim N_1^{-16(2-\alpha)},
\end{equation*}
the ultimate estimate holding true for $\alpha > 1.75$.

\smallskip

We estimate the third term in \eqref{eq:DifferenceParaproductDecomposition}, which amounts to a trilinear estimate of
\begin{equation*}
\begin{split}
&\quad N_2 N_2^{-16(2-\alpha)} N_1^{2-\alpha+\varepsilon} \int (f_{1,N_1,L_1} * f_{2,N_2,L_2}) f_{3,N_3,L_3} d\xi (d\eta)_1 d\tau \\
&\lesssim C(N) \prod_{i=1}^3 L_i^{\frac{1}{2}} (1+L_i / N_i^{\alpha+1})^{\frac{1}{4}} \| f_{i,N_i,L_i} \|_2
\end{split}
\end{equation*}
for $N_1 \sim N_3 \gg N_2$. (Note that the case $N_1 \sim N_2 \sim N_3$ can be estimated like in the previous subsection, as this corresponds merely to a shift in regularity.)

This can be reduced to a previous estimate: We have seen when estimating the first term in \eqref{eq:ParaproductDecomposition} that for $N_2 \ll N_1$ it holds
\begin{equation*}
\begin{split}
&\quad N_1 N_1^{-16(2-\alpha)} N_1^{2-\alpha+\varepsilon} \int (f_{1,N_1,L_1} * f_{2,N_2,L_2}) f_{3,N_3,L_3} d\xi (d\eta)_1 d\tau \\
&\lesssim N_2^{-8(2-\alpha)} N_1^{-8(2-\alpha)} \prod_{i=1}^3 L_i^{\frac{1}{2}} (1+L_i / N_i^{\alpha+1})^{\frac{1}{4}} \| f_{i,N_i,L_i} \|_2.
\end{split}
\end{equation*}
From the above follows
\begin{equation*}
\begin{split}
&\quad N_2 N_2^{-16(2-\alpha)} N_1^{2-\alpha+\varepsilon} \int (f_{1,N_1,L_1} * f_{2,N_2,L_2}) f_{3,N_3,L_3} d\xi (d\eta)_1 d\tau \\
&\lesssim N_2^{-8(2-\alpha)} N_1^{-8(2-\alpha)} \prod_{i=1}^3 L_i^{\frac{1}{2}} (1+L_i / N_i^{\alpha+1})^{\frac{1}{4}} \| f_{i,N_i,L_i} \|_2
\end{split}
\end{equation*}
provided that $\alpha > 2- \frac{1}{16}$, which leads to no additional constraint.

The proof is complete.
\end{proof}

\section{Conclusion of local well-posedness with short-time estimates}
\label{section:ConclusionLWP}
In \cite{HerrSchippaTzvetkov2024} we have applied the Ionescu--Kenig--Tataru  argument \cite{IonescuKenigTataru2008} to the KP-II equation in the periodic setting. Since this applies with only minor changes here, we shall be brief. Showing the existence of solutions in an appropriate scale of regularity requires further explanation.

\smallskip

The proof of local well-posedness is carried out in three steps:
\begin{itemize}
\item[1.] Establishing the existence of solutions and a priori estimates in $F^s(T)$ for $s \geq 0$,
\item[2.] Estimates for the solution to the difference equation at negative regularity,
\item[3.] Conclusion of continuous dependence via frequency envelopes.
\end{itemize}

We remark that the proof of existence employs a priori estimates for a regularized equation. For this reason we do not separate the construction of solutions and the proof of a priori estimates.

 The frequency-dependent time localization, as sketched in Section \ref{subsection:ResonanceTimeLocalization}, is given by
\begin{equation*}
T=T(N)=N^{-(2-\alpha+\varepsilon)}.
\end{equation*}
The extra time localization $N^{-\varepsilon}$ will be useful to overcome certain logarithmic divergences. 

\subsection{Existence of solutions and global a priori estimates}
\label{subsection:ExistenceSolutions}
We construct solutions to
\begin{equation}
\label{eq:fKPIIAppendix}
\left\{ \begin{array}{cl}
\partial_t u - \partial_x D_x^{\alpha} u + \partial_x^{-1} \partial_y^2 u &= u \partial_x u, \quad (t,x,y) \in \R \times \R \times \T,\\
u(0) &= \phi
\end{array} \right.
\end{equation}
for $\alpha \in (1.94,2)$. This is more admissible than stated in Theorem \ref{thm:GWPFKPII}. Since we are exclusively analyzing solutions, which satisfy more symmetries than differences of solutions, we obtain a larger range.

\smallskip

The construction employs a Galerkin approximation and the Aubin--Lions compactness lemma (cf. \cite{Lions1969}). The argument follows to a great extent the construction of solutions to KP-I equations in three dimensions (cf. \cite{HerrSanwalSchippa2024}). 

Recall the regularity scale $B^s$ introduced in \eqref{eq:BsNorms}.
 Clearly, $B^1 \hookrightarrow H^1$ and on bounded domains $D_R = (-R,R) \times \T$, the embedding $H^1(D_R) \hookrightarrow L^2(D_R)$ is compact. This will allow us to construct global solutions in $C_T L^2(\R \times \T)$ by compactness arguments.

 We shall see below that this symbol behaves well with commutator arguments and the derivative nonlinearity. This is the main reason for not working directly in isotropic Sobolev spaces. We have the following:
\begin{proposition}
\label{prop:ExistenceSolutions}
For any $s \geq 2$ and $T>0$ there is a mapping $S_T^s: B^s \to L_T^{\infty} B^s$ with
\begin{equation*}
S_T^s(\phi) \in C_T^1 H^{-1} \cap C_T H^1 \cap C_T B^0
\end{equation*}
and $S_T^s (\phi)$ is the unique distributional solution to \eqref{eq:fKPIIAppendix}. 

Consequently, there is a data-to-solution mapping $S_T^{\infty}: B^{\infty} \to C_T B^{\infty}$ for any $T > 0$. Moreover, for any $s \geq 0$, we have the global a priori estimates
\begin{equation*}
\sup_{t \in [-T,T]} \| S_T^s(\phi) \|_{H^{s,0}} \lesssim_{T,s} \| \phi \|_{H^{s,0}},
\end{equation*}
and for $s \geq 2$ we have
\begin{equation*}
\sup_{t \in [-T,T]} \| S_T^s(\phi) \|_{B^s} \lesssim_{T,s} \| \phi \|_{B^s}.
\end{equation*}
\end{proposition}
Once the proposition is established, we can work with solutions in $C_T B^{\infty}$ since $B^{\infty} \hookrightarrow L^2(\R \times \T)$ is dense. When the continuity in $H^{s,0}$ is established, we can extend $S_T^s:B^s \to C([-T,T],B^s)$ to $S_T^s:H^{s,0} \to C([-T,T],H^{s,0})$.


\medskip

\emph{Proof of Proposition \ref{prop:ExistenceSolutions}: Galerkin approximation.} We begin with considering the truncated equation \eqref{eq:FourierTruncatedEquation}. Recall that this is given by
\begin{equation*}
\left\{ \begin{array}{cl}
\partial_t u^M - \partial_x D_x^{\alpha} u^M + \partial_x^{-1} \partial_y^2 u^M &= \tilde{P}_M (\partial_x (\tilde{P}_M u)^2), \quad (t,x,y) \in \R \times \R \times \T, \\
u^M(0) &= \phi \in B^s.
\end{array} \right.
\end{equation*}

Let $S_\alpha(t) : B^s \to B^s$ denote the unitary evolution generated by $L_\alpha$. We obtain the solution to \eqref{eq:FourierTruncatedEquation} as fixed points of the integral equation:
\begin{equation*}
u^M(t) = S_\alpha(t) \phi + \int_0^t S_\alpha(t-\tau) \tilde{P}_M ( \partial_x (\tilde{P}_M u^M)^2 ) ds.
\end{equation*}

By the Cauchy-Picard-Lipschitz theorem, we obtain a local solution $u^M \in C_T B^s$ with $T=T(\| \phi \|_{B^s}, M)$ for $s \geq 0$. This is based on $\| S_\alpha(t) \tilde{P}_M \|_{B^s \to B^s} \leq C(t,M,s) < \infty $ and 
\begin{equation*}
\| \tilde{P}_M ( \partial_x (\tilde{P}_M u^M)^2) \|_{B^s} \lesssim_{M,s} \| u^M \|^2_{B^s}.
\end{equation*}
However, the dependence on $M$ is unfavorable for the construction of solutions. We shall obtain a priori estimates independent of $M$ by invoking the short-time analysis from Sections \ref{section:ShorttimeBilinearEstimates} and \ref{section:ShorttimeEnergyEstimates}.

\smallskip

First, we obtain global bounds independent on $M$ in the scale of anisotropic Sobolev spaces. Again, by Cauchy-Picard-Lipschitz we have as well solutions $u^M \in C_T H^{s,0}$ for $T \leq T_M = T(\| \phi \|_{H^{s,0}},M)$. For $T \leq T_M$ we have the set of estimates by virtue of Lemma \ref{lem:LinearEnergyEstimate}, Propositions \ref{prop:ShorttimeBilinearEstimate} and \ref{prop:ShorttimeEnergyEstimate}:
\begin{equation*}
\left\{ \begin{array}{cl}
\| u^M \|_{F^s(T)} &\lesssim \| u^M \|_{E^s(T)} + \| \partial_x \tilde{P}_M ((\tilde{P}_M u^M)^2) \|_{\mathcal{N}^s(T)}, \\
\| \partial_x \tilde{P}_M ((\tilde{P}_M u)^2) \|_{\mathcal{N}^s(T)} &\lesssim T^{\delta} \| u^M \|_{F^0(T)} \| u^M \|_{F^s(T)}, \\
\| u^M \|^2_{E^s(T)} &\lesssim \|u^M(0) \|^2_{H^{s,0}} + T^{\delta} \| u^M \|^2_{F^s(T)} \| u^M \|_{F^0(T)}.
\end{array} \right.
\end{equation*}
Above we moreover used the monotonicity of the norms $\| \tilde{P}_M u^M \|_{Z^s(T)} \leq \| u^M \|_{Z^s(T)}$ for $Z \in \{ F, \mathcal{N} \}$. This gives
\begin{equation}
\label{eq:PersistenceTruncatedI}
\| u^M \|_{F^s(T)}^2 \lesssim \| u^M(0) \|_{H^{s,0}}^2 + T^{\delta} \| u^M \|_{F^s(T)}^2 \| u^M \|_{F^0(T)}^2 + T^{\delta} \| u^M \|_{F^s(T)}^2 \| u^M \|_{F^0(T)}.
\end{equation}
Recall the limiting properties of the short-time spaces as $T \downarrow 0$ (cf. \cite[p.~279]{IonescuKenigTataru2008}):
\begin{equation*}
\lim_{T \downarrow 0} \| u^M \|_{F^s(T)} \lesssim \| u^M(0) \|_{H^{s,0}}. 
\end{equation*}

For $s=0$ this  gives for $T=T(\| \phi \|_{L^2}) > 0$ small enough by a continuity argument
\begin{equation*}
\| u^M \|_{F^0(T)} \lesssim \| u^M(0) \|_{L^2}.
\end{equation*}
By $L^2$-conservation of the truncated evolution \eqref{eq:FourierTruncatedEquation} and the persistence property established in \eqref{eq:PersistenceTruncatedI} this gives global bounds  for $s \geq 0$:
\begin{equation*}
\sup_{t \in [-T,T]} \| u^M(t) \|_{H^{s,0}} \lesssim_{T,s} \| u^M(0) \|_{H^{s,0}}.
\end{equation*}

In the following we denote the short-time spaces with additional Sobolev weight $p(\xi,\eta) = 1+\frac{|\eta|}{|\xi|}$ by $\bar{F}^s(T)$, $\bar{\mathcal{N}}^s(T)$, and $\bar{E}^s(T)$. We turn to establishing global a priori estimates with extra regularity in the $y$-variable.

\begin{proposition}
\label{prop:WeightedShorttimeEstimates}
For $s \geq 1$ the following estimates hold:
\begin{equation}
\label{eq:WeightedSobolevEstimates}
\left\{ \begin{array}{cl}
\| u^M \|_{\bar{F}^s(T)} &\lesssim \| u^M \|_{\bar{E}^s(T)} + \| \tilde{P}_M (\partial_x ( \tilde{P}_M u^M)^2) \|_{\bar{\mathcal{N}}^s(T)}, \\
\| \partial_x \tilde{P}_M ((\tilde{P}_M u^M)^2) \|_{\bar{\mathcal{N}}^s(T)} &\lesssim T^{\delta} \| u^M \|_{\bar{F}^s(T)} \| u^M \|_{F^s(T)}, \\
\| u^M \|^2_{\bar{E}^s(T)} &\lesssim \| \phi \|^2_{B^s} + T^{\delta} \| u^M \|^2_{\bar{F}^s(T)} \| u^M \|_{F^s(T)}.
\end{array} \right.
\end{equation}
\end{proposition}
The short-time analysis actually yields improved estimates, but these are presently not required. 

\begin{proof}[Proof~of~Proposition~\ref{prop:WeightedShorttimeEstimates}]

\emph{1. Linear energy estimate with Sobolev weight:}
The first estimate in \eqref{eq:WeightedSobolevEstimates} is obtained from applying Lemma \ref{lem:LinearEnergyEstimate} to $p(\nabla_x / i) u^M$. 

\smallskip

\emph{2. Short-time bilinear estimate with Sobolev weight:}
For the second estimate in \eqref{eq:WeightedSobolevEstimates} dominate
\begin{equation}
\label{eq:ShorttimeBilinearEstimateWeightedI}
\begin{split}
&\quad \| \partial_x \tilde{P}_M ((\tilde{P}_M u^M)^2) \|_{\bar{\mathcal{N}}^s(T)} \\
&\lesssim \| \tilde{P}_M (\partial_x( \tilde{P}_M u^M)^2) \|_{\mathcal{N}^s(T)} + \| \langle D_x \rangle (\tilde{P}_M u^M) ( \partial_y \tilde{P}_M u^M) \|_{\mathcal{N}^{s-1}(T)}.
\end{split}
\end{equation}
As a consequence of Proposition \ref{prop:ShorttimeBilinearEstimate} we can estimate the first term by
\begin{equation}
\label{eq:ShorttimeBilinearEstimateWeightedII}
\| \tilde{P}_M (\partial_x( \tilde{P}_M u^M)^2) \|_{\mathcal{N}^s(T)} \lesssim T^{\delta} \| \tilde{P}_M u^M \|_{F^s(T)}^2.
\end{equation}
For $s \geq 1$, we obtain for the second term again by Proposition \ref{prop:ShorttimeBilinearEstimate}
\begin{equation}
\label{eq:ShorttimeBilinearEstimateWeightedIII}
\begin{split}
\| \langle D_x \rangle (\tilde{P}_M u^M) ( \partial_y \tilde{P}_M u^M) \|_{\mathcal{N}^{s-1}(T)} &\lesssim \| \tilde{P}_M u^M \|_{F^{s-1}(T)} \| \partial_y \tilde{P}_M u^M \|_{F^{s-1}(T)} \\ &\lesssim \| u^M \|_{F^s(T)} \| u^M \|_{\bar{F}^s(T)}.
\end{split}
\end{equation}
Combining \eqref{eq:ShorttimeBilinearEstimateWeightedI}-\eqref{eq:ShorttimeBilinearEstimateWeightedIII} gives
\begin{equation*}
\| \partial_x \tilde{P}_M ((\tilde{P}_M u^M)^2) \|_{\bar{\mathcal{N}}^s(T)} \lesssim T^{\delta} \| u^M \|_{\bar{F}^s(T)} \| u^M \|_{F^s(T)}.
\end{equation*}

\emph{3. Short-time energy estimate with Sobolev weight:}
We turn to the third estimate in \eqref{eq:WeightedSobolevEstimates}. This will follow from revisiting some steps in the proof of Proposition \ref{prop:ShorttimeEnergyEstimate}. Let $w = \partial_x^{-1} \partial_y u^M$. We have
\begin{equation}
\label{eq:ShorttimeWeightedEnergyEstimate}
N^{2s} \| p(\nabla_x/i) P_N u^M(t) \|^2_{L^2} \lesssim N^{2s} \| P_N u^M(t) \|^2_{L^2} + N^{2s} \| P_N w(t) \|_{L^2}^2.
\end{equation}
The estimate for the first term is immediate from Proposition \ref{prop:ShorttimeEnergyEstimate}:
\begin{equation*}
\sum_{N \geq 1} N^{2s} \sup_{t \in [0,T]} \| P_N u^M(t) \|^2_{L^2} \lesssim \| u^M \|_{H^{s,0}}^2 + T^{\delta} \| u^M \|_{F^s(T)}^2 \| u^M \|_{F^0(T)}.
\end{equation*}

We turn to the estimate of the second term in \eqref{eq:ShorttimeWeightedEnergyEstimate}.
To this end, we note that $w$ satisfies the following evolution equation:
\begin{equation*}
\partial_t w - \partial_x D_x^{\alpha} w + \partial_x^{-1} \partial_y^2 w = u^M \partial_x w,
\end{equation*}
and we have by the fundamental theorem of calculus:
\begin{equation*}
\| P_N w(t) \|^2_{L^2} \lesssim \| P_N w(0) \|^2_{L^2} + \int_0^t \int_{\R \times \T} P_N w(s,x,y) P_N (u^M \partial_x  w) dx dy ds.
\end{equation*}
The application of the fundamental theorem is justified noting that
\begin{equation*}
    \| u^M \partial_x w \|_{L^2_{x,y}} \lesssim \| u^M \|_{L^\infty_{x,y}} \| \partial_x w \|_{L^2_{x,y}} \lesssim \| u^M \|^2_{B^2}.
\end{equation*}

We use a paraproduct decomposition
\begin{equation}
\label{eq:ParaproductDecompositionWeightedEnergyEstimate}
P_N (u^M \partial_x w) = P_N (P_{\ll N} u^M \partial_x w) + P_N (u^M P_{\ll N} \partial_x w) + P_N (P_{\gg N} u P_{\gg N} \partial_x w).
\end{equation}
By the usual commutator argument frequently used in the proof of Proposition \ref{prop:ShorttimeEnergyEstimate} we have the estimate for the first term
\begin{equation}
\label{eq:ShorttimeWeightedEnergyEstimateI}
\sum_{N \geq 1} N^{2s} \int_0^t \int_{\R \times \T} P_N (u^M P_{\ll N} \partial_x w) P_N w \lesssim T^{\delta} \| w \|^2_{F^s(T)} \| u^M \|_{F^0(T)}.
\end{equation}
The estimate for the second term in \eqref{eq:ParaproductDecompositionWeightedEnergyEstimate} is simpler, and we find
\begin{equation}
\label{eq:ShorttimeWeightedEnergyEstimateII}
\sum_{N \geq 1} N^{2s} \int_0^t \int_{\R \times \T} P_N (u^M P_{\ll N} \partial_x w) P_N w \lesssim T^{\delta} \| w \|^2_{F^s(T)} \| u^M \|_{F^0(T)}.
\end{equation}
In the third term we integrate by parts to find
\begin{equation}
\label{eq:ShorttimeWeightedEnergyEstimateIII}
\sum_{N \geq 1} N^{2s} \int_0^t \int_{\R \times \T} P_N w P_N (P_{\gg N} u^M \partial_x P_{\gg N} w) dx dy ds \lesssim T^{\delta} \| u^M \|_{F^1(T)} \| w \|^2_{F^s(T)}.
\end{equation}
Plugging \eqref{eq:ShorttimeWeightedEnergyEstimateI}-\eqref{eq:ShorttimeWeightedEnergyEstimateIII} into \eqref{eq:ShorttimeWeightedEnergyEstimate} we conclude the proof.
\end{proof}

\begin{proof}[Proof~of~Proposition~\ref{prop:ExistenceSolutions}, ctd: Applying the Aubin-Lions lemma.]
The short-time \\ estimates in weighted norms \eqref{eq:WeightedSobolevEstimates} together with the already established global estimates for $\| u^M \|_{F^s(T)}$ we find global a priori estimates in $B^s$ for $s \geq 1$:
\begin{equation*}
\sup_{t \in [-T,T]} \| u^M(t) \|_{B^s} \lesssim_{T,s} \| \phi \|_{B^s}.
\end{equation*}
To summarize, $(u^M)_{M \in 2^{\N}}$ exist globally in $L_T^{\infty} B^1$ and satisfies bounds (for fixed $T$) independent of $M$. For any domain $D_R = (-R,R) \times \T$ we have bounds uniform in $M$ and $R$:
\begin{equation*}
u^M \in L_T^{\infty} H^1(D_R), \quad \partial_t u^M \in L_T^{\infty} H^{-1} (D_R).
\end{equation*}
By the compact embedding $H^1(D_R) \hookrightarrow L^2(D_R)$ and the continuous embedding $L^2(D_R) \hookrightarrow H^{-1}(D_R)$, we can apply the Aubin--Lions lemma to obtain a subsequence $u^M \to u \in C_T L^2_{\text{loc}}$. We have $u \in L_T^{\infty} B^s$, consequently $u \in L_T^{\infty} H^1$ and $\partial_x^{-1} \partial_y u \in L_T^{\infty} L^2$.

By dual pairing in $L^2$ for $\varphi \in C^\infty_c([-T,T],\R \times \T)$ and passing to the limit $M \to \infty$, we find that $u$ is a distributional solution to fKP-II. Since $u \in C^1([-T,T],H^{-1}) \cap L_T^{\infty} H^1$, we conclude $u \in C([-T,T],H^{\sigma})$ for $\sigma \in [-1,1)$. For $s>1$ we have moreover $u \in C_T H^1$.

Lastly, we show that $u \in C_T B^0$ for $s>1$. We have already established $u \in C_T L^2$. Now it suffices to show that $w = \partial_x^{-1} \partial_y u \in C_T L^2$. To this end, we use Duhamel's formula
\begin{equation*}
w(t_2) - w(t_1) = ( S_\alpha(t_2) - S_\alpha(t_1)) \partial_x^{-1} \partial_y \phi - \int_{t_1}^{t_2} S_\alpha(t_2-s) (u (s) \partial_x w(s)) ds.
\end{equation*}
The linear term converges to zero in $L^2$ as $t_1 \to t_2$ because $(S_\alpha(t))_{t \in \R}$ is a $C_0$-semigroup. For the second term we estimate with $I = [t_1,t_2]$ for $s>1$:
\begin{equation*}
\| \int_{t_1}^{t_2} S_\alpha(t_2-s) (u (s) \partial_x w(s)) ds \|_{L^2} \leq \| u \|_{L_I^1 L^{\infty}_{x,y}} \| \partial_x w \|_{L^{\infty}_I L^2_{x,y}} \lesssim \| u \|_{L^1_I B^s} \| w \|_{L^\infty_I B^1}.
\end{equation*}
The ultimate estimate follows from Sobolev embedding. Since $u \in L_T^{\infty} B^s$ the preceding estimate implies $u \in C_T B^0$.
\end{proof}

\smallskip

We remark that as a consequence of the proof we obtain the a priori estimates for solutions $u = S_T^s(u_0)$:
\begin{equation}
\label{eq:APrioriEstimateSolutionsFs}
\| u \|_{F^s(T)} \lesssim \| u_0 \|_{H^{s,0}}.
\end{equation}

Indeed, for $s \geq 0$ invoking Lemma \ref{lem:LinearEnergyEstimate}, and Propositions \ref{prop:ShorttimeBilinearEstimate} and \ref{prop:ShorttimeEnergyEstimate} we find:
\begin{equation*}
\left\{ \begin{array}{cl}
\| u \|_{F^s(T)} &\lesssim \| u \|_{E^s(T)} + \| \partial_x (u^2) \|_{\mathcal{N}^s(T)}, \\
\| \partial_x (u^2) \|_{\mathcal{N}^s(T)} &\lesssim T^\delta \| u \|_{F^0(T)} \| u \|_{F^s(T)}, \\
\| u \|^2_{E^s(T)} &\lesssim \| u_0 \|_{H^{s,0}}^2 + T^\delta \| u \|^2_{F^s(T)} \| u \|_{F^0(T)}.
\end{array} \right.
\end{equation*}

\smallskip

By $L^2$-conservation this yields global existence in $H^{s,0}$. Indeed, for $s=0$, the estimates taken together give
\begin{equation*}
\| u \|^2_{F^0(T)} \lesssim \| u_0 \|^2_{L^2} + T^\delta (\| u \|^4_{F^0(T)} + \| u \|^3_{F^0(T)}).
\end{equation*}
Choosing $T=T(\| u_0 \|_{L^2})$ we find
\begin{equation}
\label{eq:APrioriEstimateSolution}
\| u \|_{F^0(T)} \lesssim \| u_0 \|_{L^2}.
\end{equation}

For $s \geq 0$ the estimates taken together yield the persistence of regularity property:
\begin{equation*}
\| u \|_{F^s(T)}^2 \lesssim \| u_0 \|_{H^{s,0}}^2 + T^\delta (\| u \|^2_{F^0(T)} \| u \|_{F^s(T)}^2 + \| u \|_{F^s(T)}^2 \| u \|_{F^0(T)} ),
\end{equation*}
which gives the desired a priori estimate \eqref{eq:APrioriEstimateSolutionsFs}.

\subsection{Lipschitz continuous dependence at negative Sobolev regularity}

In the next step we show Lipschitz continuous dependence in negative Sobolev regularities for solutions $u_1,u_2$ in $L^2$. Let $v=u_1-u_2$ denote the differences of solutions to \eqref{eq:fKPIIAppendix}. $v$ solves the following Cauchy problem:
\begin{equation*}
\left\{ \begin{array}{cl}
\partial_t v - \partial_x D_x^\alpha v + \partial_x^{-1} \partial_y^2 v &= \partial_x(v(u_1+u_2)), \\
v(0) &= u_1(0) - u_2(0) \in L^2(\R \times \T).
\end{array} \right.
\end{equation*}
It holds for $s'= -8(2-\alpha)$ with $\alpha \in (1.985,2)$ as a consequence of Lemma \ref{lem:LinearEnergyEstimate}, and Propositions \ref{prop:ShorttimeBilinearEstimate}, and \ref{prop:ShorttimeEnergyEstimateDifferences}:
\begin{equation*}
\left\{ \begin{array}{cl}
\| v \|_{F^{s'}(T)} &\lesssim \| v \|_{E^{s'}(T)} + \| \partial_x (v(u_1+u_2)) \|_{\mathcal{N}^{s'}(T)}, \\
\| \partial_x(v(u_1+u_2)) \|_{\mathcal{N}^{s'}(T)} &\lesssim T^\delta \| v \|_{F^{s'}(T)} ( \| u_1 \|_{F^0(T)} + \| u_2 \|_{F^0(T)}), \\
\| v \|^2_{E^{s'}(T)} &\lesssim \| v_0 \|_{H^{s',0}}^2 + T^\delta \| v \|^2_{F^{s'}(T)} ( \| u_1 \|_{F^0(T)} + \| u_2 \|_{F^0(T)}).
\end{array} \right.
\end{equation*}
The reduced symmetries of the difference equation lead to slightly more involved estimates. Subsuming the estimates we find
\begin{equation*}
\| v \|^2_{F^{s'}(T)} \lesssim \| v_0 \|^2_{H^{s',0}} + T^\delta \| v \|^2_{F^{s'}(T)} ( \| u_1 \|_{F^0(T)}^2 + \| u_2 \|_{F^0(T)}^2 + \| u_1 \|_{F^0(T)} + \| u_2 \|_{F^0(T)} ).
\end{equation*}
By the previous step we find for $T \leq T_0(\| u_i(0) \|_{L^2})$:
\begin{equation*}
\| v \|^2_{F^{s'}(T)} \lesssim \| v_0 \|_{H^{s',0}}^2 + T^\delta \| v \|^2_{F^{s'}(T)} ( \| u_1(0) \|_{L^2}^2 + \| u_2(0) \|_{L^2}^2 + \| u_1(0) \|_{L^2} + \| u_2(0) \|_{L^2}).
\end{equation*}
Choosing $T$ smaller, if necessary, we find
\begin{equation*}
\| v \|_{F^{s'}(T)} \lesssim \| v_0 \|_{H^{s',0}(\R \times \T)}.
\end{equation*}

\subsection{Conclusion of continuous dependence via frequency envelopes}

At this point the continuous dependence follows from invoking the frequency envelope argument. For details we refer to our previous implementation in \cite{HerrSchippaTzvetkov2024} to avoid repetition. The proof of Theorem \ref{thm:GWPFKPII} is complete.

\hfill $\Box$

\section{A long-time property}\label{section:long}

In this section we prove Theorem \ref{long_time}.
This elaborates on the proof of a similar result in \cite{mizu} in the case $\alpha=2$.
Indeed, Theorem \ref{long_time} with $\alpha=2$ is the contained in \cite[Lemma~5.3]{mizu}. In order to deal with the fractional dispersion in the $\xi$-frequencies, we will rely on some new arguments, in particular on \cite[Lemmas~6,~7]{KMR}. The proof of \cite[Claim~5.1]{mizu} no longer works in the case of low fractional dispersion. We will use instead an adaptation of the argument used in \cite[Lemma~2.1]{bona} to the case of periodic solutions with respect to the transverse variable. 

\smallskip

For $r>1$, which presently will be chosen close to $1$, we set
$$
\varphi(x)=\int_{-\infty}^x (1+y^2)^{-\frac{r}{2}}\, dy.
$$
The main point is the following lemma.
\begin{lemma}\label{WW}
Let $\alpha \in (4/3,2]$. There is a constant $C$ independent of $x_0\in \R$ and $0 < \varepsilon \leq 1$ such that 
\small
\begin{multline*}
 \int_{\R\times\T}\, (\varphi'(\varepsilon( x+x_0)))^2\, u^4(x,y)dxdy
 \\
\leq  C (\|
  \langle D_x\rangle^{\alpha/2}((\varphi'(\varepsilon( x+x_0-ct)))^{\frac{1}{2}}u(x,y ))
  \|^4_{L^2_{x,y}}
  +
  \|
  (\varphi'(\varepsilon( x+x_0)))^{\frac{1}{2}}\partial_x^{-1}\partial_y u(x,y)
  \|^4_{L^2_{x,y}}).
\end{multline*}
\normalsize
\end{lemma}
Before we turn to the proof in earnest, we recall the following simple commutator estimate  (see e.g.\ \cite[Lemma 2.97]{Bahouri2011} for a proof):
\begin{lemma}
\label{lem:CommEstimate}
Let $P_N$ denote the Littlewood-Paley projection in the $x$-variable for $N \in 2^{\Z}$. For $f \in C^{0,1}(\R)$, $g \in L^2(\R)$ we have the following commutator estimate:
\begin{equation}
\label{eq:commEstimate}
\| [P_N, f] g \|_{L^2_x} \lesssim_{\| f \|_{\dot{C}^{0,1}}} N^{-1} \| g \|_{L^2_x}.
\end{equation}
\end{lemma}

\begin{proof}[Proof~of~Lemma~\ref{WW}]
Set $\psi_{\varepsilon}(x)= (1+(\varepsilon( x+x_0))^2)^{-\frac{r}{4}}$ so that $\varphi'(\varepsilon(x+x_0))=(\psi_{\varepsilon}(x))^2$. 
We need to evaluate $\| \psi_\varepsilon u\|_{L^4_{x,y}}$.  Using the Littlewood-Paley decomposition and the Sobolev inequality, we can write 
\begin{equation}\label{LP}
\| \psi_\varepsilon u\|^2_{L^4_{x,y}}
\lesssim
\sum_{N-{\rm dyadic}}
N^{\frac {1} {2}} \,\| \| P_N(\psi_\varepsilon u)\|_{L^2_{x}}\|^2_{L^4_y},
\end{equation}
where $\Delta_N$ denotes a Littlewood-Paley projection to $\xi$-frequencies of size $N$. 
Set 
$$
\varphi_0(x)=\varphi(\varepsilon(x+x_0)).
$$  
For $0 \leq \alpha\leq 2$,
\begin{equation}\label{parvo_p}
\| P_N(\psi_{\varepsilon} u)\|_{L^2_{x}}\lesssim N^{-\alpha/2} \| \langle D_x\rangle^{\alpha/2}(\psi_\varepsilon u)\|_{L^{2}_x}\,.
\end{equation}
We identify $\T$ with $[-\pi,\pi]$ and using a partition of unity, we can write $u = \theta(y) u + (1-\theta) u$ with $\theta$ vanishing near zero. In the following we focus on estimating $\theta(y) u$, which is redenoted by $u$ for convenience.

Next, we write for $\sigma>0$ to be fixed 
\begin{equation*}
\begin{split}
\| P_N(\psi_{\varepsilon} \theta u)\|^2_{L^2_{x}}&\lesssim
N^{2\sigma} \int_{\R_x}  (\langle D_x\rangle^{-\sigma}(\psi_{\varepsilon} \theta u))^2
\\
&\lesssim 2 N^{2\sigma} \int_{\R_x}  
\int_0^y  \langle D_x\rangle^{-\sigma}(\psi_{\varepsilon} \theta u)\langle D_x\rangle^{-\sigma}(\partial_y (\psi_{\varepsilon} \theta u)).
\end{split}
\end{equation*}
Now we write $\partial_y (\psi_{\varepsilon} \theta u)=\psi_{\varepsilon}\theta' u+\psi_{\varepsilon}\theta\partial_y u$ and therefore 
$$
\| P_N(\psi_{\varepsilon} \theta u)\|^2_{L^2_{x}}\lesssim
N^{2\sigma} (I+II),
$$
where 
$$
I\lesssim \| \psi_{\varepsilon} u\|^2_{L^2_{x,y}}
\lesssim 
 \|
  \langle D_x\rangle^{\alpha/2}(\psi_\varepsilon u)
  \|^2_{L^2_{x,y}}.
$$

We next estimate $II$. Using an integration by parts in $x$, we can write
$$
II\lesssim \int_{\R\times\T}
|\partial_x \big(\langle D_x\rangle^{-2\sigma}(\psi_{\varepsilon} u)\psi_{\varepsilon}\big)|\,
|\partial_x^{-1}\partial_y u|
$$
We now fix $\sigma$ via the relation $1-2\sigma=\alpha/2$, i.e., 
$
\sigma=\frac{1}{2}-\frac{\alpha}{4}.
$
With this choice of $\sigma$, we have 
$$
II\lesssim
\|
\langle D_x\rangle^{\alpha/2}(\psi_{\varepsilon}  u)
\|_{L^2_{x,y}}
\|
\psi_{\varepsilon} \partial_x^{-1}\partial_y u
\|_{L^2_{x,y}},
$$
where we used that $\psi'\lesssim \psi$. 



Consequently, we obtain the estimate 
\begin{equation}\label{vtoro_p}
\| P_N(\psi_{\varepsilon} u)\|^2_{L^2_{x}}\lesssim N^{2\sigma}
\big(
\|
\langle D_x\rangle^{\alpha/2}(\psi_\varepsilon  u)
\|^2_{L^2_{x,y}}
+
\|
\psi_{\varepsilon} \partial_x^{-1}\partial_y u
\|^2_{L^2_{x,y}}
\big).
\end{equation}
It now remains to suitably interpolate between \eqref{parvo_p} and \eqref{vtoro_p}. Set 
$$
\Lambda:=\|
\langle D_x\rangle^{\alpha/2}(\psi_\varepsilon  u)
\|^2_{L^2_{x,y}}
+
\|
\psi_{\varepsilon} \partial_x^{-1}\partial_y u
\|^2_{L^2_{x,y}}\,.
$$
Then using  \eqref{parvo_p} and \eqref{vtoro_p}, we can write 
$$
N\int_{\T}\| P_N(\psi_{\varepsilon} u)\|^4_{L^2_{x}}\lesssim N^{1+2\sigma}\Lambda \int_{\T}\, N^{-\alpha} \| \langle D_x\rangle^{\alpha/2}(\psi_\varepsilon u)\|^2_{L^{2}_x}
\lesssim N^{2-\frac{3\alpha}{2}}\, \Lambda^2.
$$
Therefore, if $\alpha>4/3$, we can sum-up the right hand-side of \eqref{LP} and obtain that $\| \psi_{\varepsilon} u\|^2_{L^4_{x,y}}\lesssim \Lambda^2$.
This completes the proof of Lemma \ref{WW}.
\end{proof}
With Lemma \ref{WW} at hand, we can complete the proof of Theorem \ref{long_time}.

\begin{proof}[Proof~of~Theorem~\ref{long_time}]
If $u$ is a solution to \eqref{eq:GeneralizedKPII}, then 
\begin{multline}\label{mpmpmp}
\frac 1 2 \frac{d}{dt} \int_{\R\times\T}\, \varphi(\varepsilon(x+x_0-ct)) u^2(t,x,y)dxdy
=-\int_{\R\times\T} \partial_x(\varphi(\varepsilon(x+x_0-ct)) u)\, D_x^\alpha u 
\\
-\frac{\varepsilon}{6}\int_{\R\times\T}\varphi'(\varepsilon(x+x_0-ct)) u^3
-\frac{\varepsilon}{2} \int_{\R\times\T} \,\varphi'(\varepsilon(x+x_0-ct))\, \big((\partial_x^{-1}\partial_y u)^2+cu^2\big)\,.
\end{multline}
We now evaluate each term on the right hand-side of \eqref{mpmpmp}. 
Thanks to \cite[Lemmas~6,~7]{KMR}, we have that there exists $\delta>0$ and $C>0$ such that 
\begin{multline}
\label{eq:KMREstimate}
-\int_{\R\times\T} \partial_x(\varphi(\varepsilon( x+x_0-ct)) u)\, D_x^\alpha u 
\leq -\delta \|
  \langle D_x\rangle^{\alpha/2}( (\varphi'(\varepsilon( x+x_0-ct)))^{\frac{1}{2}} u(x,y ))
  \|^2_{L^2_{x,y}}
  \\
  +C \varepsilon^\alpha \|  (\varphi'(\varepsilon( x+x_0-ct)))^{\frac{1}{2}}\, u(x,y)  \|^2_{L^2_{x,y}}\,.
\end{multline}
Note that \cite[Lemmas~6,~7]{KMR} are formulated for the case $\varepsilon = 1$ and the general case is described in \cite[Equations~(53),~(54)]{KMR}.

Using   the Cauchy-Schwarz inequality, we may write 
$$
\big|\int_{\R\times\T}\varphi'(\varepsilon(x+x_0-ct)) u^3\big|^2\lesssim \|u\|^2_{L^2_{x,y}}\  \int_{\R\times\T}\, (\varphi'(\varepsilon(x+x_0-ct)))^2\, u^4(x,y)dxdy.
$$
We now apply the $L^2$ conservation law and  Lemma \ref{WW} to get
\begin{multline*}
\big|\int_{\R\times\T}\varphi'(\varepsilon(x+x_0-ct)) u^3\big|^2\lesssim \mu^2
\big(
\|\langle D_x\rangle^{\alpha/2}((\varphi'(\varepsilon( x+x_0-ct)))^{\frac{1}{2}}u(x,y ))
  \|^4_{L^2_{x,y}}
  \\
  +
   \|
  (\varphi'(\varepsilon(x+x_0-ct)))^{\frac{1}{2}}\partial_x^{-1}\partial_y u(x,y)
  \|^4_{L^2_{x,y}}
  \big).
\end{multline*}
Taking $\mu=\mu(\delta)$ (see \eqref{eq:KMREstimate}) small enough and coming back to \eqref{mpmpmp} we get for $\varepsilon=\varepsilon(c)$ small enough (since $\alpha$ is close to $2$)
$$
 \frac{d}{dt} \int_{\R\times\T}\, \varphi(\varepsilon(x+x_0-ct)) u^2(t,x,y)dxdy\leq 0 \quad \forall\, x_0\in\R.
$$
The point of introducing the scaling parameter and carefully tracking the dependence of the estimates on $\varepsilon$ is to be able to choose $c$ independent of $C$.

Therefore 
 $$
 \int_{\R\times\T}\, \varphi(\varepsilon(x+x_0-ct)) u^2(t,x,y)dxdy\leq  \int_{\R\times\T}\, \varphi(\varepsilon(x+x_0)) u^2(0,x,y) dxdy, \quad \forall\, x_0\in\R.
$$
Applying the last inequality with $x_0=-ct$ and letting $t \to \infty$ yields by the theorem of dominated convergence
$$
\lim_{t\rightarrow+\infty}  \int_{\R\times\T}\, \varphi(\varepsilon(x-2ct)) u^2(t,x,y)dxdy=0.
$$
Therefore we obtain the claimed result for $\gamma_0>2c$. This completes the proof of  Theorem  \ref{long_time}.

\end{proof}

\section*{Acknowledgements}
\begin{enumerate}
  \item
Funded by the Deutsche Forschungsgemeinschaft (DFG, German Research Foundation) -- Project-ID 317210226 -- SFB 1283.
\item R.S.\ gratefully acknowledges financial support from the Humboldt foundation (Feodor-Lynen fellowship) and partial support by the NSF grant DMS-2054975.
\end{enumerate}

\bibliographystyle{plain}

\end{document}